\documentclass[11pt,leqno]{amsart}

\usepackage{amsmath}
\usepackage{amssymb}
\usepackage{a4wide}
\usepackage{bbm}
\usepackage{graphics}
\usepackage{epsfig}
\usepackage{color}
\newcommand{\Subsection}[1]{\subsection{ #1} ${}^{}$}

\newtheorem{theorem}{Theorem}[section]
\newtheorem{lemma}[theorem]{Lemma}
\newtheorem{proposition}[theorem]{Proposition}
\newtheorem{definition}[theorem]{Definition}
\newtheorem{remark}[theorem]{Remark}

\newcounter{hypo}

\newenvironment{hyp}{  \begin{enumerate} \setcounter{enumi}{\value{hypo}} \item}{\stepcounter{hypo} \end{enumerate}}

\def\C{{\mathbb C}}

\def\N{{\mathbb N}} 
\def\R{{\mathbb R}} 

\def\Z{{\mathbb Z}}

\def\S{{\mathbb S}}

\def\CE{\mathcal {E}}

\def\CM{\mathcal {M}}

\def\CO{\mathcal {O}}

\def\CS{\mathcal {S}}
\def\CT{\mathcal {T}}

\def\ker{\mathop{\rm Ker}\nolimits}

\def\one{{\mathchoice {\rm 1\mskip-4mu l} {\rm 1\mskip-4mu l} {\rm 1\mskip-4.5mu l} {\rm 1\mskip-5mu l}}}
\def\re{\mathop{\rm Re}\nolimits}
 \def\im{\mathop{\rm Im}\nolimits}

\def\Op{\mathop{\rm Op}\nolimits}

\def\supp{\mathop {\rm supp}\nolimits} 
\def\arg{\mathop{\rm arg}\nolimits}

\def\bra{\langle} \def\ket{\rangle}
\def\diag{\mathop{\rm diag}\nolimits}
\def\sgn{\mathop{\rm sgn}\nolimits}

\def\<{\langle}
\def\>{\rangle}

\def\rank{\mathop{\rm Rank}\nolimits}

\def\res{\mathop{\rm Res}\nolimits}
\def\residue{\mathop{\rm Residue}\nolimits}

\def\MS{\mathop{\rm MS}\nolimits}

\def\Res{{\rm Res}}

\makeatletter
 \@addtoreset{equation}{section}
 \makeatother
 
\title[Residues for barrier-top resonances]{Spectral projection, Residue of the scattering amplitude, and Schr\"odinger group expansion for barrier-top  resonances.}

\author[J.-F. Bony]{Jean-Fran\c{c}ois Bony}
\address{Jean-Fran\c{c}ois Bony, IMB (UMR CNRS 5251), Universit\'e Bordeaux 1, 33405 Talence, France}
\email{Jean-Francois.Bony@math.u-bordeaux1.fr}
\author[S. Fujii\'e]{Setsuro Fujii\'e}
\address{Setsuro Fujii\'e, Graduate School of Material Science, University of Hyogo, Japan}
\email{fujiie@sci.u-hyogo.ac.jp}
\author[T. Ramond]{Thierry Ramond}
\address{Thierry Ramond, LMO (UMR CNRS 8628), Universit\'e Paris Sud 11, 91405 Orsay, France}
\email{thierry.ramond@math.u-psud.fr}
\author[M. Zerzeri]{Maher Zerzeri}
\address{Maher Zerzeri, LAGA (UMR CNRS 7539), Universit\'e Paris 13, 93430 Villetaneuse, France}
\email{zerzeri@math.univ-paris13.fr}

\begin{document}

\begin{abstract}
We study the spectral projection associated to a barrier-top resonance for the semiclassical Schr\"odinger operator. First, we prove a resolvent estimate for complex energies close to such a resonance. Using that estimate and an explicit representation of the resonant states, we show that the spectral projection has a semiclassical expansion in integer powers of $h$, and compute its leading term. We use this result to compute the residue of the scattering amplitude at such a resonance. Eventually, we give an expansion for large times of the Schr\"odinger group in terms of these resonances.
\end{abstract}

\maketitle

\setcounter{tocdepth}{1}
\tableofcontents

\section{Introduction}

In this paper, we study  the behavior of different physical quantities at the resonances generated by the maximum of the potential of a semiclassical Schr\"odinger operator.
In particular, we show quantitatively to what extent the presence of these resonances drives the behavior of the scattering amplitude and of the Schr\"odinger group.

The resonances generated by the maximum point of the potential (usually called barrier-top resonances) have been studied by Briet, Combes and Duclos \cite{BrCoDu87_01,BrCoDu87_02} and Sj\"ostrand \cite{Sj87_01}. These authors have given a precise description of the resonances in any disc of size $h$ centered at the maximum of the potential. In particular, they have shown that the resonances lie at distance of order $h$ from the real axis, which is in very strong contrast to the case of shape resonances (the well in the island case), with exponentially small imaginary part (see Helffer and Sj\"{o}strand \cite{HeSj86_01}). The description of resonances in larger discs of size $h^{\delta}$, $\delta \in ]0,1]$ has been obtained by Kaidi and Kerdelhu\'e \cite{KaKe00_01} under a diophantine condition. For small discs of size one, this question has been treated in the one dimensional case by the third author \cite{Ra96_01} with the complex WKB method. In the two dimensional case, the resonances in discs of size one have also been considered by Hitrik, Sj{\"o}strand and V{\~u} Ng{\d{o}}c \cite{HiSjVu07_01} (see also the references in this paper). Here, we consider only the resonances at distance $h$ of the maximum of the potential and we recall their precise localization in Section~\ref{k5}.

Resonances can be defined as the poles of the meromorphic continuation of the cut-off resolvent (see e.g. Hunziker \cite{Hu86_01}). The generalized spectral projection associated to a resonance is defined as the residue of the resolvent at this pole:
\begin{equation*}
\Pi_{z} = - \frac1{2i \pi} \oint_{\gamma_{z}} ( P - \zeta )^{-1} d \zeta ,
\end{equation*}
as an operator from $L^{2}_{\rm comp}$ to $L^{2}_{\rm loc}$. If $z$ were an isolated eigenvalue, this formula would give the usual spectral projection. Many physical quantities can be expressed in terms of these generalized spectral projections. In the case of shape resonances, their semiclassical expansion has been computed by Helffer and Sj\"ostrand in \cite{HeSj86_01}. In  Section \ref{a1} below, we obtain the semiclassical expansion of the generalized  spectral projection for barrier-top resonances. Since the resonances in the present case have a much larger imaginary part, our result is very different from that of the shape resonance case.

Resonances appear also in scattering theory (they are called scattering poles in this context). In \cite{LaPh89_01}, Lax and Phillips have shown that they coincide with the poles of the meromorphic extension of the scattering amplitude. This result, proved for the wave equation in the exterior of a compact obstacle, was extended by G\'erard and Martinez \cite{GeMa89_02} to the long range case for the Schr\"odinger equation (see also the references in this paper for earlier works). For shape resonances, the residue of the scattering amplitude was calculated in the semiclassical limit by Nakamura \cite{Na89_01,Na89_02}, Lahmar-Benbernou \cite{Be99_01} and Lahmar-Benbernou and Martinez \cite{LaMa99_01}. More generally, upper bounds on the residues of the scattering amplitude have been obtained by Stefanov \cite{St02_01} (in the compact support case) and Michel \cite{Mi03_01} (in the long range case) for resonances very close to the real axis. In Section \ref{k2}, we give the semiclassical expansion of the residues of the scattering amplitude for barrier-top resonances and we will see that these upper bounds do not hold in the present setting.

It is commonly believed that resonances play also a crucial role in quantum dynamics. Indeed, it is sometimes possible to describe the long time evolution of the cut-off propagator (for example, the Schr\"{o}dinger or wave group) in term of the resonances. Typically, if the resonances are simple, the propagator $e^{-i t P}$ truncated by $\chi \in C^{\infty}_{0}$ satisfies
\begin{equation*}
\chi e^{- i t P} \chi = \sum_{z \text{ resonance of } P} e^{- i t z} \chi \Pi_{z} \chi + \text{remainder term} .
\end{equation*}
Here, $\Pi_{z}$ is the generalized spectral projection defined previously. Such a formula generalizes the Poisson formula, valid for the operators with discrete spectrum. The resonance expansion of the wave group was first obtained by Lax and Phillips \cite{LaPh89_01} in the exterior of a star-shaped obstacle. This result has been generalized, using various techniques, to different non trapping situations (see e.g. Va{\u\i}nberg \cite{Va89_01} and the references of the second edition of the book \cite{LaPh89_01}). The trapping situations have been treated by Tang and Zworski \cite{TaZw00_01} and Burq and Zworski \cite{BuZw01_01} for very large times. On the other hand, the time evolution of the quasiresonant states (sorts of quasimodes) has been studied by G\'erard and Sigal \cite{GeSi92_01}. A specific study of the Schr\"odinger group for the shape resonances created by a well in a island has been made by Nakamura, Stefanov and Zworski \cite{NaStZw03_01}. There are also some works concerning the  situation of a hyperbolic trapped set. We refer to the work of Christiansen and Zworski \cite{ChZw00_01} for the wave equation on the modular surface and on the hyperbolic cylinder, and to the work of Guillarmou and Naud \cite{GuNa09_01} for the wave equation on convex co-compact hyperbolic manifolds. Section \ref{k3} is devoted to the computation of the asymptotic behavior for large time of the Schr\"odinger group localized in energies close to the maximum of the potential.

For the proof of the different results of this paper, we use an estimate of the distorted resolvent around the resonances, polynomial with respect to $h^{-1}$. Indeed, such a bound allows to apply the semiclassical microlocal calculus. This estimate is established in Section~\ref{k4}. To prove it, we proceed as in \cite{BoFuRaZe07_01} and use the method developed by Martinez \cite{Ma02_01}, Sj\"{o}strand \cite{Sj97_01} and Tang and Zworski \cite{TaZw98_01}. Similar bounds around the resonances are already known in various situations (see e.g. G\'erard \cite{Ge88_01} for two strictly convex obstacles, Michel and the first author \cite{BoMi04_01} in the one dimensional case). Note that, in our setting, a limiting absorption principle have been proved in \cite{AlBoRa08_01}.

\section{Assumptions and resonances}
\label{k5}

We consider the semiclassical Schr\"odinger operator on $\R^{n}$, $n \geq 1$,
\begin{equation}\label{aop}
P = - h^{2} \Delta + V (x),
\end{equation}
where $V$ is a smooth real-valued function. We denote by $p(x,\xi)=\xi^2+V(x)$ the associated classical Hamiltonian. The vector field  
\begin{equation*}
H_p=\partial_\xi p \cdot \partial_x - \partial_x p \cdot \partial_\xi = 2 \xi \cdot \partial_x - \nabla V (x) \cdot \partial_\xi ,
\end{equation*}
 is the Hamiltonian vector field associated to $p$. Integral curves $t \mapsto \exp (t H_p )( x , \xi )$ of $H_p$ are called classical trajectories or bicharacteristic curves, and $p$ is constant along such curves.
The trapped set at energy $E$ for $P$ is defined as 
\begin{equation*}
K(E) = \big\{ (x,\xi) \in p^{-1}(E); \ \exp(tH_p)(x,\xi)\not\to \infty \mbox{ as } t\to \pm \infty \big\} ,
\end{equation*}
We shall suppose that $V$ satisfies the following assumptions
\begin{hyp} \label{A1}
$V \in C^{\infty} ( \R^{n} ; \R )$ extends holomorphically in the sector 
\begin{equation*}
{\mathcal S} =\{ x\in \C^{n} ; \  \vert \im x\vert \leq \delta \bra x \ket\} ,
\end{equation*}
for some $\delta > 0$. Moreover $V(x)\to 0$ as $x\to\infty$ in ${\mathcal S}$.
\end{hyp}
\begin{hyp} \label{A2}
$V$  has a non-degenerate maximum at $x=0$ and
\begin{equation*}
V (x) = E_{0} - \sum_{j=1}^{n} \frac{\lambda_{j}^{2}}{4} x_{j}^{2} + \CO ( x^{3} ) ,
\end{equation*}
with  $E_{0}> 0$ and $0 < \lambda_{1} \leq \lambda_{2} \leq \cdots \leq \lambda_{n}$.
\end{hyp}
\begin{hyp} \label{A3}
The trapped set  at energy $E_0$ is $K (E_{0} ) = \{ ( 0,0) \}$.
\end{hyp}
Notice that \ref{A3} ensures that $x=0$ is the unique global maximum for $V$. Moreover, there exists a pointed neighborhood of $E_{0}$ in which all the energy levels are non trapping. In the following, $( \mu_{k} )_{k \geq 0}$ denote the strictly increasing sequence of linear combinations over $\N = \{ 0 , 1, 2 , \ldots \}$ of the $\lambda_{j}$'s. In particular, $\mu_{0} =0$ and $\mu_{1} = \lambda_{1}$.

The linearization $F_{p}$ at $(0,0)$ of the Hamilton vector field $H_{p}$ is given by
\begin{equation*}
F_{p} =
\left( \begin{array}{cc}
0 & 2 Id \\
\frac{1}{2} \diag ( \lambda_{1}^{2} , \ldots , \lambda_{n}^{2} ) & 0
\end{array} \right) ,
\end{equation*}
and has eigenvalues $- \lambda_{n} , \dots , - \lambda_{1} , \lambda_{1} , \dots , \lambda_{n}$. Thus $(0,0)$ is a hyperbolic fixed point for $H_{p}$ and the stable/unstable manifold theorem gives the existence of a stable incoming Lagrangian manifold $\Lambda_-$ and a stable outgoing Lagrangian manifold $\Lambda_+$ characterized by
\begin{equation*}
\Lambda_{\pm} = \big\{ ( x , \xi ) \in T^{*} \R^{n} ; \ \exp(t H_{p})( x , \xi ) \to (0,0) \mbox{ as } t \to \mp \infty \big\} \subset p^{-1} (E_{0}) .
\end{equation*}
Moreover, there exist two smooth functions $\varphi_{\pm}$, defined in a vicinity of $0$, satisfying
\begin{equation*}
\varphi_{\pm} (x) = \pm \sum_{j=1}^{n} \frac{\lambda_{j}}{4} x_{j}^{2} + \CO (x^{3}) ,
\end{equation*}
and such that $\Lambda_{\pm} = \Lambda_{\varphi_{\pm}} : = \{ (x , \xi ) ; \ \xi = \nabla \varphi_{\pm} (x) \}$ near 
$(0,0)$. Since $P$ is a Schr\"odinger operator, we have $\varphi_{-} = - \varphi_{+}$.

Under the previous assumptions, the operator $P$ is self-adjoint with domain $H^2(\R^n)$, and we define the set $\Res(P)$ of resonances for $P$ as follows (see \cite{Hu86_01}). Let $R_0 > 0$ be a large constant, and let  $F:\R^n\to\R^n$ be a smooth vector field, such that $F(x)=0$ for $\vert x\vert\leq R_0$ and $F(x)=x$ for $\vert x\vert\geq R_{0} +1$. For $\mu \in \R$ small enough, we denote $U_\mu:L^2(\R^n)\to L^2(\R^n)$ the unitary operator defined by
\begin{equation}
U_\mu \varphi (x) = \big\vert \det(1+\mu d F (x)) \big\vert^{1/2}\varphi(x+\mu F(x)) ,
\end{equation}
for $\varphi\in C^{\infty}_{0}(\R^n)$. Then the operator $U_{\mu}P(U_{\mu})^{-1}$  is a differential operator with analytic coefficients with respect to $\mu$, and can be analytically continued to small enough complex values of $\mu$. For $\theta \in \R$ small enough, we denote
\begin{equation}\label{g1}
P_{\theta}=U_{i\theta}P(U_{i\theta})^{-1}.
\end{equation}
The spectrum of $P_{\theta}$ is discrete in $\CE_{\theta}=\{z\in \C ; \ -2\theta<\arg z\leq 0\}$, and  the resonances of $P$ are by definition the eigenvalues of $P_{\theta}$ in $\CE_{\theta}$. We denote their set by $\Res(P)$.
The multiplicity of a resonance is the rank of the spectral projection
\begin{equation*}
\Pi_{z , \theta} = - \frac1{2i \pi} \oint_{\gamma} ( P_{\theta} - \zeta )^{-1} d \zeta ,
\end{equation*}
where $\gamma$ is a small enough closed path around the resonance $z$. The resonances, as well as their multiplicity, do not depend on $\theta$ and $F$. As a matter of fact, the  resonances are also the poles of the meromorphic extension from the upper complex half-plane of the resolvent $(P-z)^{-1}: L^{2}_{\rm comp} ( \R^{n} )\to L^{2}_{\rm loc} ( \R^{n} )$ (see e.g. \cite{HeMa87_01}).

In the present setting, Sj\"ostrand \cite{Sj87_01} has given a precise description of the set of resonances in any disc $D(E_{0} ,C h)$ of center $E_{0}$ and radius $C h$. This result has also been proved simultaneously by Briet, Combes and Duclos \cite{BrCoDu87_02} under a slightly stronger hypothesis (a virial assumption).

\begin{theorem}[Sj\"ostrand]\sl \label{e17}
Assume \ref{A1}--\ref{A3}. Let $C>0$ be different from  $\sum_{j=1}^{n} (\alpha_{j} + \frac{1}{2} ) \lambda_{j}$ for all $\alpha \in \N^{n}$. Then, for  $h>0$ small enough, there exists a bijection $b_h$ between the sets $\Res_0(P)\cap D(E_0,C h)$  and  $\Res(P)\cap D(E_0,C h)$, where 
\begin{equation*}
\Res_0(P) = \Big\{ z_\alpha^0=E_0 - i h \sum_{j=1}^{n} \big( \alpha_j+\frac{1}{2} \big) \lambda_j ; \ \alpha\in\N^n \Big\} ,
\end{equation*}
such that $b_h(z) - z = o (h)$.
\end{theorem}

In particular, the number of resonances in any disk $D(E_0,C h)$ is uniformly bounded with respect to $h$. For $z_{\alpha}^{0} \in \Res_{0} (P)$, we denote $z_{\alpha} = b_{h} ( z_{\alpha}^{0} )$.

\begin{definition}\sl
We shall say that $z_\alpha^0 \in \Res_0(P)$ is simple if $z_\alpha^0=z_\beta^0$ implies $\alpha=\beta$.
\end{definition}

\begin{remark}\sl \label{k1}
If $z_\alpha^0 \in \Res_{0} (P)$ is simple, the corresponding resonance $z_\alpha$ is simple for $h$ small enough and Proposition~0.3 of \cite{Sj87_01} proves that $z_\alpha$ has a complete asymptotic expansion in powers of $h$.
\end{remark}

\begin{remark}\sl
The analyticity of $V$ in a full neighborhood of $\R^{n}$ is used only for the localization of the resonances. Indeed, if the conclusions of Theorem \ref{e17} and Remark \ref{k1} hold for $V$ smooth and analytic outside of a compact set, then the results of this paper still apply under this weaker assumption.
\end{remark}

The semiclassical pseudodifferential calculus is a tool used throughout this paper, and we fix here some notations. We refer to \cite{DiSj99_01} for more details. For $m (x , \xi ,h) \geq 0$ an order function and $\delta \geq 0$, we say that a function $a (x, \xi ,h) \in C^{\infty} ( T^{*} \R^{n} )$ is a symbol of class $S_{h}^{\delta} (m)$ when, for all $\alpha \in \N^{2 n}$,
\begin{equation*}
\big\vert \partial_{x , \xi}^{\alpha} a (x , \xi ,h ) \big\vert \lesssim h^{- \delta \vert \alpha \vert} m (x , \xi ,h) .
\end{equation*}
If $a \in S_{h}^{\delta} (m)$, the semiclassical pseudodifferential operator $\Op (a)$ with symbol $a$ is defined by
\begin{equation*}
\big( \Op (a ) \varphi \big) (x) = \frac{1}{(2 \pi h)^{n}} \iint e^{i (x-y) \cdot \xi /h} a \Big( \frac{x+y}{2} , \xi ,h \Big) \varphi (y) \, d y \, d \xi ,
\end{equation*}
for all $\varphi \in C^{\infty}_{0} ( \R^{n} )$. We denote by $\Psi_{h}^{\delta} (m)$ the space of operators $\Op ( S_{h}^{\delta} (m) )$.

The rest of this paper is organized as follows. In Section \ref{k4}, we prove a resolvent estimate in the complex plane that we use in all the rest of the paper. Then, in Section \ref{a1}, we compute the spectral projection associated to a resonance. In section \ref{k2}, we give the asymptotic expansion of the residue of the scattering amplitude at a simple resonance for long range potentials. Section \ref{k3} is devoted to the computation of the asymptotic behavior for large $t$ of the Schr\"odinger group $e^{-i t P/h}$, where the spectral projection appears naturally. At last, we have placed in Appendix \ref{k6} some geometrical considerations about Hamiltonian curves in a neighborhood of the hyperbolic fixed point, that we need in Section \ref{a1}.

\section{Resolvent estimate}
\label{k4}

In this section, we prove a polynomial estimate for the resolvent of the distorted operator $P_{\theta}$ around the resonances. This estimate is used throughout the paper to control remainder terms. More precisely, we prove the following result.

\begin{theorem}[Resolvent estimate]\sl \label{aaa}
Assume \ref{A1}--\ref{A3}. There exists $\varepsilon >0$ such that, for all $C>0$ and $h$ small enough,

$i)$ The operator $P$ has no resonances in
\begin{equation*}
[ E_{0} - \varepsilon , E_{0} + \varepsilon ] + i [ -C h , 0] \setminus D (E_{0} , 2 C h) .
\end{equation*}

$ii)$ Assume $\theta = \nu h \vert \ln h \vert$ with $\nu >0$. Then, there exists $K >0$ such that
\begin{equation} \label{g3}
\big\Vert ( P_{\theta} - z )^{-1} \big\Vert \lesssim h^{-K} \prod_{z_{\alpha} \in \res (P) \cap D(E_{0},2Ch)} \vert z - z_{\alpha} \vert^{-1} ,
\end{equation}
for all $z \in [ E_{0} - \varepsilon , E_{0} + \varepsilon ] + i [ -C h , C h]$.
\end{theorem}

In particular, the previous theorem states that all the resonances in $[ E_{0} - \varepsilon , E_{0} + \varepsilon ] + i [ -C h , 0]$ are those given by Theorem \ref{e17}. The rest of this section is devoted to the proof of Theorem~\ref{aaa}. We follow the approach of Tang and Zworski \cite{TaZw98_01} and we use the constructions of \cite[Section 4]{BoFuRaZe07_01}, where the propagation of singularities through a hyperbolic fixed point is studied, and of \cite[Section 3]{AlBoRa08_01}, where a sharp estimate for the weighted resolvent for real energies is given.

\Subsection{Definition of a weighted operator $Q_{z}$}

The distorted operator $P_{\theta}$ defined in \eqref{g1} is a differential operator of order $2$ whose symbol $p_{\theta} \in S_{h}^{0} (1)$ satisfies
\begin{equation}
p_{\theta} ( x , \xi , h ) = p_{\theta , 0} ( x , \xi ) + h p_{\theta , 1} ( x , \xi ) + h^{2} p_{\theta , 2} ( x , \xi ) ,
\end{equation}
with $p_{\theta , \bullet} \in S_{h}^{0} ( \< \xi \>^{2} )$ and
\begin{equation*}
p_{\theta,0} ( x , \xi ) = p \big( x + i \theta F(x) , ( 1 + i \theta \, {}^t ( d F ( x )))^{-1} \xi \big) .
\end{equation*}
We write the Taylor expansion of $p_{\theta,0}(x,\xi)$ with respect to $\theta$ as
\begin{equation}\label{g2}
p_{\theta,0}(x,\xi)=p(x,\xi)-i\theta q(x,\xi)+\theta^2 r ( x , \xi , \theta ),\quad  q(x,\xi)=\big \{p(x,\xi),F(x)\cdot \xi\big\},
\end{equation}
for some  $r \in S_{h}^{0} ( \< \xi \>^2 )$ which vanishes in $\vert x\vert\leq R_{0}$. Notice that 
$$
q(x,\xi)=2 d F(x)\xi \cdot \xi -\nabla V(x)\cdot F(x),
$$
so that for $\varepsilon > 0$ small enough, there exists $R_{1} > R_{0} +1$ such that
\begin{equation} \label{e23}
q ( x, \xi ) \geq E_{0} ,
\end{equation}
for all $( x , \xi ) \in p^{-1}([E_0 - 2 \varepsilon , E_0 + 2 \varepsilon ])$ with $\vert x \vert\geq R_{1}$.

We want to gain as much ellipticity as we can near $(0,0)$. As in \cite[Section 4]{BoFuRaZe07_01}, we shall work with a weighted operator, and we start  by defining the weights. Let $\widetilde{p} (x , \xi ) = p ( x, \xi ) - E_{0}$ and $\widetilde{p}_{\theta} (x , \xi ,h ) = p_{\theta} ( x, \xi ,h ) - E_{0}$. There exists a symplectic map $\kappa$ defined near $B( 0 , \varepsilon_2 ) = \{ (x, \xi) \in T^{*} \R^{n} ; \ \vert ( x, \xi ) \vert \leq \varepsilon_2 \}$, with $0 < \varepsilon_{2} \ll \varepsilon$, such that, setting $(y,\eta )=\kappa(x,\xi)$,
\begin{equation}  \label{aa1}
\widetilde{p} (x, \xi) = {\mathcal B} (y, \eta ) y\cdot \eta.
\end{equation}
Here $(y,\eta)\mapsto {\mathcal B} ( y , \eta )$ is a $C^{\infty}$ map from $\kappa ( B( 0 , \varepsilon_2 ) )$ to the space $\CM_{n}( \R )$ of $n \times n$ matrices with real entries such that
\begin{equation*}
{\mathcal B} (0,0) =\diag(\lambda_1,\dots,\lambda_n).
\end{equation*}
Let $U$ be a unitary Fourier integral operator microlocally defined near $B ( 0 , \varepsilon_{2} )$ and associated to the canonical transformation $\kappa$. Then
\begin{equation}
\widehat{P} = U ( P - E_{0} ) U^{-1} ,
\end{equation}
is a pseudodifferential operator in $\Psi_{h}^{0} (1)$ with a real (modulo $S_{h}^{0} (h^{\infty})$) symbol $\widehat{p} (y, \eta )= \sum_{j \geq 0} \widehat{p}_{j} (y, \eta ) h^{j}$, such that
\begin{equation*}
\widehat{p}_{0} (y , \eta ) = {\mathcal B} (y , \eta ) y \cdot \eta .
\end{equation*}

Let $0 < \varepsilon_{1} < \varepsilon_{2}$. Since the trapped set at energy $E_0$ for $p$ is $\{ 0 \}$, we recall from \cite{GeSj87_01} that, for the compact set ${\mathcal K} = B(0 , 2 R_1) \setminus B( 0 , \varepsilon_{1} ) \cap p^{-1}([E_0 - 4 \varepsilon , E_0 + 4 \varepsilon ]) \subset T^{*} \R^{n}$, there exist $0 < \varepsilon_{0} < \varepsilon_{1}$ and a bounded function $g\in C^\infty ( T^{*} \R^{n} )$ such that $H_pg$ has compact support and
\begin{equation} \label{e26}
\left\{ \begin{aligned}
&g( x , \xi ) = 0 , &&\text{if } (x,\xi)\in B( 0 ,\varepsilon_0),\\
&H_p g(x , \xi ) \geq 0, &&\text{if } (x,\xi)\in T^*\R^n,\\
&H_p g( x , \xi ) \geq 1, &&\text{if } (x,\xi)\in {\mathcal K} .
\end{aligned} \right.
\end{equation}
As in \cite{Ma02_01}, we set, for $R \gg R_1$ to be chosen later,
\begin{equation}\label{g0}
g_{0} ( x , \xi ) = \chi_0 \Big( \frac{x}{R} \Big) \psi_0 ( p ( x , \xi ) ) g ( x , \xi ) \vert \ln h \vert ,
\end{equation}
where $\chi_{0} \in C_{0}^{\infty} ( \R^{n} ; [0,1] )$ with $\chi_{0} = 1$ on $B(0,1)$ and $\psi_{0} \in C_{0}^{\infty} ( \R ; [0,1] )$ with $\supp \psi_{0} \subset [ E_{0} - 4 \varepsilon , E_{0} + 4 \varepsilon ]$ and $\psi_{0} = 1$ in a neighborhood of $[ E_{0} - 3 \varepsilon , E_{0} + 3 \varepsilon ]$.

We also define functions on the $( y , \eta )$ side. We set
\begin{equation*}
\left\{ \begin{aligned}
&\widehat{g}_{1} ( y , \eta ) = ( y^2 - \eta^2 )\widehat{\phi}_{1} ( y , \eta ) \vert \ln h \vert , \\
&\widehat{g}_{2} ( y , \eta ) = \Big( \ln \Big\< \frac{y}{\sqrt{h M}} \Big\> - \ln \Big\< \frac{\eta}{\sqrt{h M}} \Big\> \Big) \widehat{\phi}_{2} ( y , \eta ) .
\end{aligned} \right.
\end{equation*}
Here $M > 1$ is a parameter that will be chosen later on. Since we consider the semiclassical regime, we will assume  that $h M < 1$. Moreover, $\widehat{\phi}_{\bullet} = \phi_{\bullet} \circ \kappa^{-1}$, where $\phi_{1} \in C^{\infty}_{0} ( B (0 , \varepsilon_{2} ) )$ is such that $\phi_{1} = 1$ near $ B (0, \varepsilon_{1} )$ and $\phi_{2} \in C^{\infty}_{0} ( B (0 , \varepsilon_{0} ))$ is such that $\phi_{2} = 1$ near $0$ in $T^{*} \R^{n}$. At last, we choose four cut-off functions $\chi_{1} , \chi_{2} , \chi_{3} , \chi_{4} \in C^{\infty}_{0} ( B ( 0 , \varepsilon_2 ) )$ such that, setting again $\widehat{\chi}_{\bullet} = \chi_{\bullet} \circ \kappa^{-1}$, we have
\begin{equation*}
\one_{\{ 0 \}} \prec \widehat{\phi}_{2} \prec \widehat{\phi}_{1} \prec \widehat{\chi}_{1} \prec \widehat{\chi}_{2} \prec \widehat{\chi}_{3} \prec \widehat{\chi}_{4} .
\end{equation*}
The notation $f \prec g$ means that $g =1$ near the support of $f$. We define the operators
\begin{equation*}
G_{\pm 0} = \Op \big( e^{\pm t_0 g_0} \big) , \ G_{\pm j} = \Op \big( e^{\pm t_j \widehat{g}_j} \big) \ \text{and} \ \widetilde{G}_{\pm j} = \Op \big( \widehat{\chi}_{j} e^{\pm t_j \widehat{g}_{j}} \big) ,
\end{equation*}
for $j=1,2$. Notice that $G_{\pm 0}$ is acting on functions of $( x , \xi )$, whereas the other operators are acting on functions of $( y , \eta )$. The $t_{\bullet}$'s are real constants that will be fixed below. Then,
\begin{align}
G_{\pm 0}\in \Psi_{h}^{0} \big( h^{-N_0} \big) , \ G_{\pm 1} & \in \Psi_{h}^{0} \big( h^{- N_{1}} \big) , \ G_{\pm 2} \in \Psi_{h}^{1/2} \big( h^{- N_2} \big) ,   \nonumber \\
&\widetilde{G}_{\pm 1} \in \Psi_{h}^{0} \big( h^{- N_{1}} \< \eta \>^{-\infty} \big) \ \text{ and } \ \widetilde{G}_{\pm 2} \in \Psi_{h}^{1/2} \big( h^{- N_2} \< \eta \>^{-\infty} \big) , \label{d1}
\end{align}
for some $N_{\bullet} \in \R$.

We define the operator
\begin{align}
Q_{z} = \Big( U^{-1} \big( \widetilde{G}_{-2} \widetilde{G}_{-1}-\Op(\widehat{\chi}_1) & \big) U+ Id \Big) G_{-0} (P_\theta -z) \nonumber \\
& G_{+0} \Big( U^{-1} \big( \widetilde{G}_{+1} \widetilde{G}_{+2} - \Op ( \widehat{\chi}_{1} ) \big) U + Id \Big) . \label{e1}
\end{align}
Splitting $P_{\theta} - z = \Op ( \widetilde{p}_{\theta} \chi_{4} ) + \Op ( \widetilde{p}_{\theta}( 1 - \chi_{4} ) ) - ( z - E_{0} )$, we write
\begin{equation*}
Q_{z} = Q_{1} + Q_{2} - ( z - E_{0} ) Q_{3} ,
\end{equation*}
and we compute the symbols of the operators $Q_{\bullet}$ separately.

\Subsection{Computation of $Q_{z}$}

The goal of this part is to prove the following identity.

\begin{lemma}\sl \label{e41}
Let $Q_{z}$ be the operator defined in \eqref{e1}. Then,
\begin{align}
\nonumber
Q_{z} =& \Op ( p_{\theta} ) + \Op ( i h t_{0} \{ g_{0} , p_{\theta} \} ) + U^{-1} \Op \big( i h t_{1} \{ \widehat{g}_{1} , \widehat{p}_{0} \} + i h t_{2} \{ \widehat{g}_{2} , \widehat{p}_{0} \} \big) U -z
\nonumber  \\
&+ \CO (h M^{-1} ) + \CO ( h^{\frac{3}{2}} M^{- \frac{1}{2}} \vert \ln h \vert^{2} ) + \CO ( \vert z - E_{0} \vert M^{-2} ) . \label{glou10}
\end{align}
\end{lemma}

\begin{remark}\sl \label{e40}
We will show in the proof of Lemma \ref{e41} (more precisely in \eqref{e16}) that the operators $( U^{-1} ( \widetilde{G}_{-2} \widetilde{G}_{-1}-\Op(\widehat{\chi}_1) ) U+ Id ) G_{-0}$ and $G_{+0} ( U^{-1} ( \widetilde{G}_{+1} \widetilde{G}_{+2} - \Op ( \widehat{\chi}_{1} ) ) U + Id )$ are invertible on $L^{2} ( \R^{n} )$ and $H^{2} ( \R^{n} )$ for $M^{-1}$ and $h$ small enough. Moreover, their inverses are polynomially bounded in $h^{-1}$. In particular, the resonances of $P$ are the poles of $Q_{z}^{-1}$ and to estimate $( P_{\theta} -z )^{-1}$, it is enough to estimate $Q_{z}^{-1}$.
\end{remark}

The rest of this section is devoted to the proof of Lemma~\ref{e41}. In fact, \eqref{glou10} is close to the equation (4.44) of \cite{BoFuRaZe07_01} and we will use some identities from \cite{BoFuRaZe07_01} when possible.

\begin{proof}
$\bullet$ First we consider $Q_{1}$. Since we can assume that $R_{0} > \varepsilon_{2}$, we have
\begin{equation*}
\Op ( \widetilde{p}_{\theta} \chi_{4} ) G_{+0} = \Op ( \widetilde{p} \chi_{4} ) G_{+0} = \Op ( a_{1} ) ,
\end{equation*}
with $a_{1} \in S_{h}^{0} (h^{-N_0})$ given, for any $k_0\in\N$, by
\begin{equation}  \label{e4}
a_{1} ( x , \xi ) = \sum_{k=0}^{k_{0}} \frac 1{k!}\Big( \big( \frac{i h}{2} \sigma ( D_{x} , D_{\xi} ; D_{y} , D_{\eta}) \big)^{k} \widetilde{p} \chi_{4} ( x , \xi ) e^{t_{0} g_{0} ( y , \eta )} \Big) \Big\vert_{y = x , \eta = \xi}  + h^{k_{0}-N_0} S^0_h(1).
\end{equation}
Then again
\begin{equation} \label{e6}
G_{-0} \Op ( \widetilde{p}_{\theta} \chi_{4} ) G_{+0} = G_{-0} \Op (a_{1}) = \Op ( a_{2} ) ,
\end{equation}
with $a_{2} \in S_{h}^{0} ( h^{-N_0})$ given, for any $k_1\in\N$, by
\begin{equation} \label{e5}
a_{2} (x,\xi) = \sum_{k=0}^{k_{1}} \frac 1{k!}\Big (
\big(\frac{i h}2 \sigma(D_{x},D_{\xi};D_{y},D_{\eta})\big)^k
e^{-t_{0}g_{0}(x,\xi)} a_{1} (y,\eta)
\Big) \Big \vert_{y=x,\eta=\xi}  + h^{k_{1}-N_0} S^0_h(1).  
\end{equation}
The $k$-th term in \eqref{e5} is easily seen to be $\CO ( h^{k} )$, so that choosing $k_1$ large enough, we conclude that $a_{2} \in S_{h}^{0} (1)$. Moreover $\supp a_{2} \subset \supp \chi_{4}$ modulo $S_{h}^{0} ( h^{\infty} )$, and
\begin{equation} \label{e7}
a_{2} = \widetilde{p} \chi_{4} + i h t_{0} \{ g_{0} , \widetilde{p} \chi_{4} \} + S_{h}^{0} ( h^2 \vert \ln h \vert^{2} ) = \widetilde{p} \chi_{4} + a_{3} ,
\end{equation}
for some $a_{3} \in S_{h}^{0} ( h \vert \ln h \vert )$ with $\supp a_{3} \subset \supp \chi_{4} \cap \supp g_0$ modulo $S_{h}^{0} ( h^{\infty} )$.

By Egorov's theorem,
\begin{equation} \label{e9}
U \Op ( \widetilde{p} \chi_{4} ) U^{-1} = \Op ( \widehat{a}_{4} ) \ \text{ and } \ U \Op ( a_{2} ) U^{-1} = \Op ( \widehat{a}_{5} ) ,
\end{equation}
where $\widehat{a}_{4} , \widehat{a}_{5} \in S_{h}^{0}(1)$ verify $\supp \widehat{a}_{4} , \supp \widehat{a}_{5} \subset \supp \widehat{\chi}_{4}$ modulo $S_{h}^{0} ( h^{\infty} )$. Moreover, from \eqref{e7}, we have
\begin{equation} \label{e11}
\widehat{a}_{5} = \widehat{a}_{4} + i h t_{0} \{ \widehat{g}_{0} , \widehat{p} \widehat{\chi}_{4} \} + S_{h}^{0} ( h^{2} \vert \ln h \vert^{2} ) = \widehat{a}_{4} + \widehat{a}_{6} ,
\end{equation}
with $\widehat{g}_{0} = g_{0} \circ \kappa^{-1}$ and a symbol $\widehat{a}_{6} \in S_{h}^{0}( h \vert \ln h \vert )$ satisfying $\supp \widehat{a}_{6} \subset \supp \widehat{\chi}_{4} \cap \supp \widehat{g}_0$ modulo $S_{h}^{0} ( h^{\infty} )$. Since $\phi_{1} , \phi_{2} \prec \chi_{1} \prec \chi_{2}$, we have $\widehat{g}_{1} , \widehat{g}_{2} \prec \widehat{\chi}_{1}$ and we get by pseudodifferential calculus
\begin{equation} \label{e8}
\widetilde{G}_{\pm 2} \widetilde{G}_{\pm 1} - \Op(\widehat{\chi}_1) + Id = G_{\pm 2} G_{\pm 1} + \CO ( h^{\infty} ) .
\end{equation}
Then, using \eqref{e6}, \eqref{e9}, \eqref{e11} and \eqref{e8}, we obtain
\begin{align}
Q_{1} &= U^{-1} \big( \widetilde{G}_{-2} \widetilde{G}_{-1} - \Op(\widehat{\chi}_1) + Id \big) U \Op ( a_{2} ) U^{-1} \big( \widetilde{G}_{+1} \widetilde{G}_{+2} - \Op(\widehat{\chi}_1) + Id \big) U + \CO ( h^{\infty} )  \nonumber \\
&= U^{-1} G_{-2} G_{-1} \Op ( \widehat{a}_{4} ) G_{+1} G_{+2} U + U^{-1} G_{-2} G_{-1} \Op ( \widehat{a}_{6} ) G_{+1} G_{+2} U + \CO ( h^{\infty} ) . \label{e10}
\end{align}

The first term in the right hand side of \eqref{e10} has already been computed in the equations (4.15)--(4.41) of \cite{BoFuRaZe07_01} (the reader should notice however that the symbol $p$ there has to be replaced by $p \chi_{4}$ here). We have
\begin{align}
G_{-2} G_{-1} \Op ( \widehat{a}_{4} ) G_{+1} G_{+2}=& \Op \big( \widehat{a}_{4} + i h t_{1} \{ \widehat{g}_{1} , \widehat{p}_{0} \widehat{\chi}_{4} \} + i h t_{2} \{ \widehat{g}_{2} , \widehat{p}_{0} \widehat{\chi}_{4} \} \big)   \nonumber \\
&+ \CO (h M^{-1} ) + \CO ( h^{\frac{3}{2}} M^{- \frac{1}{2}} \vert \ln h \vert^{2} ) .   \label{e12}
\end{align}

On the other hand, since $\supp \phi_{2} \subset B (0 , \varepsilon_{0})$, $\widehat{g}_{2} = 0$ near the support of $\widehat{g}_{0}$ and $\widehat{a}_{6}$. Thus,
\begin{equation*}
G_{-2} G_{-1} \Op ( \widehat{a}_{6}) G_{+1} G_{+2} = G_{-1} \Op ( \widehat{a}_{6}) G_{+1} + \CO ( h^{\infty} ) .
\end{equation*}
And then, working as in \eqref{e4}--\eqref{e7}, we obtain
\begin{equation} \label{e13}
G_{-2} G_{-1} \Op ( \widehat{a}_{6}) G_{+1} G_{+2} = \Op \big( i h t_{0} \{ \widehat{g}_{0} , \widehat{p} \widehat{\chi}_{4} \} \big) + \CO ( h^{2} \vert \ln h \vert^{2} ) .
\end{equation}

Using \eqref{e9} and collecting \eqref{e12} and \eqref{e13}, the identity \eqref{e10} gives
\begin{align}
\nonumber
Q_{1} =& \Op ( \widetilde{p} \chi_{4} ) + \Op ( i h t_{0} \{ g_{0} , \widetilde{p} \chi_{4} \} ) + U^{-1} \Op \big( i h t_{1} \{ \widehat{g}_{1} , \widehat{p}_{0} \} + i h t_{2} \{ \widehat{g}_{2} , \widehat{p}_{0} \} \big) U   \nonumber  \\
&+ \CO (h M^{-1} ) + \CO ( h^{\frac{3}{2}} M^{- \frac{1}{2}} \vert \ln h \vert^{2} ) .
\label{glou7}
\end{align}

$\bullet$ Now we consider $Q_{2}$.  As in \eqref{e4}--\eqref{e7}, we have
\begin{equation*}
G_{-0} \Op ( \widetilde{p}_\theta (1-\chi_{4} ) ) G_{+0} = \Op ( b_{1} )
\end{equation*}
for some $b_{1} \in S_{h}^{0} ( h^{- N_{0}} \< \xi \>^2)$. Moreover $\supp b_{1} \subset \supp (1 - \chi_{4} )$ modulo $S_{h}^{0} ( h^{\infty} )$ and
\begin{equation} \label{e14}
b_{1} = \widetilde{p}_{\theta} ( 1 - \chi_{4} ) + i h t_0 \{ g_0 , \widetilde{p}_\theta (1 - \chi_{4} ) \} + S_{h}^{0} ( h^2 \vert \ln h \vert^{2} ) .
\end{equation}
Since $\widehat{\chi}_{1} \prec \widehat{\chi}_{3}$, the pseudodifferential calculus gives $\widetilde{G}_{-1} = \widetilde{G}_{-1} \Op ( \widehat{\chi}_{3} ) + \Psi_{h}^{0} ( h^{\infty} \< \eta \>^{- \infty} )$. Furthermore, using Egorov's theorem, we obtain
\begin{align}
U^{-1} \big( \widetilde{G}_{-2}\widetilde{G}_{-1}-\Op(\widehat{\chi}_1) \big) U &= U^{-1} \big( \widetilde{G}_{-2}\widetilde{G}_{-1}-\Op(\widehat{\chi}_1) \big) \Op ( \widehat{\chi}_{3} ) U + \Psi_{h}^{0} ( h^{\infty} \< \xi \>^{- \infty} ) \nonumber  \\
&= U^{-1} \big( \widetilde{G}_{-2}\widetilde{G}_{-1}-\Op(\widehat{\chi}_1) \big) U \Op ( b_{2} ) + \Psi_{h}^{0} ( h^{\infty} \< \xi \>^{- \infty} ) , \label{e15}
\end{align}
where $b_{2} \in S_{h}^{0} ( \< \xi \>^{- \infty} )$ and $\supp b_{2} \subset \supp \chi_{3}$ modulo $S_{h}^{0} ( h^{\infty} \< \xi \>^{- \infty} )$. Using $\chi_{3} \prec \chi_{4}$, the supports of $b_{1}$ and $b_{2}$ are disjoint and
\begin{equation}  \label{glou9}
Q_{2} = \Op ( b_{1} ) + \CO ( h^{\infty} ) .
\end{equation}

$\bullet$ It remains to study $Q_{3}$. Working as in \eqref{e4}--\eqref{e7}, we get $G_{-0} G_{+ 0} = Id + \Op ( c_{1} )$ with $c_{1} \in S_{h}^{0} ( h^{2} \vert \ln h \vert^{2} )$ and $\supp c_{1} \subset \supp g_{0}$ modulo $S_{h}^{0} ( h^{\infty} )$. As in \eqref{e9}, we have
\begin{equation*}
U \Op ((1+ c_{1} ) \chi_{4} ) U^{-1} = \Op ( \widehat{c}_{2} ) ,
\end{equation*}
where $\widehat{c}_{2} \in S_{h}^{0} (1)$. Now \eqref{e8} and \eqref{e15} yield
\begin{align}
Q_{3} &= \Big( U^{-1} \big( \widetilde{G}_{-2}\widetilde{G}_{-1}-\Op(\widehat{\chi}_1) \big) U+ Id \Big) \big( \Op ((1+ c_{1} ) \chi_{4} ) + \Op ((1+ c_{1} ) (1 - \chi_{4} )) \big) \nonumber  \\
&\hspace{230pt} \Big( U^{-1} \big( \widetilde{G}_{+1} \widetilde{G}_{+2} - \Op ( \widehat{\chi}_{1} ) \big) U + Id \Big)   \nonumber  \\
&= U^{-1} G_{-2} G_{-1} \Op ( \widehat{c}_{2} ) G_{+1} G_{+2} U + \Op ( (1+ c_{1} ) (1 - \chi_{4} ) ) + \CO (h^{\infty} ) ,  \label{e2}
\end{align}
Working as in the equation (4.43) of \cite{BoFuRaZe07_01}, we get
\begin{equation*}
G_{-2} G_{-1} \Op ( \widehat{c}_{2} ) G_{+1} G_{+2} = \Op ( \widehat{c}_{2} ) + \CO ( M^{-2} ) + \CO ( h^{2} \vert \ln h \vert^{2} ) .
\end{equation*}
Combining \eqref{e2} with the last identity, we finally obtain
\begin{align}
Q_{3} &= U^{-1} \Op ( \widehat{c}_{2} ) U + \Op ( (1+ c_{1} ) (1 - \chi_{4} ) ) + \CO ( M^{-2} ) + \CO ( h^{2} \vert \ln h \vert^{2} ) \nonumber \\
&= Id + \CO ( M^{-2} ) + \CO ( h^{2} \vert \ln h \vert^{2} ) . \label{e3}
\end{align}

$\bullet$ The same way, one can prove
\begin{align*}
\Big( U^{-1} \big( \widetilde{G}_{-2} \widetilde{G}_{-1}-\Op(\widehat{\chi}_1) \big) U+ Id \Big) \Big( U^{-1} \big( \widetilde{G}_{+1} & \widetilde{G}_{+2} - \Op ( \widehat{\chi}_{1} ) \big) U + Id \Big) \\
&= Id + \CO ( M^{-2} ) + \CO ( h^{2} \vert \ln h \vert^{2} ) ,
\end{align*}
and the same kind of estimate holds for the product the other way round. On the other hand, $G_{-0} G_{+0} = Id + \CO ( h^{2} \vert \ln h \vert^{2} )$ and $G_{+0} G_{-0} = Id + \CO ( h^{2} \vert \ln h \vert^{2} )$. Then the two operators $( U^{-1} ( \widetilde{G}_{-2} \widetilde{G}_{-1}-\Op(\widehat{\chi}_1) ) U+ Id ) G_{-0}$ and $G_{+0} ( U^{-1} ( \widetilde{G}_{+1} \widetilde{G}_{+2} - \Op ( \widehat{\chi}_{1} ) ) U + Id )$ are invertible on $L^{2} ( \R^{n})$ for $M^{-1}$ and $h$ small enough and they satisfy
\begin{equation} \label{e16}
\begin{gathered}
\Big\Vert \Big( \Big( U^{-1} \big( \widetilde{G}_{-2} \widetilde{G}_{-1}-\Op(\widehat{\chi}_1) \big) U+ Id \Big) G_{-0} \Big)^{-1} \Big\Vert = \CO ( h^{-C} ) ,  \\
\Big\Vert \Big( G_{+0} \Big( U^{-1} \big( \widetilde{G}_{+1} \widetilde{G}_{+2} - \Op ( \widehat{\chi}_{1} ) \big) U + Id \Big) \Big)^{-1} \Big\Vert = \CO ( h^{-C} ) ,
\end{gathered}
\end{equation}
for some $C > 0$. The same thing can be done on $H^{2} ( \R^{n} )$ since the operators we consider differ from $Id$ by compactly supported pseudodifferential operators. This shows Remark \ref{e40}.

$\bullet$ Adding \eqref{glou7}, \eqref{glou9} and \eqref{e3}, we get Lemma \ref{e41}
\end{proof}

\Subsection{Estimates on the inverse of $Q_{z}$}

Let $\widehat{\varphi} \in C^{\infty}_{0} ( T^{*} \R^{n} ; [ 0 , 1] )$ be such that $\widehat{\varphi} =1$ near $0$. We define
\begin{equation}
\widetilde{K} = U^{-1} \widehat{K} U \ \text{ with } \ \widehat{K} = C_{1} \Op \Big( \widehat{\varphi} \Big( \frac{y}{\sqrt{h M}} , \frac{\eta}{\sqrt{h M}} \Big) \Big) ,
\end{equation}
for some large constant $C_{1} > 1$ fixed in the following.

\begin{lemma}\sl \label{e39}
Assume that $\delta > 0$, $C_{0} >1$ and $\theta = \nu h \vert \ln h \vert$ with $\nu >0$. Denote $r= \max ( \vert z - E_{0} \vert , h)$. Choose $M = \mu \sqrt{\frac{r}{h}}$ and fix $t_{2} , C_{1} , t_{1} , t_{0} ,R , \mu$ large enough in this order. Then, we have, for $h$ small enough,

$i)$ For $z \in [E_{0} - \varepsilon , E_{0} + \varepsilon ] + i [ - 2 C_{0} h , 2 C_{0} h ]$ and $\im z \geq \delta h$, the operator $Q_{z} : H^{2} ( \R^{n} ) \to L^{2} ( \R^{n} )$ is invertible and
\begin{equation}  \label{e24}
\big\Vert Q_{z}^{-1} \big\Vert = \CO ( h^{-1} ) .
\end{equation}

$ii)$ For $z \in [E_{0} - \varepsilon , E_{0} + \varepsilon ] + i [ - 2 C_{0} h , 2 C_{0} h ]$, the operator $Q_{z} - i h \widetilde{K} : H^{2} ( \R^{n} ) \to L^{2} ( \R^{n} )$ is invertible and
\begin{equation}  \label{e25}
\big\Vert ( Q_{z} - i h \widetilde{K} )^{-1} \big\Vert = \CO ( h^{-1} ) .
\end{equation}
\end{lemma}

This lemma is similar to Proposition~4.1 of \cite{BoFuRaZe07_01}. We will only give the proof of part $ii)$ since the first part can be proved the same way (using \eqref{e19} instead of \eqref{e20}).

\begin{proof}
Let $\omega_{1} , \ldots , \omega_{5} \in C^{\infty}_{0} ( T^{*} \R^{n} ; [ 0 , 1] )$ be such that
\begin{equation}
\one_{\{ 0 \}} \prec \omega_{1} \prec \omega_{2} \prec \phi_{2} \prec \one_{B (0 , \varepsilon_{1} )} \prec \omega_{3} \prec \omega_{4} \prec \phi_{1} \prec \omega_{5} \prec \one_{B (0 , \varepsilon_{2} )} .
\end{equation}
As usual, we denote $\widehat{\omega}_{\bullet} = \omega_{\bullet} \circ \kappa^{-1}$. We now recall some ellipticity estimates proved in \cite{BoFuRaZe07_01} by means of G{\aa}rding's inequality and Calder\`{o}n--Vaillancourt's theorem. From the equations (4.50), (4.51), (4.54), (4.55) and (4.64) of \cite{BoFuRaZe07_01}, we have
\begin{align}
\big( \Op \big( - h \{ \widehat{g}_{2} , \widehat{p}_{0} \} (1 -  \widehat{\omega}_{2}^{2} ) \big) u , u \big) & \geq - C h \vert \ln h \vert \big\Vert \Op ( \widehat{\omega}_{4} - \widehat{\omega}_{1} ) u \big\Vert^{2} + \CO ( h^{\infty} ) \Vert u \Vert^{2} ,  \label{e18}  \\
\big( \Op \big( - h \{ \widehat{g}_{2} , \widehat{p}_{0} \} \widehat{\omega}_{2}^{2} \big) u , u \big) & \geq - C h M^{-1} \Vert u \Vert^{2}  ,   \label{e19}  \\
\big( \Op \big( - h t_{2} \{ \widehat{g}_{2} , \widehat{p}_{0} \} \widehat{\omega}_{2}^{2} + C_{1} h & \widehat{\varphi} \big)  u , u \big)  \nonumber  \\
&\geq \delta \min ( t_{2} , C_{1} ) h \big\Vert \Op ( \widehat{\omega}_{2} ) u \big\Vert^{2} + \CO ( h M^{-1} ) \Vert u \Vert^{2} ,  \label{e20}  \\
\big( \Op \big( - h \{ \widehat{g}_{1} , \widehat{p}_{0} \} (1 -  \widehat{\omega}_{4}^{2} ) \big) u , u \big) & \geq - C h \vert \ln h \vert \big\Vert \Op ( \widehat{\omega}_{5} - \widehat{\omega}_{3} ) u \big\Vert^{2} + \CO ( h^{\infty} ) \Vert u \Vert^{2} ,  \label{e21} \\
\big( \Op \big( - h \{ \widehat{g}_{1} , \widehat{p}_{0} \} \widehat{\omega}_{4}^{2} \big) u , u \big) & \geq \delta h \vert \ln h \vert \big\Vert \Op ( \widehat{\omega}_{4} - \widehat{\omega}_{1} ) u \big\Vert^{2} + \CO ( h^{2} \vert \ln h \vert ) \Vert u \Vert^{2} ,  \label{e22}
\end{align}
for some $\delta , C>0$ which do not depend on $h$, $M$ and the $t_{\bullet}$'s.

From \eqref{g2} and since $\theta = \nu h \vert \ln h \vert$,
\begin{equation}  \label{e27}
\Op ( p_{\theta} ) + \Op ( i h t_{0} \{ g_{0} , p_{\theta} \} ) = \Op ( p - i \theta q + i h t_{0} \{ g_{0} , p \} ) + \Psi_{h}^{0} ( h^{2} \vert \ln h \vert^{2} \< \xi \>^{2} ) .
\end{equation}
Let $\omega_{6} \in C^{\infty}_{0} ( T^{*} \R^{n} ;[0,1])$ be such that
\begin{equation}
\one_{B ( 0 , R_{1} ) \cap p^{-1} ( [ E_{0} - 2 \varepsilon , E_{0} + 2 \varepsilon ] )} \prec \omega_{6} \prec \one_{B ( 0 , 2 R_{1} ) \cap p^{-1} ( [ E_{0} - 3 \varepsilon , E_{0} + 3 \varepsilon ] )} .
\end{equation}
From the definition \eqref{g0} of $g_{0}$, we have
\begin{equation*}
- \{ g_{0} , p \} = \chi_{0} \Big( \frac{x}{R} \Big) \psi_{0} ( p ) H_{p} g \vert \ln h \vert + \frac{2}{R} \xi \cdot ( \partial_{x} \chi_{0} ) \Big( \frac{x}{R} \Big) \psi_{0} ( p (x, \xi )) g \vert \ln h \vert .
\end{equation*}
Using G{\aa}rding's inequality, \eqref{e26} implies
\begin{align}
\big( \Op ( - h t_{0} \{ g_{0} , p \} \omega_{6}^{2} ) u , u \big) &\geq t_{0} h \vert \ln h \vert \big\Vert \Op ( \omega_{6} - \omega_{3} ) u \big\Vert^{2} \nonumber \\
&\hspace{50pt} - C \frac{t_{0}}{R} h \vert \ln h \vert \big\Vert \Op ( 1 - \omega_{2} ) u \big\Vert^{2} + \CO ( h^{2} \vert \ln h \vert ) \Vert u \Vert^{2} .   \label{e28}
\end{align}

Let $\psi \in C^{\infty}_{0} ( [E_{0} - 2 \varepsilon , E_{0} + 2 \varepsilon ] ;[0,1])$ with $\psi =1$ near $[E_{0} - \varepsilon , E_{0} + \varepsilon ]$. Using the functional calculus for pseudodifferential operators, we can write
\begin{align*}
\big( \Op (q ) u , u \big) &= \big( \Op (q ) \psi (P) u , u \big) + \big( \Op (q ) (1- \psi (P)) u , u \big)  \\
&= \big( \Op (q \psi (p) ) u , u \big) + \big( \Op (q ) (P+i)^{-1} (P+i) (1- \psi (P)) u , u \big) + \CO (h) \Vert u \Vert^{2} .
\end{align*}
Note that the operator $\Op (q ) (P+i)^{-1}$ is uniformly bounded on $L^{2} ( \R^{n} )$. G{\aa}rding's inequality together with \eqref{e23} give
\begin{align}
\big( \Op (q ) u , u \big) &\geq \delta \big\Vert \Op ( \psi (p) ( 1 - \omega_{6} ) ) u \big\Vert^{2} - C \big\Vert ( P +i) ( 1- \psi (P) ) u \big\Vert \Vert u \Vert  \nonumber  \\
&\hspace{140pt} - C \big\Vert \Op ( \omega_{6} - \omega_{4} ) u \big\Vert^{2} + \CO ( h ) \Vert u \Vert^{2} .   \label{e29}
\end{align}

Adding \eqref{e18}, \eqref{e20}, \eqref{e21} and \eqref{e22} and using G{\aa}rding's inequality, we obtain
\begin{align}
- \im \big( \big( U^{-1} \Op \big( i h t_{1} & \{ \widehat{g}_{1} , \widehat{p}_{0} \} + i h t_{2} \{ \widehat{g}_{2} , \widehat{p}_{0} \} \big) U - i h \widetilde{K} \big) u , u \big)    \nonumber \\
\geq & \delta t_{1} h \vert \ln h \vert \big\Vert \Op ( \omega_{4} - \omega_{1} ) u \big\Vert^{2} + \delta \min ( t_{2} , C_{1} ) h \big\Vert \Op ( \omega_{2} ) u \big\Vert^{2}   \nonumber  \\
&- C t_{1} h \vert \ln h \vert \big\Vert \Op ( \omega_{5} - \omega_{3} ) u \big\Vert^{2} - C t_{2} h \vert \ln h \vert \big\Vert \Op ( \omega_{4} - \omega_{1} ) u \big\Vert^{2}  \nonumber \\
&+ \CO ( h M^{-1} ) \Vert u \Vert^{2} + \CO ( h^{2} \vert \ln h \vert ) \Vert u \Vert^{2} .  \label{e30}
\end{align}
Combining the formulas \eqref{glou10} and \eqref{e27} and the estimates \eqref{e28}, \eqref{e29} and \eqref{e30}, we get
\begin{align}
- \im \big( (Q_{z} - i h \widetilde{K} ) u , u \big) \geq & \delta \min ( t_{2} , C_{1} ) h \big\Vert \Op ( \omega_{2} ) u \big\Vert^{2} + \delta t_{1} h \vert \ln h \vert \big\Vert \Op ( \omega_{4} - \omega_{1} ) u \big\Vert^{2}      \nonumber \\
&+ t_{0} h \vert \ln h \vert \big\Vert \Op ( \omega_{6} - \omega_{3} ) u \big\Vert^{2} + \delta \nu h \vert \ln h \vert \big\Vert \Op ( \psi (p) ( 1 - \omega_{6} ) ) u \big\Vert^{2}   \nonumber  \\
&- C t_{2} h \vert \ln h \vert \big\Vert \Op ( \omega_{4} - \omega_{1} ) u \big\Vert^{2} - C t_{1} h \vert \ln h \vert \big\Vert \Op ( \omega_{5} - \omega_{3} ) u \big\Vert^{2}   \nonumber \\
&- C \frac{t_{0}}{R} h \vert \ln h \vert \big\Vert \Op ( 1 - \omega_{2} ) u \big\Vert^{2} - C \nu h \vert \ln h \vert \big\Vert \Op ( \omega_{6} - \omega_{4} ) u \big\Vert^{2}     \nonumber \\
&- C \nu h \vert \ln h \vert \big\Vert ( P +i) ( 1- \psi (P) ) u \big\Vert \Vert u \Vert + \im z \Vert u \Vert^{2}  \nonumber \\
&+ \CO ( h^{\frac{3}{2}} M^{- \frac{1}{2}} \vert \ln h \vert^{2} ) \Vert u \Vert^{2} + \CO ( h M^{-1} ) \Vert u \Vert^{2} + \CO ( \vert z - E_{0} \vert M^{-2} ) \Vert u \Vert^{2} . \label{k10}
\end{align}
Now, assume that $\im z \in [ - 2 C_{0} h , 2 C_{0} h]$ and $\re z - E_{0}$ is small. We choose the parameters, in this order, $\min ( t_{2} , C_{1}) \gg C_{0}$, $t_{1} \gg t_{2}$, $t_{0} \gg \max ( t_{1} , \nu )$ then $R \gg 1$ and finally $M = \mu \sqrt{\frac{r}{h}}$ with $\mu \gg 1$. Then, for $h$ small enough, G{\aa}rding's inequality implies
\begin{align}
\big\Vert (Q_{z} - i h \widetilde{K} ) u \big\Vert \Vert u \Vert &\geq - \im \big( (Q_{z} - i h \widetilde{K} ) u , u \big)  \nonumber \\
&\geq h \Vert \psi ( P ) u \Vert^{2} + \CO ( h \vert \ln h \vert ) \big\Vert ( P +i) ( 1- \psi (P) ) u \big\Vert^{2} . \label{e31}
\end{align}

On the other hand, from \eqref{glou10}, we have
\begin{equation*}
Q_{z} - i h \widetilde{K} = P - z + \Psi_{h}^{0} \big( h \vert \ln h \vert \< \xi \>^{2} \big) + \CO ( h \vert \ln h \vert ).
\end{equation*}
Then,
\begin{align}
\big\Vert ( Q_{z} - i h \widetilde{K} ) u \big\Vert &\geq \big\Vert ( 1 - \psi (P ) ) ( Q_{z} - i h \widetilde{K} ) u \big\Vert  \nonumber \\
&\geq \big\Vert ( 1 - \psi (P ) ) ( P-z ) u \big \Vert + \CO ( h \vert \ln h \vert ) \big\Vert ( P +i) u \big\Vert  \nonumber  \\
&\gtrsim \big\Vert ( P +i) (1 - \psi (P) ) u \big\Vert + \CO ( h \vert \ln h \vert ) \big\Vert ( P +i) u \big\Vert  \nonumber \\
&\gtrsim \big\Vert ( P +i) (1 - \psi (P) ) u \big\Vert + \CO ( h \vert \ln h \vert ) \big\Vert \psi ( P) u \big\Vert ,  \label{e32}
\end{align}
for all $h$ small enough.

Summing \eqref{e31} and $C_{2} h \vert \ln h \vert$ times the square of \eqref{e32}, we obtain
\begin{equation*}
\big\Vert ( Q_{z} - i h \widetilde{K} ) u \big\Vert \Vert u \Vert + C_{2} h \vert \ln h \vert \big\Vert ( Q_{z} - i h \widetilde{K} ) u \big\Vert^{2} \gtrsim h \Vert (P +i) u \Vert^{2} ,
\end{equation*}
for $C_{2}$ fixed large enough. Then, using $\Vert ( Q_{z} - i h \widetilde{K} ) u \Vert \Vert u \Vert \leq \delta h \Vert u \Vert^{2} + \frac{1}{\delta h} \Vert ( Q_{z} - i h \widetilde{K} ) u \Vert^{2}$ with $0 < \delta \ll 1$, we finally obtain
\begin{equation}
\big\Vert ( Q_{z} - i h \widetilde{K} ) u \big\Vert \gtrsim h \Vert ( P+i) u \Vert .
\end{equation}
Since we can obtain the same way the same estimate for the adjoint $( Q_{z} - i h \widetilde{K} )^{*}$, we get the lemma.
\end{proof}

To prove the part $i)$ of Theorem \ref{aaa} (the resonance free zone), we will use in addition the following lemma.

\begin{lemma}\sl \label{e37}
Assume $\vert z - E_{0} \vert \geq h$. Under the assumptions of Lemma \ref{e39}, we have
\begin{equation*}
\big\Vert \widetilde{K} Q_{z} u \big\Vert = \vert z - E_{0} \vert \big\Vert \widetilde{K} u \big\Vert + \CO ( h^{\frac{1}{2}} \vert z -E_{0} \vert^{\frac{1}{2}} ) \Vert u \Vert .
\end{equation*}
\end{lemma}

\begin{proof}
Since $\Vert \widetilde{K} \Vert \lesssim 1$, \eqref{glou10} gives
\begin{align}
\widetilde{K} Q_{z} = & \widetilde{K} \Op ( \widetilde{p}_{\theta} ) + \widetilde{K} \Op ( i h t_{0} \{ g_{0} , p_{\theta} \} ) + \widetilde{K} U^{-1} \Op \big( i h t_{1} \{ \widehat{g}_{1} , \widehat{p}_{0} \} + i h t_{2} \{ \widehat{g}_{2} , \widehat{p}_{0} \} \big) U    \nonumber  \\
&- ( z - E_{0} ) \widetilde{K} + \CO (h M^{-1} ) + \CO ( h^{\frac{3}{2}} M^{- \frac{1}{2}} \vert \ln h \vert^{2} ) + \CO ( \vert z - E_{0} \vert M^{-2} ) .  \label{e33}
\end{align}
Since the support of $\widehat{g}_{0}$ does not intersect the support of the symbol of $\widehat{K}$, we obtain
\begin{equation} \label{e42}
\widetilde{K} \Op ( i h t_{0} \{ g_{0} , p_{\theta} \} ) = \CO (h^{\infty} ).
\end{equation}
Moreover, working as in \eqref{e15},
\begin{align*}
\widetilde{K} \Op ( \widetilde{p}_{\theta} ) &= U^{-1} \widehat{K} U \Op ( \widetilde{p} \chi_{4} ) + \CO (h^{\infty} )  \\
&= U^{-1} \widehat{K} \Op ( \widehat{p} ) U + \CO ( h^{\infty} ) .
\end{align*}
We now rescale the variables as in \cite{BuZw04_01} and in the equation (4.18) of \cite{BoFuRaZe07_01}. We define a unitary transformation $V$ on $L^{2} ( \R^{n})$ by
\begin{equation*}
(V f ) ( y ) = ( h M)^{- \frac{n}{4}} f \big( (h M)^{- \frac{1}{2}} y \big) .
\end{equation*}
If $a (y , \eta )$ is a symbol, then
\begin{equation*}
V^{-1} \Op_{h} ( a ( y , \eta ) ) V = \Op_{\frac{1}{M}} \big( a \big( (h M)^{\frac{1}{2}} Y , (h M)^{\frac{1}{2}} H \big) \big) .
\end{equation*}
If possible, we will identify in the following an operator with its conjugation by $V$. As in \cite[(4.24)]{BoFuRaZe07_01}, we define the class of symbols $a \in \widetilde{{\mathcal S}}_{\frac{1}{M}} ( m  )$, for an order function $m (Y,H)$, by
\begin{equation*}
\big\vert \partial_{x}^{\alpha} \partial_{H}^{\beta} a ( Y ,H) \big\vert \lesssim \< Y \>^{- \frac{\vert \alpha \vert}{2}} \< H \>^{- \frac{\vert \beta \vert}{2}} m (Y,H).
\end{equation*}
We refer to the appendix of \cite{BoFuRaZe07_01} for the pseudodifferential calculus in $\widetilde{{\mathcal S}}_{\frac{1}{M}}$. From \cite[(4.23)]{BoFuRaZe07_01}, we have that $\widehat{p} \in \widetilde{{\mathcal S}}_{\frac{1}{M}} ( h M \< ( Y, H ) \>^{2} )$. Since $\varphi \in C^{\infty}_{0} (T^{*} \R^{n} )$, we also have $\varphi (Y ,H) \in \widetilde{{\mathcal S}}_{\frac{1}{M}} ( \< ( Y, H ) \>^{- \infty} )$. Then, the pseudodifferential calculus in $\widetilde{{\mathcal S}}_{\frac{1}{M}}$ implies
\begin{equation} \label{e34}
\widetilde{K} \Op ( \widetilde{p}_{\theta} ) = \CO ( h M ) .
\end{equation}
The same way, \cite[Equation (4.38)]{BoFuRaZe07_01} gives $i h t_{1} \{ \widehat{g}_{1} , \widehat{p}_{0} \} \in \widetilde{{\mathcal S}}_{\frac{1}{M}} ( h^{\frac{3}{2}} M^{\frac{1}{2}} \vert \ln h \vert \< ( Y, H ) \> )$. So,
\begin{equation} \label{e35}
\widetilde{K} U^{-1} \Op ( i h t_{1} \{ \widehat{g}_{1} , \widehat{p}_{0} \} ) U = U^{-1} \widehat{K} \Op ( i h t_{1} \{ \widehat{g}_{1} , \widehat{p}_{0} \} ) U + \CO (h^{\infty} ) = \CO ( h^{\frac{3}{2}} M^{\frac{1}{2}} \vert \ln h \vert ) .
\end{equation}
Working in $S_{h}^{1/2}$, we get
\begin{equation*}
\widetilde{K} U^{-1} \Op ( i h t_{2} \{ \widehat{g}_{2} , \widehat{p}_{0} \} ) U = U^{-1} \widehat{K} \Op ( i h t_{2} \{ \widehat{g}_{2} , \widehat{p}_{0} \} \widehat{\omega}_{2} ) U + \CO (h^{\infty} ).
\end{equation*}
Since $\widehat{\omega}_{2} \prec \widehat{\phi}_{2}$, \cite[Equation (4.48)]{BoFuRaZe07_01} yields that $i h t_{2} \{ \widehat{g}_{2} , \widehat{p}_{0} \} \widehat{\omega}_{2} \in S_{\frac{1}{M}}^{0} ( h)$. Using Calder\`{o}n--Vaillancourt's theorem for this operator, we finally obtain
\begin{equation} \label{e38}
\widetilde{K} U^{-1} \Op ( i h t_{2} \{ \widehat{g}_{2} , \widehat{p}_{0} \} ) U = \CO ( h ).
\end{equation}
The lemma follows from \eqref{e33}, the choice of $M$ in Lemma \ref{e39} and the estimates \eqref{e42}, \eqref{e34}, \eqref{e35} and \eqref{e38}.
\end{proof}

\Subsection{Proof of Theorem \ref{aaa}}

We first prove that \eqref{g3} holds for 
\begin{equation*}
z \in [ E_{0} - A h , E_{0} + A h] + i [- C_{0} h , C_{0} h] .
\end{equation*}
Here, $A>0$ is any fixed constant. We used a method due to Tang and Zworski \cite{TaZw98_01}. For $z \in [ E_{0} - 2 A h , E_{0} + 2 A h] + i [- 2 C_{0} h , 2 C_{0} h]$, the quantity $M$ can always be replaced by $\mu \gg 1$ in Lemma \ref{e39} (see \eqref{k10}--\eqref{e31}). Then, $z \mapsto Q_{z}$ is holomorphic in this set and $\Vert \widetilde{K} \Vert_{\rm tr} = \CO (1)$. As usual (see Section 4 of \cite{BoFuRaZe07_01} for instance), we can find an operator $K$ such that $\Vert K \Vert \lesssim 1$, $\rank K = \CO (1)$ and such that \eqref{e25} holds with $\widetilde{K}$ replaced by $K$. Furthermore, thanks to Remark~\ref{e40}, the resonances coincide with the poles of $Q_{z}^{-1}$ (with the same multiplicity). Mimicking the proof of Proposition 4.2 of \cite{BoFuRaZe07_01} or Lemma 6.5 of \cite{BoMi04_01} (which are adaptations of Lemma 1 of \cite{TaZw98_01}), the estimates \eqref{e24} and \eqref{e25} imply
\begin{equation*}
\big\Vert Q_{z}^{-1} \big\Vert \lesssim h^{-K_{1}} \prod_{z_{\alpha} \in \res (P) \cap D ( E_{0} , 2 C_{0} h )} \vert z - z_{\alpha} \vert^{-1} ,
\end{equation*}
for some $K_{1} > 0$ and any $z \in [ E_{0} - A h , E_{0} + A h] + i [- C_{0} h , C_{0} h]$. On the other hand, Remark~\ref{e40} gives
\begin{equation*}
\big\Vert ( P_{\theta} - z)^{-1} \big\Vert \lesssim h^{- K_{2}} \big\Vert Q_{z}^{-1} \big\Vert ,
\end{equation*}
for some $K_{2} > 0$. This proves \eqref{g3} for $z \in [ E_{0} - A h , E_{0} + A h] + i [- C_{0} h , C_{0} h]$.

Thanks to Theorem~\ref{e17} which describes all the resonances in any neighborhood of size $h$ of $E_{0}$, it remains to prove that $P$ has no resonance in
\begin{equation} \label{e43}
\big( [E_{0} - \varepsilon , E_{0} + \varepsilon ] \setminus [ E_{0} - A h , E_{0} + A h] \big) + i [- C_{0} h , C_{0} h] ,
\end{equation}
for one $A >0$ and that the resolvent satisfies in this region an upper bound polynomial with respect to $h^{-1}$. In particular, we can assume that $\vert z - E_{0} \vert \geq h$. Using Lemma~\ref{e39}, Lemma~\ref{e37} and $\Vert \widetilde{K} Q_{z} u \Vert \lesssim \Vert Q_{z} u \Vert$, we get
\begin{align*}
\Vert Q_{z} u \Vert & \geq \delta h \Vert (P +i) u \Vert - h \big\Vert \widetilde{K} u \big\Vert ,  \\
\Vert Q_{z} u \Vert &\geq \delta \vert z - E_{0} \vert \big\Vert \widetilde{K} u \big\Vert + \CO ( h^{\frac{1}{2}} \vert z -E_{0} \vert^{\frac{1}{2}} ) \Vert u \Vert ,
\end{align*}
for some $\delta > 0$. Then, summing the first identity with $h \delta^{-1} \vert z - E_{0} \vert^{-1}$ times the second one, we obtain
\begin{equation*}
\Vert Q_{z} u \Vert \gtrsim h \Vert (P +i) u \Vert + \CO ( h^{\frac{3}{2}} \vert z -E_{0} \vert^{- \frac{1}{2}} ) \Vert u \Vert ,
\end{equation*}
since $h \delta^{-1} \vert z - E_{0} \vert^{-1} \lesssim 1$. If now we assume that $\vert z - E_{0} \vert \geq A h$, we get
\begin{equation*}
\Vert Q_{z} u \Vert \gtrsim h \Vert (P +i) u \Vert + \CO ( h A^{- \frac{1}{2}} ) \Vert u \Vert \gtrsim h \Vert (P +i) u \Vert ,
\end{equation*}
for $A$ large enough. Thanks to Remark~\ref{e40}, this implies that $P$ has no resonance in the region given in \eqref{e43} and that \eqref{g3} holds in this set.

\section{Spectral projection}
\label{a1}

The purpose of this part is to give the asymptotic expansion of the generalized spectral projection $\Pi_{z_{\alpha}}$ associated to an isolated resonance $z_{\alpha}$ in some $D ( E_{0} , C h)$. We recall that $\Pi_{z_{\alpha}}$ is the operator from $L^{2}_{\rm comp} ( \R^{n} )$ to $L^{2}_{\rm loc} ( \R^{n} )$ defined by
\begin{equation*}
\Pi_{z_{\alpha}} = - \frac1{2i \pi} \oint_{\gamma} ( P -z)^{-1} d z ,
\end{equation*}
where $\gamma$ is a simple loop in the complex plane, oriented counterclockwise, such that $z_{\alpha}$ is the only resonance in the bounded domain delimited by $\gamma$.

\begin{theorem}[Asymptotic expansion for the spectral projection]\sl \label{a2}
Assume \ref{A1}--\ref{A3}. Let $\alpha \in \N^{n}$ be such that $z_{\alpha}^{0}$ is simple. Then, as operators from $L^{2}_{\rm comp} ( \R^{n} )$ to $L^{2}_{\rm loc} ( \R^{n} )$,
\begin{equation} \label{a23}
\Pi_{z_{\alpha}} = c ( \, \cdot \, , \overline{f} ) f ,
\end{equation}
where
\begin{equation} \label{a24}
c (h) = h^{- \vert \alpha \vert - \frac{n}{2}} \frac{e^{- i \frac{\pi}{2} ( \vert \alpha \vert + \frac{n}{2} )}}{\alpha ! (2 \pi)^{\frac{n}{2}}} \prod_{j=1}^{n} \lambda_{j}^{\alpha_{j} + \frac{1}{2}} ,
\end{equation}
and the function $f (x,h)$ satisfies the following properties:

$i)$ It is locally uniformly in $L^{2} ( \R^{n} )$: for all $\varphi \in C^{\infty}_{0} ( \R^{n})$,
\begin{equation*}
\Vert \varphi f \Vert_{L^{2} ( \R^{n})} \lesssim 1.
\end{equation*}

$ii)$ It satisfies the Schr\"{o}dinger equation:
\begin{equation*}
( P -z_{\alpha} ) f = 0 .
\end{equation*}

$iii)$ It is outgoing: there exists $R >0$ such that
\begin{equation*}
f =0 \text{ microlocally near each } (x,\xi) \text{ with } \vert x \vert >R, \text{ } \cos (x, \xi ) < -1/2 .
\end{equation*}

$iv)$ Finally, locally near $(0,0)$, we have
\begin{equation*}
f = d (x, h) e^{i \varphi_{+} (x) /h} ,
\end{equation*}
where $d (x,h) \in S_{h}^{0} (1)$ is a classical symbol satisfying
\begin{equation*}
d(x,h) \sim \sum_{j =0}^{+ \infty} d_{j} (x) h^{j} \quad \text{ and } \quad d_{0} (x) = x^{\alpha} + \CO ( x^{\vert \alpha \vert +1}) .
\end{equation*}
\end{theorem}

We prove this result the following way. Using \cite{BoFuRaZe07_01}, we compute $(P -z)^{-1} v$ for some well prepared WKB function $v$ and $z$ on a loop around the resonance $z_{\alpha}$. Integrating with respect to $z$, we get $\Pi_{z_{\alpha}} v$ and thus the resonant state $f$. To finish the proof, we obtain the constant $c$ computing $( v , \overline{f} )$ by a stationary phase argument.

\begin{remark}\sl
$i)$ Since $f$ is not necessarily in ${\mathcal S} ' ( \R^{n} )$, saying ``$f =0$ microlocally near $\rho_{0}$'' means that there exists $\phi \in C^{\infty}_{0} ( \R^{2n})$ with $\phi ( \rho_{0} ) \neq 0$ such that, for every $\chi \in C^{\infty}_{0} (\R^{n})$, $\Op ( \phi ) ( \chi f ) = \CO ( h^{\infty} )$ in $L^{2}( \R^{n})$.

$ii)$ The properties $i)$--$iv)$ of Theorem \ref{a2} characterize uniquely the resonant state $f (x,h)$ modulo $\CO (h^{\infty})$. In particular, the usual propagation of singularities implies that this function is a classical Lagrangian distribution of order $0$ with Lagrangian manifold $\Lambda_{+}$.
\end{remark}

For the punctual well in the island situation, the generalized spectral projection has been computed by Helffer and Sj\"{o}strand \cite{HeSj86_01}. In particular, they have proved that this operator is almost orthogonal. Indeed, if the resonance $z$ is isolated and the cut-off $\chi \in C^{\infty}_{0} ( \R^{n} )$ is equal to $1$ near the well, then $\chi \Pi_{z} \chi$ is exponentially close to the spectral projection associated to the Dirichlet problem in the well and $\Vert \chi \Pi_{z} \chi \Vert = 1 + \CO (e^{- \delta /h})$ for some $\delta > 0$. The situation is very different in the present setting since, for $\chi \neq 0$, $\Vert \chi \Pi_{z_{\alpha}} \chi \Vert$ is of order $h^{- \vert \alpha \vert - \frac{n}{2}}$.

From the previous discussion, the polynomial upper bound on the resolvent proved in Theorem~\ref{aaa} occurs effectively. More precisely, in every disc $D (z_{\alpha} , \varepsilon h )$, with $\varepsilon >0$, the cut-off resolvent can not be bounded by anything smaller than $\frac{h^{- C_{\alpha}}}{\vert z - z_{\alpha} \vert}$ for some $C_{\alpha}>0$. Moreover, since $C_{\alpha} \geq \vert \alpha \vert + \frac{n}{2}$, this constant can not be taken uniformly with respect to $z_{\alpha}$.

One may perhaps prove Theorem \ref{a2} with other methods than the one we use here. In the one dimensional case, the resolvent can be written in term of a basis of solutions of $(P-z) u =0$ and of their Wronskian. Thus, it must be possible to use the results of \cite{Ra96_01} in which the scattering amplitude, which can be expressed through the Wronskians of the Jost solutions, has been computed. In any dimension, another approach is perhaps also possible. One may first try to calculate the resonant state $f$ with various methods (using, for example, the works of Briet, Combes and Duclos \cite{BrCoDu87_01}, Sj\"{o}strand \cite{Sj87_01} or Hassell, Melrose and Vasy \cite{HaMeVa08_01}). It then remains to calculate the constant $c$. This question is equivalent to the calculation of the scalar product $( f , \overline{f} ) = \int f^{2}$. If we neglect the problems of integration at infinity, this calculation is reduced to a problem of stationary phase at point $0$. But, since $f^{2}$ vanishes to order $2 \vert \alpha \vert$, the knowledge of $d_{0}$ is not enough and we must explicitly know the $\vert \alpha \vert$ first terms in the expansion of $f$ in powers of $h$. In this computation, the situation becomes, in a sense, similar to that of the eigenvectors of the harmonic oscillator for which the ``good variable'' is $\frac{x}{\sqrt{h}}$. However, this is not the case in Theorem \ref{a2} since the factor $e^{i \varphi_{+} (x) /h}$ in $f$ has modulus $1$.

It may be possible to obtain some results when $z_{\alpha}^{0}$ is not simple. In that case, various situations may occur: several resonances can be very close to each other, the resonances can have a non-trivial multiplicity and they can be multiple poles of the resolvent. We refer to \cite[Section 4]{Sj87_01} where such phenomena are shown. In the remainder of this discussion, we consider the simplest case where a double resonance can appear. We assume that $\lambda_{1} = \lambda_{2} < \lambda_{3}$ and that $\widetilde{z} = b_{h} ( z_{(1,0, \ldots)} ) = b_{h} ( z_{(0,1,0, \ldots )})$ is a double resonance. Then, near $\widetilde{z}$, the resolvent can be written
\begin{equation*}
( z - P)^{-1} = \frac{\Pi_{2}}{(z - \widetilde{z})^{2}} + \frac{\Pi_{1}}{z - \widetilde{z}} + H (z),
\end{equation*}
where $H$ is holomorphic near $\widetilde{z}$. In that case, $\rank \Pi_{2} \leq 1$ and $\rank \Pi_{1} = 2$. It seems possible to calculate $\Pi_{1}$ with a proof similar to that of Theorem \ref{a2}. Using Proposition~\ref{d10}, we can construct two initial data $v_{1} , v_{2}$ such that the microsupport of $v_{j}$ and $\Lambda_{-}$ intersect along a Hamiltonian curve which goes to $0$ along the $j$-th vector basis. Then, computing the residue of $( z-P)^{-1} v_{j}$, we obtain that $\Pi_{1} ( v_{j} + ( P - \widetilde{z} ) \partial_{z} v_{j} )$ is of the form $f_{j} = x_{j} e^{i \varphi_{+} (x) /h}$ modulo a constant. In the following, we can neglect $( P - \widetilde{z} ) \partial_{z} v_{j}$ as it gives lower order terms. Since $f_{1}$ and $f_{2}$ can not be collinear, $\{ f_{1} , f_{2} \}$ (resp. $\{ \overline{f_{1}} , \overline{f_{2}} \}$) forms a basis of $\im \Pi_{1}$ (resp. $\im \Pi_{1}^{*}$). To finish the computation of $\Pi_{1}$, it is sufficient to calculate $( v_{j} , \overline{f_{k}} )$. The scalar products $( v_{j} , \overline{f_{j}} )$ can be calculated as in the proof of Theorem \ref{a2}. But, according to the choice of the $v_{j}$'s and to the form of the $f_{k}$'s, $( v_{j} , \overline{f_{k}} )$ appears to be smaller when $j \neq k$. Eventually, in the $\{ f_{1} , f_{2} \}$  and $\{ \overline{f_{1}} , \overline{f_{2}} \}$ bases, the operator $\Pi_{1}$ seems to be a $2 \times 2$-matrix whose diagonal coefficients are given by \eqref{a24} at the first order and whose off-diagonal coefficients are of lower order. One can probably also say something about $\Pi_{2}$. But, one may need to calculate several lower order terms in the semiclassical expansions (for the resonance for example). This operator seems to have a smaller norm.

\Subsection{Construction of ``test functions''}
\label{a34}

To prove the theorem, it is enough to show that
\begin{equation}
\chi \Pi_{z_{\alpha}} \chi = c ( \, \cdot \, , \overline{\chi f} ) \chi f ,
\end{equation}
for $\chi \in C^{\infty}_{0} ( \R^{n} )$. Let $\Pi_{z_{\alpha} , \theta}$ be the spectral projection of $P_{\theta}$ at the resonance $z_{\alpha}$. It is the operator on $L^{2} ( \R^{n} )$ defined by
\begin{equation} \label{a14}
\Pi_{z_{\alpha} , \theta} = - \frac1{2i \pi} \oint_{\gamma} ( P_{\theta} -z)^{-1} d z ,
\end{equation}
We now assume that the distortion occurs outside of the support of $\chi$. In particular, $\chi \Pi_{z_{\alpha}} \chi = \chi \Pi_{z_{\alpha} , \theta} \chi$. Let $J$ be the anti-linear operator on $L^{2} ( \R^{n})$ defined by
\begin{equation*}
\begin{aligned}
J: \\
{}^{}
\end{aligned}
\left\{
\begin{gathered}
L^{2} ( \R^{n} )  \\
u
\end{gathered}
\begin{gathered}
\longrightarrow  \\
{}^{}
\end{gathered}
\begin{gathered}
L^{2} ( \R^{n} )  \\
\overline{u} .
\end{gathered} \right.
\end{equation*}
Since $P$ is a Schr\"{o}dinger operator with a real potential, $J P = P J$ and a direct calculation gives $( P_{\theta} -z)^{-1} = J \big( ( P_{\theta} -z)^{-1} \big)^{*} J$. Thus, $\Pi_{z_{\alpha} , \theta}$ can be written $\Pi_{z_{\alpha} , \theta} = ( \, \cdot \, , \overline{g_{\theta}} ) g_{\theta}$ with $g_{\theta} \in L^{2} ( \R^{n})$. The same way, $\Pi_{z_{\alpha}} = ( \, \cdot \, , \overline{g} ) g$ for some $g \in L^{2}_{\rm loc} ( \R^{n} )$. Moreover, from \cite{SjZw91_01}, we can always assume that $g_{\theta} = U_{\theta} g$. In particular, $\chi g_{\theta} = \chi g$.

Since $z_{\alpha}^{0}$ is simple, for all $j \in \{1 \} \cup \supp \alpha$ (where $\supp \alpha = \{ j \in \N ; \ \alpha_{j} \neq 0 \}$), $\lambda_{j} = \lambda \cdot \beta$ with $\beta \in \N^{n}$ implies $\vert \beta \vert =1$. Then, from Lemma \ref{d3} and Proposition \ref{d10}, there exists a Hamiltonian curve $\gamma^{-} = (x(t) , \xi (t) ) \subset \Lambda_{-}$ such that, for all $j \in \{1 \} \cup \supp \alpha$, we have $\gamma_{\lambda_{j}}^{-} = \gamma_{\lambda_{j} ,0}^{-} \neq 0$.

We now construct the ``test functions'', supported microlocally near the ``test curve'' $\gamma^{-}$, on which we will evaluate the spectral projection. Let $u (x , z , h)$ be a function defined in a vicinity of $0$ but not at $0$. We assume that $u$ is a WKB solution of $(P -z) u =0$. More precisely, near the $x$-projection of $\gamma^{-} \setminus \{ 0 \}$, we have
\begin{equation} \label{a26}
u (x, z,h) = b (x , z,h) e^{i \psi ( x) /h} .
\end{equation}
Here $\psi$ is a $C^{\infty}$ function solving the eikonal equation $\vert \nabla \psi \vert^{2} + V (x) = E_{0}$. We assume that $\Lambda_{\psi} = \{ ( x, \nabla \psi (x) ) \}$ intersects transversely $\Lambda_{-}$ along $\gamma^{-}$. Note that the construction of such a phase, whose associated Lagrangian manifold projects nicely on the $x$-space in a vicinity of $\gamma^{-}$, can always be done thanks to \cite[Proposition C.1]{AlBoRa08_01}. The symbol $b (x , z ,h)$ is classical: for all $N\in \N$,
\begin{equation*}
b (x , z,h) = \sum_{j=0}^{N} b_{j} (x ,z) h^{j} + \CO (h^{N+1}) ,
\end{equation*}
uniformly for $z \in D ( E_{0} , C_{0} h)$. Moreover, $b$ and the $b_{j}$'s are $C^{\infty}$ with respect to $x$ and analytic with respect to $z \in D ( E_{0} , C_{0} h)$. Finally, we assume that $u$ satisfies
\begin{equation*}
( P -z ) u = \CO ( h^{\infty} ),
\end{equation*}
and $b_{0} ( x,z) \neq 0$ near the $x$-projection of $\gamma^{-}$. For that, it is enough to solve the usual transport equations. Finally, we suppose that $u=0$ outside a neighborhood of the spacial projection of $\gamma^{-}$. Then, we set
\begin{equation} \label{a25}
v = [P , \tau ] u ,
\end{equation}
where $\tau \in C^{\infty}_{0} ( \R^{n} )$ with $\supp \tau$ close to $0$ and $\tau =1$ near $0$. We consider
\begin{equation} \label{a12}
w = ( P_{\theta} - z)^{-1} v .
\end{equation}
In all the proof of Theorem \ref{a2}, we will work with $z$ in a ring ${\mathcal R}_{h} = D ( z_{\alpha}^{0} , C_{2} h) \setminus D ( z_{\alpha}^{0} , C_{1} h)$ such that $z_{\alpha}^{0}$ is the unique element of $\Res_{0} (P)$ in $D ( z_{\alpha}^{0} , C_{2} h)$. Note that Theorem \ref{aaa} implies that $\Vert w \Vert_{H^{2} ( \R^{n} )} \lesssim h^{-C}$ uniformly for $z \in {\mathcal R}_{h}$, for some $C >0$.

\Subsection{Calculation of $w$ before the critical point}

We begin the proof by showing that $w$ is $0$ in the incoming region. More precisely, we have

\begin{lemma}\sl \label{a4}
Let $\rho \in \R^{2n}$ be such that $\rho \notin \Lambda_{+}$ and $\exp ( ]- \infty , 0 ] H_{p}) ( \rho )$ does not meet the microsupport of $v$. Then, $w =0$ microlocally near $\rho$, uniformly in $z \in {\mathcal R}_{h}$.
\end{lemma}

\begin{proof}
This lemma can be proved as Theorem 2 of \cite{BoMi04_01}. First, assume $\rho \notin p^{-1} ( E_{0})$. Using the elliptic equation $(P_{\theta} -z) w = v$, the norm estimates $\Vert v \Vert , \Vert w \Vert \lesssim h^{-C}$ and the condition $\rho \notin \MS (v)$, the standard pseudodifferential calculus implies that $\rho \notin \MS (w)$. More precisely, for all $f \in C^{\infty}_{0} ( \R )$ with $f = 1$ near $E_{0}$, we have
\begin{equation} \label{a5}
(1 - f (P) ) w = \CO ( h^{\infty} ) ,
\end{equation}
uniformly in $z \in {\mathcal R}_{h}$.

Assume now that $\rho \in p^{-1} (E_{0} )$. From the hypotheses, the half-curve $\exp (t H_{p}) ( \rho )$, $t \leq 0$, does not meet $\MS ( v)$ and goes to $\infty$ as $t \to - \infty$. Then, one can find a symbol $\omega \in S_{h}^{0} (1)$ such that $\omega =1$ near $\rho$, $H_{p} \omega \leq 0$, $\exp ( ]- \infty , 0 ] H_{p}) ( \supp \omega  )$ does not meet $\MS (v)$ and $\exp (-T H_{p} ) ( \supp \omega ) \subset \Gamma^{-} (R , d , \sigma)$ for some $T, R \gg 1$, $d >0$ and $\sigma <0$. Here, $\Gamma^{-} (R , d , \sigma) = \{ (x , \xi ) \in T^{*} \R^{n} ; \ \vert x \vert > R, \ d^{-1} < \vert \xi \vert < d \text{ and } \cos (x , \xi ) \leq \sigma \}$. Then, mimicking the proof of \cite[Theorem 2]{BoMi04_01}, we get $\Op ( \omega ) w = \CO ( h^{\infty} )$, uniformly in $z \in {\mathcal R}_{h}$. The unique difference with its proof is that the $0$ in the left hand side of \cite[(3.4)]{BoMi04_01} is replaced by $\CO (h^{\infty})$ (here, we use that $\supp \omega \cap \MS (v) = \emptyset$).
\end{proof}

We will now calculate $w$ on $\Lambda_{-}$ near $0$. First, using $\MS (v) \cap \Lambda_{-} \subset \gamma^{-}$, the previous lemma implies the following consequence.

\begin{remark}\sl \label{a7}
We have $w = 0$ microlocally near each point of $\Lambda_{-} \setminus \gamma^{-}$.
\end{remark}

On the other hand, near $\gamma^{-}$, we have the following lemma. Note that the results of this lemma and of Remark \ref{a7} are uniform for $z \in {\mathcal R}_{h}$.

\begin{lemma}\sl \label{a6}
Let $\rho \in \gamma^{-}$ be a point close enough to $0$. Then, $w = u$ microlocally near $\rho$.
\end{lemma}

\begin{figure}
\begin{center}
\begin{picture}(0,0)%
\includegraphics{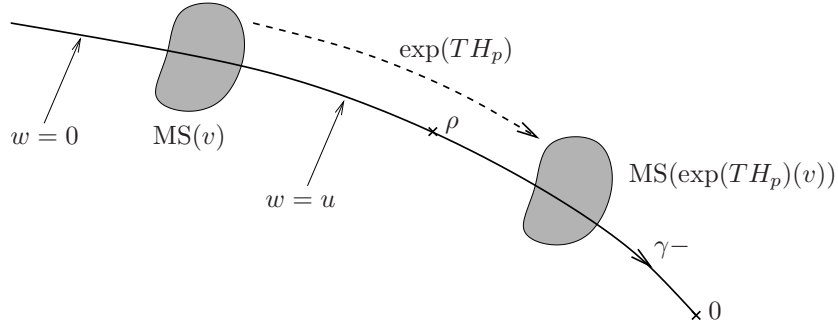}%
\end{picture}%
\setlength{\unitlength}{1381sp}%
\begingroup\makeatletter\ifx\SetFigFont\undefined%
\gdef\SetFigFont#1#2#3#4#5{%
  \reset@font\fontsize{#1}{#2pt}%
  \fontfamily{#3}\fontseries{#4}\fontshape{#5}%
  \selectfont}%
\fi\endgroup%
\begin{picture}(12573,5760)(-1457,-5666)
\put(-1424,-2461){\makebox(0,0)[lb]{\smash{{\SetFigFont{10}{12.0}{\rmdefault}{\mddefault}{\updefault}$w=0$}}}}
\put(1126,-2461){\makebox(0,0)[lb]{\smash{{\SetFigFont{10}{12.0}{\rmdefault}{\mddefault}{\updefault}$\MS (v)$}}}}
\put(5551,-886){\makebox(0,0)[lb]{\smash{{\SetFigFont{10}{12.0}{\rmdefault}{\mddefault}{\updefault}$\exp ( T H_{p} )$}}}}
\put(3151,-3586){\makebox(0,0)[lb]{\smash{{\SetFigFont{10}{12.0}{\rmdefault}{\mddefault}{\updefault}$w=u$}}}}
\put(6376,-2161){\makebox(0,0)[lb]{\smash{{\SetFigFont{10}{12.0}{\rmdefault}{\mddefault}{\updefault}$\rho$}}}}
\put(9676,-3136){\makebox(0,0)[lb]{\smash{{\SetFigFont{10}{12.0}{\rmdefault}{\mddefault}{\updefault}${\rm MS} (\exp ( T H_{p} ) ( v))$}}}}
\put(10126,-4411){\makebox(0,0)[lb]{\smash{{\SetFigFont{10}{12.0}{\rmdefault}{\mddefault}{\updefault}$\gamma{-}$}}}}
\put(11101,-5611){\makebox(0,0)[lb]{\smash{{\SetFigFont{10}{12.0}{\rmdefault}{\mddefault}{\updefault}$0$}}}}
\end{picture}%
\end{center}
\caption{The geometrical setting of Lemma \ref{a6}.}
\label{f1}
\end{figure}

\begin{proof}
We define
\begin{equation} \label{a3}
\widetilde{w} = \frac{i}{h} \int_{0}^{T} e^{-i t ( P-z) /h} [P , \tau ] u \, d t ,
\end{equation}
where $T>0$ is chosen such that $\exp ( T H_{p} ) ( \MS (v) ) \cap \exp ( ] - \infty , 0 ] H_{p} ) ( \rho ) = \emptyset$ (see Figure~\ref{f1}). Then, microlocally near each point of $\gamma^{-}$, we have
\begin{align*}
( P_{\theta} -z) \widetilde{w} &= ( P -z) \widetilde{w}    \\
&= [P , \tau ] u - e^{-i T (P -z)/h} [P , \tau ] u .
\end{align*}
For the first equality, we have used that $P = P_{\theta}$ near the spacial projection of $\MS ( \widetilde{w} ) \cap \gamma^{-} \subset \exp ( [ 0 , + \infty [ H_{p} ) ( \MS (v))$. Thus, microlocally near $\gamma^{-}$, we have
\begin{equation*}
( P_{\theta} -z) ( w - \widetilde{w} ) = e^{-i T (P -z)/h} [P , \tau ] u .
\end{equation*}
In particular, the choice of $T$ and the Egorov theorem imply $( P_{\theta} -z) ( w - \widetilde{w} ) =0$ microlocally near $\exp ( ] - \infty , 0 ] H_{p} ) ( \rho )$. On the other hand, combining Lemma \ref{a4} (for $w$) and the Egorov theorem (for $\widetilde{w}$), we obtain $w - \widetilde{w} = 0$ microlocally near $\exp ( -S H_{p} ) ( \rho )$, for all $S$ large enough. Using moreover that $\Vert w - \widetilde{w} \Vert \leq h^{-C}$, the propagation of singularities implies that
\begin{equation}
w = \widetilde{w} \text{ microlocally near } \rho .
\end{equation}

Then, microlocally near $\rho$, we have
\begin{align}
\widetilde{w} &= \frac{i}{h} \int_{0}^{T} e^{-i t ( P-z) /h} [P , \tau ] u \, d t = \frac{i}{h} \int_{0}^{T} e^{-i t ( P-z) /h} \big( ( P-z) \tau u  - \tau ( P -z) u \big) d t  \nonumber  \\
&= \frac{i}{h} \int_{0}^{T} e^{-i t ( P-z) /h} ( P-z) \tau u \, d t = - e^{- i T (P-z) /h} \tau u + \tau u = u ,
\end{align}
which proves the lemma.
\end{proof}

In fact, one can prove more directly Lemma \ref{a4} and Lemma \ref{a6} by applying the proof of Theorem 2 of \cite{BoMi04_01} to the function $w - \widetilde{w}$.

\Subsection{Representation of $w$ at the critical point}

We will use the variable $\sigma = ( z - E_{0} )/h$, the notation $\sigma_{\alpha}^{0} = ( z_{\alpha}^{0} - E_{0} )/h$ and the set ${\mathcal R} = D ( \sigma_{\alpha}^{0} , C_{2} ) \setminus D ( \sigma_{\alpha}^{0} , C_{1} )$. Note that $\sigma_{\alpha}^{0}$ and ${\mathcal R}$ does not depend on $h$ and ${\mathcal R}_{h} = E_{0} + h {\mathcal R}$. Since $\tau =1$ near $0$, we have
\begin{equation} \label{a8}
(P-z) w = (P_{\theta} -z) w = [P , \tau] u = 0 ,
\end{equation}
in a neighborhood of $0$. On the other hand, let $\rho \in \Lambda_{-} \cap \{ \vert x \vert = \delta \}$ with $\delta >0$ small enough. From Remark \ref{a7} and Lemma \ref{a6}, we have
\begin{equation}
w = \left\{ \begin{aligned}
&b e^{i \psi /h} &&\text{if } \rho \in \gamma^{-} \\
&0 &&\text{if } \rho \notin \gamma^{-} ,
\end{aligned} \right.
\end{equation}
microlocally near $\rho$. Moreover $\Vert w \Vert \lesssim h^{-C}$. Then, we are in position to apply Theorem 2.1 and Theorem 2.5 of \cite{BoFuRaZe07_01} which give a representation of $w$ microlocally near $(0,0)$. More precisely, Theorem 5.1 of \cite{BoFuRaZe07_01} states that, microlocally near $(0,0)$,
\begin{equation} \label{a15}
w = \frac{1}{(2 \pi h)^{\frac{1}{2}}} e^{i \varphi_{+} (x) /h} e^{i \psi (0) /h} A_{-} (x , \sigma ,h) + \frac{1}{(2 \pi h)^{\frac{1}{2}}} \int_{-1}^{+ \infty} e^{i \varphi (t ,x) /h} A_{+} (t ,x , \sigma ,h) \, d t .
\end{equation}
Concerning the symbol $A_{+}$, we will only use that $\sigma \mapsto A_{+}$ is a holomorphic function of $\sigma \in D (0 , C_{0})$ which decays uniformly exponentially in $t$ (see \cite[Proposition 5.11]{BoFuRaZe07_01}). The constant $\psi (0)$ is defined by
\begin{equation} \label{a29}
\psi (0) := \lim_{t \to + \infty} \psi (x (t)) = \psi (x (s)) - \varphi_{-} (x (s)) ,
\end{equation}
for all $s \geq 0$.

The symbol $A_{-} (x , \sigma ,h) \in S_{h}^{0} (h^{-C})$, holomorphic for $\sigma \in {\mathcal R}$, is constructed the following way. There exists an expandible symbol $a (t,x, \sigma ,h) \in S_{h}^{0} (1)$ of the form
\begin{equation*}
a (t,x, \sigma ,h) \sim \sum_{j=0}^{+ \infty} a_{j} (t,x, \sigma ) h^{j} ,
\end{equation*}
where the $a_{j}$'s satisfy
\begin{equation*}
a_{j} (t,x, \sigma ) \sim \sum_{k=0}^{+ \infty} a_{j , \mu_{k}} (t,x, \sigma ) e^{-(S + \mu_{k}) t} \qquad \text{ and } \qquad a_{j , \mu_{k}} (t,x, \sigma ) = \sum_{\ell=0}^{M_{j , \mu_{k}}} a_{j , \mu_{k} , \ell} (x, \sigma ) t^{\ell} .
\end{equation*}
We refer to Helffer and Sj\"{o}strand \cite{HeSj85_01} for the definition of expandible functions. Here, $S$ is defined by
\begin{equation*}
S  = S( \sigma ) := \sum_{j=1}^{n} \frac{\lambda_{j}}{2} - i \sigma .
\end{equation*}
The symbols $a_{j}, a_{j , \mu_{k}} , a_{j , \mu_{k} ,\ell}$ are holomorphic for $\sigma \in D (0 , C_{0} )$. Moreover, as in \cite[(6.26)]{BoFuRaZe07_01}, $a_{0,0}$ does not depend on $t$ (and $\sigma$) and
\begin{equation} \label{a9}
a_{0,0}(0) = \vert g^{-}_{\lambda_{1}} \vert \lambda_{1}^{\frac{3}{2}} e^{- i \frac{\pi}{4}} e^{- \int_{0}^{+ \infty} \Delta \psi (x(s)) - ( \sum \lambda_{j} /2 - \lambda_{1}) \, d s} b_{0} (x(0)) ,
\end{equation}
with $g^{-}_{\lambda_{j}} = \pi_{x} ( \gamma_{\lambda_{j}}^{-} )$, $\pi_{x}$ being the spatial projection. Let $\varphi_{\star} (t,x) = \varphi (t,x) - ( \varphi_{+} (x) + \psi (0) )$ be the expandible function
\begin{equation*}
\varphi_{\star} (t,x) \sim \sum_{k =1}^{+ \infty} \varphi_{\mu_{k}} (t,x) e^{- \mu_{k} t} \qquad \text{ and } \qquad \varphi_{\mu_{k}} (t,x) = \sum_{\ell =0}^{N_{\mu_{k}}} \varphi_{\mu_{k} , \ell} (x) t^{\ell} ,
\end{equation*}
constructed in \cite[Section 5]{BoFuRaZe07_01}. Recall that $\varphi (t,x)$ satisfies the eikonal equation
\begin{equation} \label{a10}
\partial_{t} \varphi + ( \partial_{x} \varphi )^{2} + V (x) = E_{0} .
\end{equation}
We consider the expandible symbol (see (5.77) of \cite{BoFuRaZe07_01})
\begin{equation} \label{a17}
\widetilde{a} = \sum_{q < Q_{1}} \frac{a}{q !} \left( \frac{i \varphi_{\star}}{h} \right)^{q} \sim \sum_{j = 1-Q_{1}}^{+ \infty} \widetilde{a}_{j} (t,x, \sigma ) h^{j} ,
\end{equation}
for some $Q_{1} \in \N$ fixed large enough,
\begin{equation*}
\widetilde{a}_{j} (t,x, \sigma ) \sim \sum_{k=0}^{+ \infty} a_{j , \mu_{k}} (t,x, \sigma ) e^{-(S + \mu_{k}) t} \qquad \text{and} \qquad \widetilde{a}_{j , \mu_{k}} (t,x, \sigma ) = \sum_{\ell=0}^{\widetilde{M}_{j , \mu_{k}}} \widetilde{a}_{j , \mu_{k} , \ell} (x, \sigma ) t^{\ell} .
\end{equation*}
Then, $A_{-} (x, \sigma ,h)$ is a symbol, holomorphic with respect to $\sigma \in {\mathcal R}$, such that
\begin{equation}\label{a16}
A_{-} (x, \sigma ,h) \sim \sum_{j=1-Q_{1}}^{+ \infty} h^{j} \sum_{k=0}^{K_{1}} \sum_{\ell =0}^{\widetilde{M}_{j , \mu_{k}}}  \frac{\ell !}{( S + \mu_{k})^{\ell +1}} \widetilde{a}_{j, \mu_{k} , \ell} (x, \sigma ) ,
\end{equation}
for some $K_{1} \in \N$ large enough.

In the following, we will need some informations on the $\varphi_{\mu_{k}}$. Let $j \in \{ 1 , \ldots , n\}$ be such that $\alpha_{j} \neq 0$ or $j=1$. Since $z_{\alpha}^{0}$ is simple, $\lambda_{j}$ can not be written as a non-trivial combination of the $\lambda_{k}$'s (i.e. $\lambda_{j} = \lambda \cdot \beta$ implies $\beta_{k} = \delta_{j,k}$). Therefore, calculating the term in $e^{- \lambda_{j} t}$ of \eqref{a10}, we obtain
\begin{equation*}
- \lambda_{j} \varphi_{\lambda_{j}} + \partial_{t} \varphi_{\lambda_{j}} + 2 \partial_{x} \varphi_{+} \cdot \partial_{x} \varphi_{\lambda_{j}} = 0 .
\end{equation*}
Working as in Section 6.1 of \cite{AlBoRa08_01} (see also (5.59) of \cite{BoFuRaZe07_01} for $j=1$), one can prove that $\varphi_ {\lambda_{j}}$ does not depend on $t$ (i.e. $N_{\lambda_{j}} =0$), that
\begin{equation} \label{a11}
2 \partial_{x} \varphi_{+} \cdot \partial_{x} \varphi_{\lambda_{j}} - \lambda_{j}\varphi_ {\lambda_{j}} =0 ,
\end{equation}
and that
\begin{equation} \label{a18}
\varphi_{\lambda_{j}} (x) = - \lambda_{j} g_{\lambda_{j}}^{-} x_{j} + \CO (x^{2} ) .
\end{equation}
Since $g_{\lambda_{j}}^{-}$ is collinear to the $j$-th vector of basis, we also denote this $j$-th component of the vector $g_{\lambda_{j}}^{-}$ by $g_{\lambda_{j}}^{-}$.

\Subsection{Integration with respect to $z$}
\label{k9}

Let $\gamma$ be a fixed simple loop in ${\mathcal R}$ around $0$ oriented counterclockwise and $\gamma_{h} = E_{0} + h \gamma \subset {\mathcal R}_{h}$. We integrate $w$ on the loop $\gamma_{h}$. First, since $z_{\alpha}$ is a simple resonance for $h$ small and since $v$ is a holomorphic function with respect to $z \in D (E_{0} , C_{0} h)$, the equations \eqref{a14} and \eqref{a12} give
\begin{equation} \label{a13}
\Pi_{z_{\alpha} , \theta} v (x, z_{\alpha} ) = - \frac1{2 i \pi} \oint_{\gamma_{h}} w (x,z) \, d z = - \frac{h}{2 i \pi} \oint_{\gamma} w (x, \sigma ) \, d \sigma .
\end{equation}

On the other hand, we can also calculate this quantity microlocally near $(0,0)$ with the help of \eqref{a15}. Since $\sigma \mapsto A_{+} ( \sigma )$ is holomorphic in $D ( 0 , C_{0} )$, the second term in the right hand side of \eqref{a15} gives no contribution to this integral. Moreover, for $\mu_{k} \neq \lambda \cdot \alpha$, the function $(S + \mu_{k})^{-1}$ is holomorphic for $\sigma \in D ( \sigma_{\alpha}^{0} , C_{2} )$. This implies that only the terms of \eqref{a16} with $\mu_{k} = \lambda \cdot \alpha$ give a non-zero contribution to the integral over $\sigma$.

We now look for the terms with $\mu_{k} = \lambda \cdot \alpha$ in \eqref{a17}. Among these terms, the one which gives the higher possible power of $h^{-1}$, is given by $q = \vert \alpha \vert$ and is equal to
\begin{equation*}
\frac{a_{0 ,0} (x)}{\vert \alpha \vert !} \left( \frac{i}{h} \right)^{\vert \alpha \vert} \frac{\vert \alpha \vert !}{\prod_{j=1}^{n} \alpha_{j} !} \prod_{j=1}^{n} \big( \varphi_{\lambda_{j}} (x) \big)^{\alpha_{j}} .
\end{equation*}
Here, we have used the fact that $z_{\alpha}^{0}$ is simple. Note that, since $a_{0,0}$ and $\varphi_{\lambda_{j}}$, with $\alpha_{j} \neq 0$, does not depend on $t$ (see the discussion before \eqref{a11}), this term does not depend on $t$. Then, $A_{-}$ satisfies, as $h \to 0$,
\begin{equation*}
A_{-} (x, \sigma ,h) \sim \sum_{j= 0}^{+ \infty} a^{-}_{j} ( x , \sigma ) h^{- \vert \alpha \vert + j} + {\mathcal H} (x, \sigma ,h),
\end{equation*}
where the $a^{-}_{j}$'s are holomorphic with respect to $\sigma \in {\mathcal R}$ and $C^{\infty}$ with respect to $x$ near $0$. The function $\sigma \mapsto {\mathcal H}$ is holomorphic in $D ( \sigma_{\alpha}^{0} , C_{2} )$. Moreover,
\begin{equation*}
a^{-}_{0} ( x , \sigma ) = \frac{i^{\vert \alpha \vert} a_{0 ,0} (x)}{( \sum \lambda_{j} /2 + \lambda \cdot \alpha - i \sigma ) \alpha !} \prod_{j=1}^{n} \big( \varphi_{\lambda_{j}} (x) \big)^{\alpha_{j}} .
\end{equation*}

Using the previous discussion, together with \eqref{a15} and \eqref{a13}, we obtain that
\begin{equation} \label{a21}
\Pi_{z_{\alpha} , \theta} v (x, z_{\alpha} ) = - \frac1{2i \pi} \oint_{\gamma} w (x,z) \, d z \sim e^{i \varphi_{+} (x) /h} e^{i \psi (0) /h} \sum_{j=0}^{+ \infty} \widehat{a}_{j} (x) h^{\frac{1}{2} - \vert \alpha \vert + j} ,
\end{equation}
microlocally near $( 0, 0)$. Moreover,
\begin{equation} \label{a19}
\widehat{a}_{0} ( x ) = - \frac{i^{\vert \alpha \vert + 1}}{(2 \pi )^{\frac{1}{2}} \alpha !} a_{0 ,0} (x) \prod_{j=1}^{n} \big( \varphi_{\lambda_{j}} (x) \big)^{\alpha_{j}} .
\end{equation}

To be more precise, in the $C^{\infty}$ case, Theorem 2.1 of \cite{BoFuRaZe07_01} gives only uniqueness for $z$ outside of a set $\Gamma (h)$, which is finite uniformly with respect to $h$. Then, to prove \eqref{a13}, we integrate first on a loop $\widetilde{\gamma}_{h} \in {\mathcal R}_{h} \setminus ( \Gamma (h) + D (0 , \varepsilon h) )$ of length of order $h$ and which may depend on $h$ in a non trivial way. But, since the function $w$ is holomorphic in ${\mathcal R}_{h}$, we can deform the contour $\widetilde{\gamma}_{h}$ to $\gamma_{h}$ and thus justify \eqref{a13}.

\Subsection{Construction and properties of $\widetilde{f}$}
\label{a46}

We define the functions $\widetilde{f}$ and $\widetilde{f}_{\theta}$ by
\begin{equation} \label{a20}
\widetilde{f} (x,h) := \widehat{c}^{-1} \Pi_{z_{\alpha}} v (x, z_{\alpha} ) \qquad \text{and} \qquad \widetilde{f}_{\theta} (x,h) := \widehat{c}^{-1} \Pi_{z_{\alpha} , \theta} v (x, z_{\alpha} ) ,
\end{equation}
where, using the notation $( g^{-} )^{\alpha} = \prod_{j=1}^{n} ( g^{-}_{\lambda_{j}} )^{\alpha_{j}}$,
\begin{equation} \label{a22}
\widehat{c} (h) : = - \frac{i^{\vert \alpha \vert + 1}}{(2 \pi )^{\frac{1}{2}} \alpha !} a_{0 ,0} (0) ( - \lambda g^{-})^{\alpha} h^{\frac{1}{2} - \vert \alpha \vert} e^{i \psi (0) /h} .
\end{equation}
As usual, we have $\chi \widetilde{f} = \chi \widetilde{f}_{\theta}$ if the distortion holds outside of the support of $\chi \in C^{\infty}_{0}$. From \eqref{a20}, $\widetilde{f}$ (resp. $\widetilde{f}_{\theta}$) is in the image of $\Pi_{z_{\alpha}}$ (resp. $\Pi_{z_{\alpha} , \theta}$). Moreover, using \eqref{a9} (which gives that $a_{0,0} (0) \neq 0$), \eqref{a18}, \eqref{a21}, \eqref{a19}, \eqref{a20} and \eqref{a22}, we have, microlocally near $(0,0)$,
\begin{equation*}
\widetilde{f} = \widetilde{d} (x, h) e^{i \varphi_{+} (x) /h} ,
\end{equation*}
where $\widetilde{d} (x,h) \in S_{h}^{0} (1)$ is a classical symbol satisfying
\begin{equation} \label{a42}
\widetilde{d} (x,h) \sim \sum_{j =0}^{+ \infty} \widetilde{d}_{j} (x) h^{j} \quad \text{ and } \quad \widetilde{d}_{0} (x) = x^{\alpha} + \CO ( x^{\vert \alpha \vert +1}) .
\end{equation}
In particular, $\widetilde{f}$ is not identically zero. Then, $\Pi_{z_{\alpha}}$ can be written as
\begin{equation} \label{a41}
\Pi_{z_{\alpha}} = \widetilde{c} \big( \, \cdot \, , \overline{\widetilde{f}} \big) \widetilde{f} , 
\end{equation}
and $\widetilde{f}$ satisfies $iv)$ of Theorem \ref{a2}. Furthermore, using Lemma~\ref{a4}, integrating over $z$ and coming back to the definition of $\widetilde{f}$ (see \eqref{a20}), we immediately obtain the point $iii)$ of Theorem~\ref{a2}. Since $\widetilde{f}$ is in the image of $\Pi_{z_{\alpha}}$ which is the spectral projection at a simple resonance, the point $ii)$ of Theorem~\ref{a2} is clear. Combining $iii)$, $iv)$, \eqref{a5}, which gives a uniform bound outside of the energy level, together with ``the transport equation'' $ii)$, we get the point $i)$ by a standard argument of propagation of singularities.

\Subsection{Calculation of $\big( v (z_{\alpha} ) , \overline{\widetilde{f}} \big)$}
\label{a53}

Here we calculate the scalar product between $v (x , z_{\alpha} )$ and $\overline{\widetilde{f}} (x)$. From \eqref{a26} and \eqref{a25}, the function $v$ is supported near $\supp \partial_{x} \tau$ and micro-supported near $\{ (x, \xi ) \in \R^{2 n} ; \ x \in \supp \partial_{x} \tau \text{ and } (x , \xi ) \in \Lambda_{\psi} \}$. Then, if $\supp \tau$ is close enough to $0$, the previous section and \eqref{a26} imply that
\begin{equation*}
\big( v (z_{\alpha} ) , \overline{\widetilde{f}} \big) = \big( [P , \tau] b e^{i \psi /h} , \overline{\widetilde{d} e^{i \varphi_{+} /h}} \big) + \CO (h^{\infty}) .
\end{equation*}
A direct calculus gives
\begin{equation}
[ P , \tau ] \big( b e^{i \psi /h} \big) = \widetilde{b} (x , h) e^{i \psi (x) /h},
\end{equation}
with
\begin{equation} \label{a36}
\widetilde{b} (x,h) \sim \sum_{j=0}^{+ \infty} \widetilde{b}_{j} (x) h^{1 + j} \qquad \text{ and } \qquad \widetilde{b}_{0} (x) = - 2 i \partial_{x} \tau \cdot \partial_{x} \psi b_{0} (x) .
\end{equation}
Then, using that $\varphi_{+} = - \varphi_{-}$, we get
\begin{equation} \label{a27}
\big( v (z_{\alpha} ) , \overline{\widetilde{f}} \big) = \int \widetilde{b} (x,h) \widetilde{d} (x,h) e^{i ( \psi (x) - \varphi_{-} (x) )  /h} d x + \CO (h^{\infty}) .
\end{equation}

The critical points of the phase $\psi - \varphi_{-}$ (i.e. the points $x$ such that $\nabla \psi (x) = \nabla \varphi_{-} (x)$) are the points in the spatial projection of $\Lambda_{\psi} \cap \Lambda_{-} = \gamma^{-}$. Moreover, since this intersection is transversal, the phase function $\psi - \varphi_{-}$ is non degenerate in the directions that are transverse to $\pi_{x} \gamma^{-}$ ($\pi_{x}$ being the spatial projection). Then, applying the method of the stationary phase in the orthogonal directions of $\pi_{x} \gamma^{-}$ (written $(\pi_{x} \gamma^{-} )^{\perp}$) and parameterizing the curve $\pi_{x} \gamma^{-}$ by $x (t)$, \eqref{a27} gives
\begin{equation} \label{a28}
( v (z_{\alpha} ) , \overline{\widetilde{f}}) = \int r (t, h) e^{i ( \psi (x(t)) - \varphi_{-} (x(t)) ) /h} d t + \CO (h^{\infty}) ,
\end{equation}
with $r (t,h) \sim \sum_{j=0}^{+ \infty} r_{j} (t) h^{\frac{n+1}{2} + j}$ and
\begin{equation} \label{a37}
r_{0} (t) = (2 \pi)^{\frac{n-1}{2}} \frac{e^{i \frac{\pi}{4} \sgn ( \psi - \varphi_{-} )_{\vert_{( \pi_{x} \gamma^{-} )^{\perp}}}''}}{\big\vert \det ( \psi - \varphi_{-} )_{\vert_{( \pi_{x} \gamma^{-} )^{\perp}}}'' \big\vert^{\frac{1}{2}}} \vert \partial_{t} x(t) \vert \widetilde{b}_{0} (x (t)) \widetilde{d}_{0} (x (t)) .
\end{equation}
From \eqref{a29}, we have $\psi (x(t)) - \varphi_{-} (x(t)) = \psi (0)$ for all $t \in \R$. In particular, \eqref{a28} can be written
\begin{equation} \label{a30}
\big( v (z_{\alpha} ) , \overline{\widetilde{f}} \big) = e^{i \psi (0) /h} s (h) ,
\end{equation}
with
\begin{equation} \label{a40}
s (h) \sim \sum_{j=0}^{+ \infty} s_{j} h^{\frac{n+1}{2} + j} \qquad \text{ and } \qquad s_{0} = \int r_{0} (t) \, d t .
\end{equation}

From \eqref{a19} and \eqref{a22}, we have
\begin{equation} \label{a32}
\widetilde{d}_{0} ( x ) = \frac{a_{0 ,0} (x)}{a_{0,0} (0)} \prod_{j=1}^{n} \Big( \frac{\varphi_{\lambda_{j}} (x)}{- \lambda_{j} g_{\lambda_{j}}^{-}} \Big)^{\alpha_{j}} .
\end{equation}
Using \eqref{a11}, we have the transport equation
\begin{align*}
\partial_{t} \varphi_{\lambda_{j}} (x(t)) &= \partial_{t} x(t) \cdot \partial_{x} \varphi_{\lambda_{j}} (x(t)) = 2 \xi (t) \cdot \partial_{x} \varphi_{\lambda_{j}} (x(t))   \\
&= - 2 \partial_{x} \varphi_{+} (x(t)) \cdot \varphi_{\lambda_{j}} (x(t)) = - \lambda_{j} \varphi_{\lambda_{j}} (x(t)) ,
\end{align*}
which gives
\begin{equation} \label{a31}
\varphi_{\lambda_{j}} (x(t)) = e^{- \lambda_{j} t} \varphi_{\lambda_{j}} (x(0)) .
\end{equation}
On the other hand, since $\varphi_{\lambda_{j}} (x)$ is $C^{\infty}$ and $x (t)$ is expandible, the function $t \mapsto \varphi_{\lambda_{j}} (x(t))$ is expandible. Moreover, since $\lambda_{j}$ can not be written as a non-trivial combination of the $\lambda_{k}$'s, the Taylor expansion \eqref{a18} of $\varphi_{\lambda_{j}}$ shows that the term in $e^{-\lambda_{j} t}$ in the expansion of $\varphi_{\lambda_{j}} (x(t))$ is $- \lambda_{j} ( g_{\lambda_{j}}^{-} )^{2} e^{-\lambda_{j} t}$. Since \eqref{a31} gives another asymptotic expansion, the uniqueness of the development implies that
\begin{equation*}
\varphi_{\lambda_{j}} (x(t)) = - \lambda_{j} ( g_{\lambda_{j}}^{-} )^{2} e^{- \lambda_{j} t} .
\end{equation*}
Then, combining with \eqref{a32}, we obtain
\begin{equation} \label{a33}
\widetilde{d}_{0} ( x (t) ) = ( g^{-} )^{\alpha} e^{- \lambda \cdot \alpha t} \big( 1 + \CO (e^{- \varepsilon t}) \big) .
\end{equation}
Note here that the curve $\gamma^{-}$ has been chosen in Section \ref{a34} such that $( g^{-} )^{\alpha} \neq 0$.

From the construction of $u$ in \eqref{a26} and since $z_{\alpha}$ is a classical symbol (see Remark \ref{k1}) with $z_{\alpha} = z_{\alpha}^{0} + \CO (h^{2}) = E_{0} - i h ( \lambda\cdot \alpha + \sum \lambda_{j}/2) + \CO (h^{2})$, the function $b_{0}$ satisfies the usual transport equation
\begin{equation*}
2 \partial_{x} \psi \cdot \partial_{x} b_{0} + \Big( \Delta \psi - \lambda \cdot \alpha - \sum \lambda_{j} /2 \Big) b_{0} =0 .
\end{equation*}
Mimicking the proof of \eqref{a31}, we get
\begin{equation} \label{a43}
b_{0} ( x (t) ) = e^{- \int_{0}^{t} \Delta \psi (x(s)) - ( \sum \lambda_{j} /2 + \lambda \cdot \alpha ) \, d s} b_{0} (x (0)) .
\end{equation}
Therefore, \eqref{a36} gives
\begin{align}
\widetilde{b}_{0} ( x (t) ) &= -i b_{0} ( x(t) ) \partial_{t} \tau (x(t))    \nonumber  \\
&= - i e^{- \int_{0}^{t} \Delta \psi (x(s)) - ( \sum \lambda_{j} /2 + \lambda \cdot \alpha ) \, d s} b_{0} (x (0)) \partial_{t} \tau (x(t)) . \label{a35}
\end{align}

From Proposition C.1 of \cite{AlBoRa08_01} and since $g_{1}^{-} \neq 0$, we have
\begin{equation*}
(\psi - \varphi_{-} ) '' (x (t)) =  \left(
\begin{array}{cccccccc}
0 & & & & \\
& & \lambda_{2} & & \\
& & & \ddots & \\
& & & & \lambda_{n}
\end{array} \right)+ \CO ( e^{ - \varepsilon t}) ,
\end{equation*}
Using $x(t) = g_{1}^{-} e^{- \lambda_{1} t} + \CO ( e^{- ( \lambda_{1} + \varepsilon )t} )$, we get
\begin{equation} \label{a38}
\big\vert \det ( \psi - \varphi_{-} )_{\vert_{( \pi_{x} \gamma^{-} )^{\perp}}}'' (x(t)) \big\vert^{\frac{1}{2}} = \Big( \prod_{j=2}^{n} \lambda_{j} \Big)^{\frac{1}{2}} + \CO ( e^{ - \varepsilon t}) .
\end{equation}
and
\begin{equation} \label{a39}
\sgn ( \psi - \varphi_{-} )_{\vert_{( \pi_{x} \gamma^{-} )^{\perp}}}'' (x(t)) = n-1 ,
\end{equation}
for $t$ large enough.

Finally, using the expansion of $x(t)$, we have
\begin{equation} \label{a44}
\vert \partial_{t} x(t) \vert = \vert g_{\lambda_{1}}^{-} \vert \lambda_{1} e^{- \lambda_{1} t} \big( 1 + \CO (e^{- \varepsilon t}) \big) .
\end{equation}

Combining the definitions of $s_{0}$ \eqref{a40} and of $r_{0}$ \eqref{a37} with the relations \eqref{a33}, \eqref{a35}, \eqref{a38}, \eqref{a39} and \eqref{a44}, the constant $s_{0}$ does not vanish if $\partial_{t} \tau (x(t)) \geq 0$ and the support of $\partial_{t} \tau (x(t))$ is sufficiently small near $T$ large enough.

\Subsection{End of the proof of Theorem \ref{a2}}
\label{a57}

From \eqref{a20} and \eqref{a41}, we have
\begin{equation*}
\widehat{c} \widetilde{f} = \widetilde{c} \big( v (z_{\alpha}), \overline{\widetilde{f}} \big) \widetilde{f} .
\end{equation*}
In particular, using \eqref{a22} and \eqref{a30}, we get
\begin{equation}
\widetilde{c} = \frac{\widehat{c}}{\big( v (z_{\alpha}), \overline{\widetilde{f}} \big)} = - \frac{i^{\vert \alpha \vert + 1}}{(2 \pi )^{\frac{1}{2}} \alpha ! s (h)} a_{0 ,0} (0) ( - \lambda g^{-})^{\alpha} h^{\frac{1}{2} - \vert \alpha \vert} \sim \sum_{j=0}^{+ \infty} \widetilde{c}_{j} h^{- \frac{n}{2} - \vert \alpha \vert + j} ,
\end{equation}
with
\begin{equation}
\widetilde{c}_{0} = - \frac{i^{\vert \alpha \vert + 1} a_{0 ,0} (0) ( - \lambda g^{-})^{\alpha}}{(2 \pi )^{\frac{1}{2}} \alpha ! s_{0}} .
\end{equation}

At this point, the function $\widetilde{f}$ and the constant $\widetilde{c}$ may depend on $v$. Nevertheless, since $\Pi_{z_{\alpha}} = \widetilde{c} ( \, \cdot \, , \overline{\widetilde{f}} ) \widetilde{f}$ and $\widetilde{d}_{0}$ (the first term in the development of $\widetilde{f}$ given in \eqref{a42}) do not depend on $v$, the constant $\widetilde{c}_{0}$ also does not depend on $v$.

We choose a sequence of functions $\tau$ (say $\tau_{N}$), with $\partial_{t} \tau_{N} (x(t)) \geq 0$, such that $\partial_{t} \tau_{N} (x(t))$ converges to the Dirac mass $\delta_{t}$ for some fixed $t>0$. Then, from the definition of $s_{0}$ \eqref{a40} and of $\widetilde{b}_{0}$ \eqref{a35}, we get
\begin{equation*}
\widetilde{c}_{0} = \frac{i^{\vert \alpha \vert + 1} a_{0 ,0} (0) ( - \lambda g^{-})^{\alpha}}{i (2 \pi)^{\frac{n}{2}} \vert \partial_{t} x(t) \vert b_{0} (x (t)) \widetilde{d}_{0} (x (t)) \alpha !} \frac{\big\vert \det ( \psi - \varphi_{-} )_{\vert_{( \pi_{x} \gamma^{-} )^{\perp}}}'' \big\vert^{\frac{1}{2}}}{e^{i \frac{\pi}{4} \sgn ( \psi - \varphi_{-} )_{\vert_{( \pi_{x} \gamma^{-} )^{\perp}}}''}} .
\end{equation*}
Combining \eqref{a9}, \eqref{a33}, \eqref{a43}, \eqref{a38}, \eqref{a39} and \eqref{a44}, we obtain
\begin{equation*}
\widetilde{c}_{0} = \frac{i^{\vert \alpha \vert} e^{- i \frac{\pi}{4}} e^{- \int_{0}^{+ \infty} \Delta \psi (x(s)) - ( \sum \lambda_{j} /2 - \lambda_{1}) \, d s} ( - \lambda )^{\alpha}}{(2 \pi)^{\frac{n}{2}} e^{- \int_{0}^{t} \Delta \psi (x(s)) - ( \sum \lambda_{j} /2 - \lambda_{1}) \, d s} \alpha !} \frac{\big( \prod_{j=1}^{n} \lambda_{j} \big)^{\frac{1}{2}}}{e^{i (n-1) \frac{\pi}{4}}} \big( 1 + \CO (e^{- \varepsilon t} ) \big) .
\end{equation*}
Then, letting $t$ going to $+ \infty$ and using that $\widetilde{c}_{0}$ does not depend on $t$, it follows
\begin{equation} \label{a45}
\widetilde{c}_{0} = \frac{i^{\vert \alpha \vert} ( - \lambda )^{\alpha} \big( \prod_{j=1}^{n} \lambda_{j} \big)^{\frac{1}{2}}}{(2 \pi)^{\frac{n}{2}} e^{i n \frac{\pi}{4}} \alpha !} .
\end{equation}

We now consider a fixed $v$ as in the beginning of this subsection. With $c (h)$ as in \eqref{a24}, \eqref{a45} gives that $\widetilde{c} = c \breve{c}$ where
\begin{equation*}
\breve{c} \sim \sum_{j=0}^{+ \infty} \breve{c}_{j} h^{j} \qquad \text{ and } \qquad \breve{c}_{0} =1 .
\end{equation*}
Now, we define $f := \breve{c}^{\frac{1}{2}} \widetilde{f}$. Then, \eqref{a41} gives \eqref{a23} and the properties of $f$ given in Theorem~\ref{a2} follow from the properties of $\widetilde{f}$ given in Section~\ref{a46} and $\breve{c}_{0} =1$.

\section{Residue of the scattering amplitude}
\label{k2}

In this section, we give the semiclassical expansion of the residue of the scattering amplitude at an isolated resonance. To define the scattering matrix, we assume that the potential is long range:
\begin{hyp} \label{h7}
For some $\rho > 0$, we have $\vert V(x) \vert \lesssim \< x \>^{-\rho}$ for all $x \in \CS$.
\end{hyp}
Using the constructions of Isozaki and Kitada (see \cite{IsKi85_01} and \cite{IsKi86_01}), the assumption \ref{h7} allows to define the scattering matrix $S (z,h)$, $z \in ] 0, + \infty [$ related to the pair $P_{0} = - h^{2} \Delta$ and $P$ as a unitary operator
\begin{equation*}
S (z,h) : L^{2} ( \S^{n-1} ) \longrightarrow L^{2} ( \S^{n-1} ) .
\end{equation*}
In the short range case (i.e. $\rho >1$), this operator coincides with the usual scattering matrix. Next, introduce the operator $\CT (z,h)$ defined by
\begin{equation*}
S (z,h) = Id - 2 i \pi \CT (z,h) .
\end{equation*}
Its kernel $\CT ( \omega , \omega ' , z,h)$ is smooth away from the diagonal of $\S^{n-1} \times \S^{n-1}$ (see \cite{IsKi86_01}). Here, $\omega$ (resp. $\omega '$) is called the outgoing (resp. incoming) direction. Finally, the scattering amplitude is defined for $\omega \neq \omega '$ by
\begin{equation*}
{\mathcal A} ( \omega , \omega ' , z ,h) = c (z ,h) \CT ( \omega , \omega ' , z,h) ,
\end{equation*}
with
\begin{equation*}
c (z , h) = - ( 2 \pi ) z^{- \frac{n-1}{4}} (2 \pi h)^{\frac{n-1}{2}} e^{- i \frac{(n-3) \pi}{4}} .
\end{equation*}
In \cite{GeMa89_02}, G\'erard and Martinez have shown that for $\omega \neq \omega '$ fixed, the scattering amplitude has a meromorphic continuation to a neighborhood of $]0, + \infty [$, whose poles are the resonances of $P$. Moreover, the multiplicity of each pole is less or equal to the multiplicity of the resonance. Notice that, since the kernel of the residue of the scattering matrix is not singular at $\omega = \omega '$ (see Theorem 1.1 (iii) of \cite{GeMa89_02}), we drop the assumption $\omega \neq \omega '$ in the sequel.

We will now make some hypotheses on the behavior of the classical curves. Let $(x (t) , \xi (t) ) = \exp ( t H_{p} ) (x, \xi )$ be a Hamiltonian curve in $p^{-1} (E_{0})$. Under the hypotheses \ref{A1}--\ref{h7}, there are only two possible behaviors for $x (t)$ as $t \to \pm \infty$: either it escapes to $ \infty$, or it goes to $0$. From the long range assumption \ref{h7}, if $x (t)$ escapes to $\infty$, then $\xi (t)$ has a limit in $\sqrt{E_{0}} \S^{n-1}$. Moreover the set of points with asymptotic direction $\omega$ and $\omega '$,
\begin{gather*}
\Lambda_{\omega '}^{-} = \big\{ (x, \xi ) \in p^{-1} ( E_{0} ) ; \ \xi (t) \longrightarrow \sqrt{E_{0}} \omega ' \text{ as } t \to - \infty \big\} ,  \\
\Lambda_{\omega}^{+} = \big\{ (x, \xi ) \in p^{-1} ( E_{0} ) ; \ \xi (t) \longrightarrow \sqrt{E_{0}} \omega \text{ as } t \to + \infty \big\} ,
\end{gather*}
are Lagrangian submanifolds of $T^{*} \R^{n}$ (see \cite{DeGe97_01}). We suppose that
\begin{hyp}  \label{h8}
$\Lambda_{\omega '}^{-}$ and $\Lambda_-$ (resp. $\Lambda^+_{\omega}$ and $\Lambda_+$) intersect in a finite number $N_-$ (resp $N_+$) of bicharacteristic curves, with each intersection transverse.
\end{hyp}
We denote these curves, respectively,
\begin{equation*}
\gamma_k^-:t\mapsto \gamma_{k}^{-} (t) = (x_k^-(t),\xi^-_{k} (t) ), \quad 1\leq k\leq N_-, 
\end{equation*}
and 
\begin{equation*}
\gamma_\ell^+:t\mapsto \gamma_{\ell}^{+} (t)=(x^+_{\ell} (t),\xi^+_{\ell} (t) ), \quad 1\leq \ell\leq N_+.
\end{equation*}
Note that, from Proposition~2.5 of \cite{AlBoRa08_01}, the intersections $\Lambda_{\omega '}^{-} \cap \Lambda_-$ and $\Lambda^{+}_{\omega} \cap \Lambda_+$ are never empty (i.e. $N_{-} \geq 1$ and $N_{+} \geq 1$). From \cite{HeSj85_01}, the curve $\gamma_{\star}^{\pm}$ with $\star = k , \ell$ satisfies
\begin{equation*}
x_{\star}^{\pm} (t) \sim \sum_{j=1}^{+ \infty} g^{\star , \pm}_{\mu_{j}} (t) e^{\pm \mu_{j} t} \quad \text{ with } \quad g^{\star , \pm}_{\mu_{j}} (t) = \sum_{m=0}^{M^{\star , \pm}_{\mu_{j}}} g^{\star , \pm}_{\mu_{j} , m} t^{m} \quad \text{ as } t \to \mp \infty .
\end{equation*}
From Lemma~\ref{d3}, if $\lambda_{j}$ satisfies $\lambda \cdot \alpha = \lambda_{j} \Longrightarrow \vert \alpha \vert =1$, then $M_{\lambda_{j}}^{\star , \pm} =0$. Moreover, there always exists a $\mu_{j}$ such that $g_{\mu_{j}}^{\star , \pm} \neq 0$. We define
\begin{equation*}
\lambda_{\star}^{\pm} = \min \{ \mu_{j} ; \ g_{\mu_{j}}^{\star , \pm} \neq 0 \} .
\end{equation*}
We know that $\lambda_{\star}^{\pm}$ is one of the $\lambda_{j}$'s and that $M_{\lambda_{j}}^{\star , \pm} =0$ (see \cite[(2.18)]{AlBoRa08_01}). We shall denote
\begin{equation*}
S_{k}^{-} = \int_{- T_{k}^{-}}^{+ \infty} x_{k}^{-} (s) \partial_{x} V (x_{k}^{-} (s)) \, d s \quad \text{ and } \quad S_{\ell}^{+} = \int_{- \infty}^{T_{\ell}^{+}} x_{\ell}^{+} (s) \partial_{x} V (x_{\ell}^{+} (s)) \, d s ,
\end{equation*}
for some $T_{\star}^{\pm}$ large enough which is equal to $+ \infty$ in the short range case $\rho > 1$.

Moreover, in the short range case $\rho > 1$, the bicharacteristic curves in $\Lambda_{\alpha}^{\pm}$, $\alpha \in \S^{n-1}$, are the bicharacteristic curves $\gamma_{\pm} ( t , z , \alpha ) = ( x_{\pm} ( t , z , \alpha ) , \xi_{\pm}(t , z , \alpha ))$ for which there exists a $z \in \alpha^{\perp} \sim \R^{n-1}$ such that
\begin{align*}
&\lim_{t \to \pm \infty} \big\vert x_{\pm} ( t , z , \alpha ) - 2 \sqrt{E_{0}} \alpha t - z \big\vert = 0 , \\
&\lim_{t \to \pm \infty} \big\vert \xi_{\pm} (t , z , \alpha ) - \sqrt{E_{0}} \alpha \big\vert = 0 .
\end{align*}
These trajectories are smooth with respect to $t ,z, \alpha$. We denote by $z_{\star}^{\pm}$ the impact parameter of the curve $\gamma_{\star}^{\pm}$. Let
\begin{align*}
D_{k}^- &=\lim_{t\to +\infty} \Big\vert\det\frac{\partial x_-(t,z,\omega ' )}{\partial(t,z)}\vert_{z=z^-_k}\Big\vert\; e^{-(\Sigma_{j} \lambda_j-2\lambda_{k}^{-} )t},\\
D_{\ell}^+ &=\lim_{t\to -\infty}\Big\vert\det\frac{\partial x_+(t,z,\omega )}{\partial(t,z)}\vert_{z=z^+_\ell}\Big\vert\; e^{(\Sigma_{j} \lambda_j - 2 \lambda_{\ell}^{+} ) t} ,
\end{align*}
be the Maslov determinants for $\gamma_{\star}^{\pm}$ which exist and satisfy $0 < D_{\star}^{\pm} < + \infty$ (see \cite{AlBoRa08_01}). We shall also denote by $\nu_{\star}^{\pm}$ the Maslov index of the curve $\gamma_{\star}^{\pm}$.

\begin{theorem}[Residue of the scattering amplitude]\sl \label{a47}
Assume \ref{A1}--\ref{h8}. Let $\alpha \in \N^{n}$ be such that $z_{\alpha}^{0}$ is simple. Then, the residue of the scattering amplitude satisfies
\begin{equation*}
\residue \big( {\mathcal A} ( \omega , \omega ' ,z , h) , z = z_{\alpha} \big) = \sum_{k=1}^{N_{-}} \sum_{\ell =1}^{N_{+}} a_{k, \ell} h^{- \vert \alpha \vert + \frac{1}{2}} e^{i ( S_{k}^{-} + S_{\ell}^{+} ) /h} + \CO (h^{\infty}) ,
\end{equation*}
where
\begin{equation*}
a_{k, \ell} (h) = b_{k}^{-} (h) b_{\ell}^{+} (h) \quad \text{ and } \quad b_{\star}^{\pm} (h) \sim \sum_{j=0}^{+ \infty} b_{\star , j}^{\pm} h^{j} .
\end{equation*}
Moreover, $b_{\star , 0}^{\pm} = 0$ if and only if $( g^{\star , \pm} )^{\alpha} =0$. Finally, in the short range case $\rho >1$, we have
\begin{align*}
b_{k,0}^{-} b_{\ell , 0} ^{+} =& \frac{e^{- i \frac{\pi}{2} ( \vert \alpha \vert - \frac{1}{2} )}}{\sqrt{2 \pi} \alpha !} E_{0}^{\frac{n-1}{4}} ( \lambda_{k}^{-} \lambda_{\ell}^{+} )^{\frac{3}{2}} \prod_{j=1}^{n} \lambda_{j}^{\alpha_{j} - \frac{1}{2}}     \\
&\qquad \times e^{- i \nu_{k}^{-} \pi /2} e^{- i \nu_{\ell}^{+} \pi /2} ( D_{k}^{-} D_{\ell}^{+} )^{- \frac{1}{2}} ( g^{k , -} )^{\alpha} ( g^{\ell , +} )^{\alpha} \vert g_{\lambda_{k}^{-}}^{k , -} \vert \vert g_{\lambda_{\ell}^{+}}^{\ell , +} \vert .
\end{align*}
\end{theorem}

In the last formula $( g^{\star , \pm} )^{\alpha}$ is a shorthand for $\prod_{j=1}^{n} ( g^{\star , \pm}_{\lambda_{j}} )^{\alpha_{j}}$ where $g^{\star , \pm}_{\lambda_{j}}$ is identified with its $j$-th coordinate. To prove the theorem, we first obtain a representation formula for the scattering amplitude involving the resolvent. Then we apply Theorem \ref{a2} to express the residue of the scattering amplitude with the help of the resonant state $f$. Finally, the result follows from the computation of two scalar products which are done with the stationary phase method.

\begin{remark}\sl
Stefanov \cite{St02_01} (in the compact support case) and Michel \cite{Mi03_01} (in the long range case) have given a priori estimates for the residue of the scattering amplitude. For the resonances $z_{0}$ very close to the real axis (more precisely $\vert \im z_{0} \vert \lesssim h^{\frac{3n+5}{2}}$) and under a separation condition, they have proved that the residue satisfies
\begin{equation*}
\big\vert \residue \big( {\mathcal A} ( \omega , \omega ' ,z , h) , z = z_{0} \big) \big\vert \lesssim h^{- \frac{n-1}{2}} \vert \im z_{0} \vert .
\end{equation*}
In the present situation, these results do not apply since the resonances are ``too far'' from the real axis. Furthermore, the previous estimate does not hold. Indeed, the imaginary part of $z_{\alpha}$ behaves like $- \vert \alpha \vert h$ but the residue is typically of order $h^{-\vert \alpha \vert + \frac{1}{2}}$.
\end{remark}

In the one dimensional case, Theorem \ref{a47} can probably be deduced from the computation of the scattering amplitude obtained by the third author in \cite{Ra96_01}.

For a punctual well in the island case and under some geometrical assumptions, the asymptotic of the residue of the scattering amplitude has been computed by Nakamura \cite{Na89_01,Na89_02}, Lahmar-Benbernou \cite{Be99_01} and Lahmar-Benbernou and Martinez \cite{LaMa99_01}.

It is possible to compare Theorem \ref{a47} with the semiclassical expansion of the scattering amplitude for real energy obtained in \cite{AlBoRa08_01}. Assume for simplicity that the $\lambda_{j}$'s are non-resonant ($\Z$-independent for example), $\lambda_{n} < 2 \lambda_{1}$, $N_{-} = N_{+} = 1$, $N_{\infty} = 0$ and $g_{\lambda_{j}}^{1,-} \neq 0$ for all $j \in \{ 1 , \ldots , n\}$. In particular, we have $k = \ell = 1$. Let $J \in \{ 1 , \ldots , n\}$ be the first $j$ with $g_{\lambda_{j}}^{1,+} \neq 0$ (thus, $\lambda_{J} = \lambda^{+}_{1}$). In that case, Theorem 2.6 {\rm (a)} of \cite{AlBoRa08_01} gives
\begin{equation} \label{k7}
{\mathcal A} ( \omega , \omega ' , E , h) = \Big( f (E) \Gamma \Big( \frac{\Sigma (E)}{\lambda_{J}} \Big) + o (1) \Big) h^{\frac{\Sigma (E)}{\lambda_{J}} - \frac{1}{2}} e^{i ( S_{1}^{-} + S_{1}^{+} ) /h} ,
\end{equation}
for $E$ real with $E - E_{0} = \CO (h)$. Here,
\begin{equation*}
\Sigma (E) = \sum_{j=1}^{n} \frac{\lambda_{j}}{2} - i \frac{E-E_{0}}{h} ,
\end{equation*}
and $f(E)$ is an explicit function, analytic near $E_{0}$. Thus, the main term in \eqref{k7}, defined in \cite{AlBoRa08_01} for $E$ real, has a meromorphic extension in a fix neighborhood of $E_{0}$. Moreover, its poles are exactly the pseudo-resonances $z_{\alpha}^{0} \in \res_{0} (P)$ with $\alpha = ( 0 , \ldots , 0 , \alpha_{J} , 0 , \ldots , 0 )$ and the corresponding residue coincides with that given in Theorem \ref{a47}. In particular, this principal term does not contribute to the residue at the other (pseudo)-resonances. The cases {\rm (b)} and {\rm (c)} in Theorem 2.6 of \cite{AlBoRa08_01} only appear for resonant $\lambda_{j}$'s and the corresponding main terms in the semiclassical expansion of the scattering amplitude have poles at some $z_{\alpha}^{0} \in \res_{0} (P)$ which are not simple.

\Subsection{Representation formula for the scattering amplitude}

In this section, we recall a representation formula of the scattering amplitude for complex energies due to G\'erard and Martinez \cite{GeMa89_02}. Their approach consists in extending the formula of Isozaki and Kitada \cite{IsKi86_01} to complex energies. For this purpose, they show that the phases and the symbols involved in that formula can be chosen to be analytic in a suitable complex neighborhood of $\R^{2n}$. We only recall what will be useful in the following and refer to \cite{GeMa89_02} for the details.

For $R>0$ large enough, $d>0$, $\varepsilon >0$ and $\sigma \in ]0, 1[$, we denote
\begin{align*}
\Gamma^{\pm}_{\C} ( R ,d , \varepsilon , \sigma ) &= \big\{ ( x, \xi ) \in \C^{2n} ; \ \vert \re x \vert > R , \ d^{-1} < \vert \re \xi \vert < d , \ \vert \im x \vert \leq \varepsilon \< \re x \> , \\
&\qquad \qquad \qquad \qquad \qquad \vert \im \xi \vert \leq \varepsilon \< \re \xi \> \text{ and } \pm \cos ( \re x , \re \xi ) \geq \pm \sigma \big\} ,  \\
\Gamma^{\pm} ( R ,d , \sigma ) &= \Gamma^{\pm}_{\C} ( R ,d , \varepsilon , \sigma ) \cap \R^{2 n}.
\end{align*}
Let $\varepsilon >0$, $d \gg 1$, $-1 < \sigma_{1}^{-} < \sigma_{1}^{+} < 0 < \sigma_{2}^{-} < \sigma^{+}_{2} < 1$ and $R_{1} > 0$ be sufficiently large. For $k = 1,2$, we denote $\Gamma^{k} = \Gamma^{+}_{\C} ( R_{1} ,d, \varepsilon , \sigma_{k}^{+} ) \cup \Gamma^{-}_{\C} ( R_{1} ,d, \varepsilon , \sigma_{k}^{-} )$. In \cite{GeMa89_02}, G\'erard and Martinez construct some phases $\varphi_{k} \in C^{\infty} ( \R^{2n} ; \R )$ and some symbols $t_{k} \in C^{\infty} ( \R^{2n} ) \cap S_{h}^{0} (1)$ satisfying the general assumptions of Isozaki and Kitada \cite{IsKi85_01} and the following properties.

The phases $\varphi_{k}$ have a holomorphic extension to $\Gamma^{k}$ and satisfy
\begin{equation} \label{a49}
\left\{ \begin{aligned}
& ( \nabla_{x} \varphi_{k} (x, \xi ) )^{2} + V (x) = \xi^{2} , \\
& \partial^{\alpha}_{x} \partial^{\beta}_{\xi} \big( \varphi_{k} (x, \xi ) - x \cdot \xi \big) = \CO \big( \< x \>^{1- \rho - \vert \alpha \vert} \big) ,
\end{aligned} \right.
\end{equation}
uniformly in $\Gamma^{k}$. Moreover, $\Lambda_{\varphi_{k} ( \cdot , \sqrt{E_{0}} \omega )} = \{ (x, \partial_{x} \varphi_{k} ( x , \sqrt{E_{0}} \omega )) \} \subset \Lambda^{-}_{\omega} \cup \Lambda^{+}_{\omega}$.

There exist two symbols $a_{k} (x, \xi ,h) \in C^{\infty} ( \R^{2n} , \C )$ supported inside $\Gamma^{k} \cap \R^{2 n}$, with
\begin{equation*}
a_{k} (x, \xi ,h) \sim \sum_{j = 0}^{+ \infty} a_{k ,j } (x, \xi ) h^{j} ,
\end{equation*}
such that
\begin{equation*}
\big\vert \partial_{x}^{\alpha} \partial_{\xi}^{\beta} a_{k} (x ,\xi ,h) \big\vert \lesssim \< x \>^{- \vert \alpha \vert} \quad \text{ and } \quad \big\vert \partial_{x}^{\alpha} \partial_{\xi}^{\beta} a_{k , j} (x ,\xi) \big\vert \lesssim \< x \>^{- j - \vert \alpha \vert} .
\end{equation*}
Moreover, for some $\delta >0$ with $-1 < \sigma_{k}^{-} - \delta < \sigma_{k}^{+} + \delta < 1$, we have
\begin{equation}
\big\vert \partial_{x}^{\alpha} \partial_{\xi}^{\beta} ( a_{k,0} (x , \xi ) -1 ) \big \vert \lesssim \< x \>^{-\rho - \vert \alpha \vert} ,
\end{equation}
for $(x , \xi ) \in \Gamma^{+} ( 2 R_{1} , d/2 , \sigma_{k}^{+} + \delta ) \cup \Gamma^{-} ( 2 R_{1} , d/2 , \sigma_{k}^{-} - \delta )$. Finally, they extend holomorphically with respect to $X = \vert x \vert$ and $\Xi = \vert \xi \vert$ for $X$ in $\{ \re X > 3 R_{1} , \ \vert \im X \vert < \varepsilon \< \re X \> \}$ and $\Xi$ in a complex neighborhood of $\sqrt{E_{0}}$. Furthermore, their extensions continue to satisfy estimates analogous to the previous ones.

The symbols $t_{k}$ are then defined by
\begin{equation} \label{a51}
t_{k} (x, \xi , h) = e^{- i \varphi_{k} (x, \xi ) /h} ( P - \xi^{2} ) \big( a_{k} ( \cdot , \xi ,h) e^{i \varphi_{k} ( \cdot , \xi ) / h} \big) ,
\end{equation}
and satisfy, for some $\widetilde{\varepsilon} >0$,
\begin{equation} \label{a50}
\big\vert \partial_{x}^{\alpha} \partial_{\xi}^{\beta} t_{k} (x, \xi ,h) \big\vert = \CO \big( e^{- \widetilde{\varepsilon} \< x \> /h} \big) ,
\end{equation}
uniformly with respect to $h$ and $(x , \xi ) \in \Gamma^{+}_{\C} ( 2 R_{1} ,d/2 , \varepsilon , \sigma_{k}^{+} + \delta) \cup \Gamma^{-}_{\C} ( 2 R_{1} ,d /2 , \varepsilon , \sigma_{k}^{-} -\delta )$.

Under the assumption \ref{h7}, G\'erard and Martinez \cite{GeMa89_02} have proved that the scattering amplitude can be written
\begin{equation} \label{a48}
{\mathcal A} ( \omega , \omega ' ,z , h) = \widetilde{c} (z , h) g ( \omega , \omega ' , z ,h) + f ( \omega , \omega ' , z ,h) ,
\end{equation}
where $f ( \omega , \omega ' , z ,h)$ has a holomorphic extension in a (fixed) neighborhood of $E_{0}$,
\begin{equation*}
g ( \omega , \omega ' , z ,h) = \Big( ( P_{\theta} - z )^{-1} U_{i \theta} \big( e^{i \varphi_{2} (x , \sqrt{z} \omega ') /h} t_{2} ( x , \sqrt{z} \omega ' ,h ) \big) , U_{\overline{i \theta}} \big( e^{i \varphi_{1} (x, \sqrt{\overline{z}} \omega) /h} t_{1} (x , \sqrt{\overline{z}} \omega ,h ) \big) \Big) ,
\end{equation*}
and
\begin{equation} \label{a72}
\widetilde{c} (z , h) = \pi ( 2 \pi h )^{- \frac{n+1}{2}} z^{\frac{n-3}{4}} e^{-i\frac{(n-3)\pi}{4}} .
\end{equation}
By assumption, the resonance $z_{\alpha}$ is simple for $h$ small enough. Moreover, Theorem \ref{aaa} implies that $\Pi_{z_{\alpha} , \theta} = \CO ( h^{-M} )$ for $\theta = \nu h \vert \ln h \vert$ and some $M >0$. Then Lemma 5.4 of \cite{BoMi04_01} (see also Proposition 5.1 of \cite{Be99_01} in the case of a well in the island) states that
\begin{align*}
{\mathcal R} : = & \residue \big( {\mathcal A} ( \omega , \omega ' ,z , h) , z = z_{\alpha} \big)   \\
= & - \widetilde{c} ( z_{\alpha} , h) \Big( \Pi_{z_{\alpha} , \theta} \chi  U_{i \theta} \big( e^{i \varphi_{2} (x , \sqrt{z_{\alpha}} \omega ') /h} t_{2} ( x , \sqrt{z_{\alpha}} \omega ' ,h ) \big) ,   \\
&\qquad \qquad \qquad \qquad \qquad \qquad \qquad \chi U_{\overline{i \theta}} \big( e^{i \varphi_{1} (x, \sqrt{\overline{z_{\alpha}}} \omega) /h} t_{1} (x , \sqrt{ \overline{z_{\alpha}}} \omega ,h ) \big) \Big) + \CO ( h^{\infty} ) ,
\end{align*}
where $\chi \in C^{\infty}_{0} ( \R^{n} )$ satisfies $\one_{\vert x \vert \leq 2 R_{1}} \prec \chi \prec \one_{\vert x \vert \leq 3 R_{1}}$ with $R_{0} \gg R_{1}$. In particular, there is no distortion (i.e. $F =0$) on the support of $\chi$ and Theorem \ref{a2} implies
\begin{equation}\label{a67}
{\mathcal R} = \widehat{c} \Big( f , e^{i \varphi_{1} (x , \sqrt{\overline{z_{\alpha}}} \omega ) /h} \chi t_{1} ( x , \sqrt{\overline{z_{\alpha}}} \omega ,h ) \Big) \Big( e^{i \varphi_{2} (x , \sqrt{z_{\alpha}} \omega ') /h} \chi t_{2} ( x , \sqrt{z_{\alpha}} \omega ' ,h ) , \overline{f} \Big) + \CO ( h^{\infty} ) ,
\end{equation}
where $\widehat{c} = - \widetilde{c} ( z_{\alpha} , h ) c (h)$ with $c (h)$ given by \eqref{a24}.

\Subsection{Calculation of $( f , e^{i \varphi_{1} /h} \chi t_{1} )$}
\label{a68}

We will calculate the scalar product $( f , e^{i \varphi_{1} /h} \chi t_{1} )$ by the stationary phase method. First, we will prove that this quantity has an asymptotic expansion in power of $h$ and then calculate the first term using a limit at the origin. We will use arguments close to the ones developed in Section \ref{a53} or \cite[Section 7]{AlBoRa08_01}.

\begin{figure}
\begin{center}
\begin{picture}(0,0)%
\includegraphics{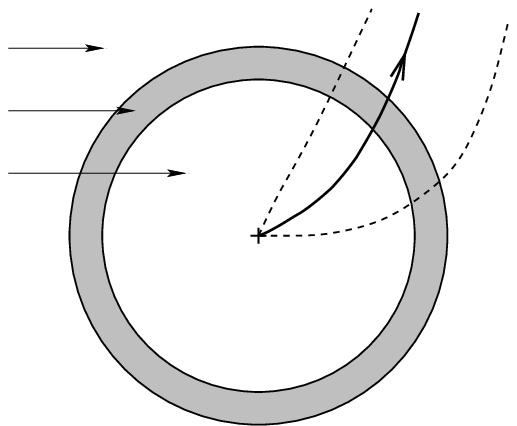}%
\end{picture}%
\setlength{\unitlength}{987sp}%
\begingroup\makeatletter\ifx\SetFigFont\undefined%
\gdef\SetFigFont#1#2#3#4#5{%
  \reset@font\fontsize{#1}{#2pt}%
  \fontfamily{#3}\fontseries{#4}\fontshape{#5}%
  \selectfont}%
\fi\endgroup%
\begin{picture}(10023,8033)(811,-7611)
\put(9151,-286){\makebox(0,0)[lb]{\smash{{\SetFigFont{10}{12.0}{\rmdefault}{\mddefault}{\updefault}$\gamma_{\ell}^{+}$}}}}
\put(5851,-4636){\makebox(0,0)[lb]{\smash{{\SetFigFont{10}{12.0}{\rmdefault}{\mddefault}{\updefault}$0$}}}}
\put(9376,-2011){\makebox(0,0)[lb]{\smash{{\SetFigFont{10}{12.0}{\rmdefault}{\mddefault}{\updefault}$\Omega$}}}}
\put(826,-361){\makebox(0,0)[rb]{\smash{{\SetFigFont{10}{12.0}{\rmdefault}{\mddefault}{\updefault}$\chi = 0$}}}}
\put(901,-1486){\makebox(0,0)[rb]{\smash{{\SetFigFont{10}{12.0}{\rmdefault}{\mddefault}{\updefault}$\supp ( \nabla \chi )$}}}}
\put(901,-2761){\makebox(0,0)[rb]{\smash{{\SetFigFont{10}{12.0}{\rmdefault}{\mddefault}{\updefault}$\chi = 1$}}}}
\end{picture}%
$\qquad \qquad$
\begin{picture}(0,0)%
\includegraphics{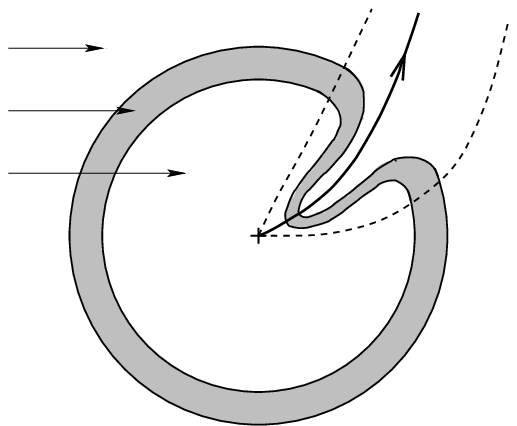}%
\end{picture}%
\setlength{\unitlength}{987sp}%
\begingroup\makeatletter\ifx\SetFigFont\undefined%
\gdef\SetFigFont#1#2#3#4#5{%
  \reset@font\fontsize{#1}{#2pt}%
  \fontfamily{#3}\fontseries{#4}\fontshape{#5}%
  \selectfont}%
\fi\endgroup%
\begin{picture}(10023,8033)(811,-7611)
\put(9151,-286){\makebox(0,0)[lb]{\smash{{\SetFigFont{10}{12.0}{\rmdefault}{\mddefault}{\updefault}$\gamma_{\ell}^{+}$}}}}
\put(5851,-4636){\makebox(0,0)[lb]{\smash{{\SetFigFont{10}{12.0}{\rmdefault}{\mddefault}{\updefault}$0$}}}}
\put(826,-361){\makebox(0,0)[rb]{\smash{{\SetFigFont{10}{12.0}{\rmdefault}{\mddefault}{\updefault}$\widetilde{\chi} = 0$}}}}
\put(901,-1486){\makebox(0,0)[rb]{\smash{{\SetFigFont{10}{12.0}{\rmdefault}{\mddefault}{\updefault}$\supp ( \nabla \widetilde{\chi} )$}}}}
\put(901,-2761){\makebox(0,0)[rb]{\smash{{\SetFigFont{10}{12.0}{\rmdefault}{\mddefault}{\updefault}$\widetilde{\chi} = 1$}}}}
\put(9376,-2011){\makebox(0,0)[lb]{\smash{{\SetFigFont{10}{12.0}{\rmdefault}{\mddefault}{\updefault}$\Omega$}}}}
\end{picture}%
\end{center}
\caption{The functions $\chi$ and $\widetilde{\chi}$.}
\label{f2}
\end{figure}

Denote $u = e^{i \varphi_{1} (x , \sqrt{\overline{z_{\alpha}}} \omega ) /h} a_{1} ( x , \sqrt{\overline{z_{\alpha}}} \omega ,h )$ and $v  = e^{i \varphi_{1} (x , \sqrt{\overline{z_{\alpha}}} \omega ) /h} t_{1} ( x , \sqrt{\overline{z_{\alpha}}} \omega ,h )$. From Theorem \ref{a2} $ii)$ and \eqref{a51}, we have
\begin{equation}
( f , \chi v ) = \big( f , \chi ( P- \overline{z_{\alpha}} ) u \big) = \big( ( P- z_{\alpha} ) f , \chi u \big) + \big( f , [ \chi , P ] u \big) = - \big( f , [ P , \chi ] u \big) .
\end{equation}
From \eqref{a50} and the choice of $\chi$, $( P - \overline{z_{\alpha}} ) u = 0$ microlocally near $\Gamma^{+} ( 2 R_{1} ,d/2 , \sigma_{1}^{+} + \delta)$. Moreover, since $z_{\alpha}$ has an asymptotic expansion in power of $h$, we can write, microlocally near $\Gamma^{+} ( 2 R_{1} ,d/2 , \sigma_{1}^{+} + \delta)$,
\begin{equation*}
u = \widetilde{a} (x, h) e^{i \varphi_{1} (x , \sqrt{E_{0}} \omega ) /h} ,
\end{equation*}
where $\widetilde{a}$ has an asymptotic expansion in power of $h$. Note that $( \supp \nabla \chi \times \R^{n} ) \cap \bigcup_{\ell} \gamma_{\ell}^{+} \subset \Gamma^{+} ( 2 R_{1} ,d/2 , \sigma_{1}^{+} + \delta)$. Using Maslov's theory, we can extend the function $u$ near $\Omega$, a small neighborhood of $\bigcup_{\ell} \gamma_{\ell}^{+} \cap ( B (0 , 3 R_{0} ) \times \R^{n} )$, such that $u$ is still a solution of $(P - \overline{z_{\alpha}} ) u =0$ microlocally in $\Omega$.  Let $\widetilde{\chi} ( x, \xi ) \in C^{\infty} ( T^{*} \R^{n} )$ be such that $\widetilde{\chi} (x, \xi ) = \chi (x)$ out of $\Omega$ (see Figure~\ref{f2}). In particular, $( P - \overline{z_{\alpha}} ) u =0$ microlocally near the support of $\chi - \widetilde{\chi}$. So, we have
\begin{align} \label{a52}
( f , \chi v ) =& - \big( f , [ P , \Op ( \widetilde{\chi} ) ] u \big) - \big( (P - z_{\alpha} ) f , \Op ( \chi - \widetilde{\chi}) u \big) + \big(f , \Op ( \chi - \widetilde{\chi}) (P - \overline{z_{\alpha}} )  u \big)   \nonumber \\
=& - \big( f , [ P , \Op ( \widetilde{\chi} ) ] u \big) + \CO (h^{\infty} ). 
\end{align}

On the other hand, since $\overline{z_{\alpha}} = E_{0} + \CO (h)$, the microsupport of $[ P , \Op ( \widetilde{\chi} ) ] u$ satisfies
\begin{align*}
\MS \big( [ P , \Op ( \widetilde{\chi} ) ] u \big) &\subset \Lambda_{\varphi_{1} ( \cdot , \sqrt{E_{0}} \omega ) } \cap \supp \nabla \widetilde{\chi}   \\
&\subset \big( \Lambda^{-}_{\omega} \cap \Gamma^{-} ( R_{1} , d , \sigma^{-}_{1} /2) \big) \cup \big( \Lambda^{+}_{\omega} \cap \Omega \big) .
\end{align*}
Moreover, Theorem~\ref{a2} gives $\MS ( f) \subset \Lambda_{+}$. Then, modulo $\CO (h^{\infty} )$, the non-zero contributions to $( f , [ P , \Op ( \widetilde{\chi} ) ] u )$ comes from the values of the functions $f$ and $[ P , \Op ( \widetilde{\chi} ) ] u$ microlocally on the set $\bigcup_{\ell} \gamma^{+}_{\ell}$ (which constitute the intersection of the two microsupports). Let $g_{\ell}^{+}$ be $C_{0}^{\infty} ( T^{*} \R^{n} )$ functions with support in a small enough neighborhood of $\gamma_{\ell}^{+} \cap ( B ( 0 ,3 R_{0}) \times \R^{n} )$ such that $g_{\ell}^{+} =1$ in a similar neighborhood. Then, \eqref{a52} becomes
\begin{equation} \label{a54}
( f , \chi v ) = - \sum_{\ell =1}^{N_{+}} \big( f , \Op ( g_{\ell}^{+} ) [ P , \Op ( \widetilde{\chi} ) ] u \big) + \CO (h^{\infty} ). 
\end{equation}

We now compute $\Op ( g_{\ell}^{+} ) [ P , \Op ( \widetilde{\chi} ) ] u$. From Proposition C.1 of \cite{AlBoRa08_01}, the Lagrangian manifold $\Lambda_{\omega}^{+}$ has a nice projection with respect to $x$ in a neighborhood of any point of  $\gamma_{\ell}^{+}$ close to $(0,0)$. Then, Maslov's theory implies that $u$ can be written as
\begin{equation*}
u (x) = a_{\ell}^{+} (x, h) e^{i \psi_{\ell}^{+} (x) /h} ,
\end{equation*}
microlocally in such a neighborhood. From the construction of \cite{IsKi85_01} and \cite{GeMa89_02}, we see that
\begin{equation} \label{a56}
\psi_{\ell}^{+} ( x_{\ell}^{+} (t)) = x_{\ell}^{+} (t) \xi_{\ell}^{+} (t) - \int_{t}^{T_{\ell}^{+}} x_{\ell}^{+} (s) \partial_{x} V (x_{\ell}^{+} (s)) \, d s ,
\end{equation}
for some $T_{\ell}^{+} >0$ large enough (equal to $+ \infty$ in the short range case). The symbol $a_{\ell}^{+}$ has an asymptotic expansion $a_{\ell}^{+} (x,h) \sim \sum_{j} a_{\ell , j}^{+} (x) h^{j}$ with $a_{\ell , 0}^{+} ( x_{\ell}^{+} (t)) \neq 0$. Moreover, in the short range case, Equation (7.12) of \cite{AlBoRa08_01} gives
\begin{equation} \label{a64}
a_{\ell , 0}^{+} ( x_{\ell}^{+} (t) ) = e^{i \nu_{\ell}^{+} \pi /2} 2^{\frac{1}{2}} E_{0}^{\frac{1}{4}} ( D_{\ell}^{+} (t) )^{- \frac{1}{2}} e^{- t ( \sum \lambda_{k} /2 + \lambda \cdot \alpha )} ,
\end{equation}
where $\nu_{\ell}^{+}$ is the Maslov index of the curve $\gamma_{\ell}^{+}$ and $D_{\ell}^{+} (t)$ is the Maslov's determinant
\begin{equation*}
D_{\ell}^{+} (t) = \bigg\vert \det \frac{\partial x_{+} ( t , z , \omega )}{\partial (t,z)} \vert_{z = z_{\ell}^{+}} \bigg\vert.
\end{equation*}
Moreover, from Section 6 of \cite{AlBoRa08_01}, we know that
\begin{equation} \label{a65}
D_{\ell}^{+} = \lim_{t \to - \infty} D_{\ell}^{+} (t) e^{t ( \sum \lambda_{k} - 2 \lambda_{\ell}^{+})} ,
\end{equation}
exists and satisfies $0 < D_{\ell}^{+} < + \infty$. So,
\begin{equation}
\Op (g_{\ell}^{+} ) [ P , \Op ( \widetilde{\chi} ) ] u = \widetilde{a}_{\ell}^{+} (x ,h ) e^{i \psi_{\ell}^{+} (x) /h} ,
\end{equation}
with
\begin{equation*}
\widetilde{a}_{\ell}^{+} (x ,h) \sim \sum_{j = 0}^{+ \infty} \widetilde{a}_{\ell , j}^{+} (x) h^{1 + j} ,
\end{equation*}
and
\begin{equation} \label{a60}
\widetilde{a}_{\ell , 0}^{+} (x) = - i ( \{ p , \widetilde{\chi} \} g_{\ell}^{+} ) ( x, \partial_{x} \psi_{\ell}^{+} (x)) a_{\ell , 0}^{+} (x) .
\end{equation}

Since the support of $g_{\ell}^{+} (x, \xi ) \partial_{x, \xi} \widetilde{\chi} (x, \xi )$ is close enough to $(0,0)$, Theorem~\ref{a2} $iv)$ and \eqref{a54} imply that
\begin{equation} \label{a55}
( f , \chi v ) = - \sum_{\ell =1}^{N_{+}} \int d (x,h) \overline{\widetilde{a}_{\ell}^{+} (x ,h)} e^{i ( \varphi_{+} (x) - \psi_{\ell}^{+} (x) ) /h} d x + \CO (h^{\infty} ). 
\end{equation}
We proceed now as in \eqref{a28}. In the support of the symbol $\widetilde{a}_{\ell}^{+}$, the critical points of the phase function $\varphi_{+} - \psi_{\ell}^{+}$ (i.e. the points $x$ such that $\partial_{x} \varphi_{+} (x) = \partial_{x} \psi_{\ell}^{+} (x)$) are the points in the spacial projection of $\gamma_{\ell}^{+}$. Since this intersection $\Lambda_{+} \cap \Lambda_{\omega}^{+} = \gamma_{\ell}^{+}$ is transverse from the assumption \ref{h8}, the phase $\varphi_{+} - \psi_{\ell}^{+}$ is non degenerate in the directions transverse to $\pi_{x} \gamma_{\ell}^{+}$. Therefore, performing the method of the stationary phase in the orthogonal directions of $\pi_{x} \gamma^{+}_{\ell}$ (as in \eqref{a28}) and parameterizing the curve $\pi_{x} \gamma^{+}_{\ell}$ by $x_{\ell}^{+} (t)$, \eqref{a55} gives
\begin{equation}
( f , \chi v ) = - \sum_{\ell =1}^{N_{+}} \int b_{\ell}^{+} (t,h) e^{i ( \varphi_{+} ( x_{\ell}^{+} (t) ) - \psi_{\ell}^{+} ( x_{\ell}^{+} (t) ) ) /h} d t + \CO (h^{\infty} ) .
\end{equation}
with $b_{\ell}^{+} (t,h) \sim \sum_{j \geq 0} b_{\ell , j}^{+} (t) h^{\frac{n+1}{2} +j}$ and
\begin{equation*}
b_{\ell , 0}^{+} (t) = ( 2 \pi )^{\frac{n-1}{2}} \frac{e^{i \frac{\pi}{4} \sgn ( \varphi_{+} - \psi_{\ell}^{+} ) ''_{\vert_{( \pi_{x} \gamma_{\ell}^{+} )^{\perp}}}}}{\big\vert \det ( \varphi_{+} - \psi_{\ell}^{+} )_{\vert_{( \pi_{x} \gamma^{+}_{\ell} )^{\perp}}} \big\vert^{\frac{1}{2}}} \vert \partial_{t} x_{\ell}^{+} (t) \vert d_{0} ( x_{\ell}^{+} (t) ) \overline{\widetilde{a}_{\ell ,0}^{+} ( x_{\ell}^{+} (t) )} .
\end{equation*}

Since $\gamma_{\ell}^{+} \in \Lambda_{+} \cap \Lambda_{\omega}^{+}$, $\varphi_{+} ( x_{\ell}^{+} (t) )$ and $\psi_{\ell}^{+} ( x_{\ell}^{+} (t) )$ have the same derivative (with respect to $t$), and \eqref{a56} gives
\begin{equation*}
\varphi_{+} ( x_{\ell}^{+} (t) ) - \psi_{\ell}^{+} ( x_{\ell}^{+} (t) ) = \int_{- \infty}^{T} x_{\ell}^{+} (s) \partial_{x} V (x_{\ell}^{+} (s)) \, d s = S_{\ell}^{+},
\end{equation*}
for all $t \geq 0$. Then, combining with \eqref{a60}, we get
\begin{equation} \label{a66}
( f , \chi v ) = \sum_{\ell =1}^{N_{+}} e^{i S_{\ell}^{+} /h} h^{\frac{n+1}{2}} \widetilde{b}_{\ell}^{+} (h) ,
\end{equation}
with $\widetilde{b}_{\ell}^{+} (h) \sim \sum_{j \geq 0} \widetilde{b}_{\ell , j}^{+} h^{j}$ and
\begin{align}
\widetilde{b}_{\ell , 0}^{+} = - i ( 2 \pi )^{\frac{n-1}{2}} \int \frac{e^{i \frac{\pi}{4} \sgn ( \varphi_{+} - \psi_{\ell}^{+} ) ''_{\vert_{( \pi_{x} \gamma_{\ell}^{+} )^{\perp}}}}}{\big\vert \det ( \varphi_{+} - \psi_{\ell}^{+} )_{\vert_{( \pi_{x} \gamma^{+}_{\ell} )^{\perp}}} \big\vert^{\frac{1}{2}}} \vert \partial_{t} & x_{\ell}^{+} (t) \vert d_{0} ( x_{\ell}^{+} (t) )   \nonumber \\
& \times \overline{a_{\ell , 0}^{+} ( x_{\ell}^{+} (t) )} \partial_{t} \widetilde{\chi} ( x_{\ell}^{+} (t) , \xi_{\ell}^{+} (t) ) \, d t .   \label{a63}
\end{align}

From \eqref{a52}, $( f , \chi v )$ does not depend on $\widetilde{\chi}$, modulo $\CO ( h^{\infty} )$. In particular, changing $\widetilde{\chi}$ in a neighborhood of a fixed curve $\gamma_{\ell}^{+}$, we obtain that each $b_{\ell , 0}^{+}$ does not depend on $\widetilde{\chi}$. From Proposition~C.1 of \cite{AlBoRa08_01}, we have, up to a linear change of variables in $\R^{n}$,
\begin{equation*}
( \varphi_{+} - \psi_{\ell}^{+} ) '' (x (t)) =  \diag \big( \lambda_{1} , \cdots , \lambda_{j ( \ell ) -1}, 0 , \lambda_{j ( \ell )+1} , \cdots , \lambda_{n} \big) + \CO ( e^{ - \varepsilon t}) ,
\end{equation*}
where $j ( \ell )$ is such that $\lambda_{j ( \ell )} = \lambda_{\ell}^{+}$. Since $x_{\ell}^{+} (t) = g_{\lambda_{\ell}^{+}}^{\ell ,+} e^{\lambda_{\ell}^{+} t} + \CO ( e^{( \lambda_{\ell}^{+} + \varepsilon )t} )$ as expandible symbol (see \cite[Definition 5.2]{BoFuRaZe07_01}), this implies
\begin{gather}
\big\vert \det ( \varphi_{+} - \psi_{\ell}^{+} )_{\vert_{( \pi_{x} \gamma^{+}_{\ell} )^{\perp}}} \big\vert^{\frac{1}{2}} (x_{\ell}^{+} (t)) = \Big( \prod_{j \neq j ( \ell )} \lambda_{j} \Big)^{\frac{1}{2}} + \CO ( e^{\varepsilon t}) ,  \label{a58} \\
\sgn ( \varphi_{+} - \psi_{\ell}^{+} )_{\vert_{( \pi_{x} \gamma^{+}_{\ell} )^{\perp}}}'' (x_{\ell}^{+} (t)) = n-1 ,  \label{a59} \\
\vert \partial_{t} x_{\ell}^{+} (t) \vert = \vert g_{\lambda_{\ell}^{+}}^{\ell , +} \vert \lambda_{\ell}^{+} e^{\lambda_{\ell}^{+} t} \big( 1 + \CO ( e^{\varepsilon t} ) \big) , \label{a61}
\end{gather}
as $t \to - \infty$. On the other hand, \eqref{a33} gives
\begin{equation} \label{a62}
d_{0} ( x_{\ell}^{+} (t)) = ( g^{\ell , +} )^{\alpha} e^{\lambda \cdot \alpha t} \big( 1 + \CO ( e^{\varepsilon t} ) \big) .
\end{equation}

We first consider the long range case ($\rho >0$). If $( g^{\ell , +} )^{\alpha} = 0$, then \eqref{a63} and \eqref{a62} imply that $b_{\ell , 0}^{+} =0$. We will now prove that $b_{\ell , 0}^{+} \neq 0$ if $( g^{\ell , +} )^{\alpha} \neq 0$. Let $T >0$ be sufficiently large such that the quantities in \eqref{a58}, \eqref{a59} and \eqref{a62} do not vanish and \eqref{a61} holds at $t = -T$. Then, if $\widetilde{\chi} ( x_{\ell}^{+} (t) )$ satisfies $\partial_{t} \widetilde{\chi}( x_{\ell}^{+} (t) ) \leq 0$ and has its support close enough to $T$, the previous discussion, $a_{\ell , 0}^{+} ( x_{\ell}^{+} (T)) \neq 0$ and \eqref{a63} imply that $b_{\ell , 0}^{+} \neq 0$.

Let us now consider the short range case ($\rho > 1$). Assume that the support of $\partial_{t} \widetilde{\chi} (x_{\ell}^{+} (t))$ is sufficiently negative. Then, the formula \eqref{a63} and the estimates \eqref{a64}, \eqref{a65}, \eqref{a58}, \eqref{a59}, \eqref{a61} and \eqref{a62} give
\begin{equation*}
\widetilde{b}_{\ell , 0}^{+} = - i ( 2 \pi )^{\frac{n-1}{2}} \frac{2^{\frac{1}{2}} e^{i (n-1) \frac{\pi}{4}}}{\big( \prod_{j \neq j ( \ell )} \lambda_{j} \big)^{\frac{1}{2}}} \vert g_{\lambda_{\ell}^{+}}^{\ell , +} \vert \lambda_{\ell}^{+} ( g^{\ell , +} )^{\alpha} e^{- i \nu_{\ell}^{+} \pi /2} E_{0}^{\frac{1}{4}} ( D_{\ell}^{+} )^{- \frac{1}{2}} \int \partial_{t} \widetilde{\chi} ( x_{\ell}^{+} (t) ) ( 1 + o (1)) \, d t ,
\end{equation*}
where the $o (1)$ does not depend on $\widetilde{\chi}$. Now, we take a sequence of functions $\widetilde{\chi}$ such that the support of $\partial_{t} \widetilde{\chi} ( x_{\ell}^{+} (t))$ goes to $- \infty$ and $\partial_{t} \widetilde{\chi} (x_{\ell}^{+} (t)) \leq 0$ (see Figure~\ref{f2}). Since $b_{\ell , 0}^{+}$ does not depend on $\widetilde{\chi}$, the previous expression gives
\begin{equation} \label{a69}
\widetilde{b}_{\ell , 0}^{+} = i ( 2 \pi )^{\frac{n-1}{2}} \frac{2^{\frac{1}{2}} e^{i (n-1) \frac{\pi}{4}}}{\big( \prod_{j \neq j ( \ell )} \lambda_{j} \big)^{\frac{1}{2}}} \vert g_{\lambda_{\ell}^{+}}^{\ell , +} \vert \lambda_{\ell}^{+} ( g^{\ell , +} )^{\alpha} e^{- i \nu_{\ell}^{+} \pi /2} E_{0}^{\frac{1}{4}} ( D_{\ell}^{+} )^{- \frac{1}{2}} .
\end{equation}

\Subsection{End of the proof of Theorem \ref{a47}}

Following the approach of Section \ref{a68}, one can prove that
\begin{equation} \label{a70}
\Big( e^{i \varphi_{2} (x , \sqrt{z_{\alpha}} \omega ') /h} \chi t_{2} ( x , \sqrt{z_{\alpha}} \omega ' ,h ) , \overline{f} \Big) = \sum_{k =1}^{N_{-}} e^{i S_{k}^{-} /h} h^{\frac{n+1}{2}} \widetilde{d}_{k}^{-} (h) ,
\end{equation}
with $\widetilde{d}_{k}^{-} (h) \sim \sum_{j \geq 0} \widetilde{d}_{k , j}^{-} h^{j}$ and $\widetilde{d}_{k , j}^{-} = 0$ if and only if $( g^{k , -} )^{\alpha} = 0$. Moreover, in the short range case, we have
\begin{equation} \label{a71}
\widetilde{d}_{k , 0}^{-} = i ( 2 \pi )^{\frac{n-1}{2}} \frac{2^{\frac{1}{2}} e^{i (n-1) \frac{\pi}{4}}}{\big( \prod_{j \neq j ( k )} \lambda_{j} \big)^{\frac{1}{2}}} \vert g_{\lambda_{k}^{-}}^{k , -} \vert \lambda_{k}^{-} ( g^{k , -} )^{\alpha} e^{- i \nu_{k}^{-} \pi /2} E_{0}^{\frac{1}{4}} ( D_{k}^{-} )^{- \frac{1}{2}} .
\end{equation}
Then, combining the representation of the residue given in \eqref{a67} with the constants given in \eqref{a24} and \eqref{a72}, and the scalar products \eqref{a66} and \eqref{a70}, we obtain 
\begin{equation*}
\residue \big( {\mathcal A} ( \omega , \omega ' ,z , h) , z = z_{\alpha} \big) = h^{- \vert \alpha \vert + \frac{1}{2}} \sum_{k=1}^{N_{-}} \sum_{\ell =1}^{N_{+}} a_{k, \ell} (h) e^{i ( S_{k}^{-} + S_{\ell}^{+} ) /h} ,
\end{equation*}
with $a_{k, \ell} (h) \sim \sum_{j\geq 0} a_{k, \ell}^{j} h^{j}$ and $a_{k , \ell}^{0} = 0$ if and only if $( g^{k , -} )^{\alpha} ( g^{\ell , +} )^{\alpha} =0$. Moreover, in the short range case, \eqref{a69} and \eqref{a71} imply
\begin{align}
a_{k , \ell}^{0} =& \frac{e^{- i \frac{\pi}{2} ( \vert \alpha \vert - \frac{1}{2} )}}{\sqrt{2 \pi} \alpha !} E_{0}^{\frac{n-1}{4}} ( \lambda_{k}^{-} \lambda_{\ell}^{+} )^{\frac{3}{2}} \prod_{j=1}^{n} \lambda_{j}^{\alpha_{j} - \frac{1}{2}}    \nonumber \\
&\times e^{- i \nu_{k}^{-} \pi /2} e^{- i \nu_{\ell}^{+} \pi /2} ( D_{k}^{-} D_{\ell}^{+} )^{- \frac{1}{2}} ( g^{k , -} )^{\alpha} ( g^{\ell , +} )^{\alpha} \vert g_{\lambda_{k}^{-}}^{k , -} \vert \vert g_{\lambda_{\ell}^{+}}^{\ell , +} \vert .
\end{align}

\section{Large time behavior of the Schr\"odinger group}
\label{k3}

In this section, we prove a resonance expansion for the cut-off Schr\"{o}dinger propagator. The proof relies on the resolvent estimate in Theorem \ref{aaa} and on standard arguments.

\begin{theorem}[Schr\"{o}dinger group expansion]\sl \label{c1}
Assume \ref{A1}--\ref{A3}.
Let $\mu >0$ be different from  $\sum_{j=1}^{n} (\alpha_{j} + \frac{1}{2} ) \lambda_{j}$ for all $\alpha \in \N^{n}$. Let $\chi \in C^{\infty}_{0} ( \R^{n} )$ and $\psi \in C^{\infty}_{0} ( [E_{0} - \varepsilon , E_{0} + \varepsilon ] )$ for some $\varepsilon > 0$ small enough. Then, there exists $K = K ( \mu ) >0$ such that
\begin{align*}
\chi e^{- i t P /h} \chi \psi ( P )= \sum_{z_{\alpha} \in \res (P) \cap D (E_{0} , \mu h)} - \chi \residue \big( e^{- i t z /h} (P - z)^{-1} & , z = z_{\alpha} \big) \chi \psi (P)   \\
&+ \CO ( h^{\infty} ) + \CO ( e^{- \mu t} h^{- K} ) ,
\end{align*}
for all $t \geq 0$. In particular, if all the $z_{\alpha}^{0}$ in $D ( E_{0} , \mu h )$ are simple, we have
\begin{equation*}
\chi e^{- i t P /h} \chi \psi ( P ) = \sum_{z_{\alpha} \in \res (P) \cap D (E_{0} , \mu h)} e^{- i t z_{\alpha} /h} \chi \Pi_{z_{\alpha}} \chi \psi (P) + \CO ( h^{\infty} ) + \CO ( e^{- \mu t} h^{- K} ) ,
\end{equation*}
for all $t \geq 0$. Here, $\Pi_{z_{\alpha}}$ is the generalized spectral projection associated to $z_{\alpha}$ and described in Theorem~\ref{a2}.
\end{theorem}

\begin{remark}\sl
Note that the previous expansions make sense only for $t > \frac{K}{\mu} \vert \ln h \vert$.
\end{remark}

One might think that the resonance expansion holds for shorter times. But, in fact, it is not possible to do much better. This follows from the paper of De Bi\`evre and Robert \cite{DeRo03_01} which is stated with slightly different hypotheses. In the one dimensional case, they have proved that the coherent states propagate through a maximum of the potential for times of order $\frac{1}{\lambda_{1}} \vert \ln h \vert$ and that they stay at $(0,0)$ before. On the other hand, the sum of the generalized spectral projections over the resonances appearing in Theorem~\ref{c1} can not be microlocalized only at $(0,0)$ thanks to Theorem~\ref{a2}. Thus, if the resonance expansion with a small error holds at time $t \geq 0$, we have necessarily $t \geq \frac{1}{\lambda_{1}} \vert \ln h \vert$ in the one dimensional case. If we only want to prove that $t \to + \infty$ as $h \to 0$, we can more simply apply the standard propagation of singularities with an initial data microlocalized in $\Lambda_{-} \setminus \{ ( 0, 0) \}$.

There is also a simplest way to justify this phenomena. Let $\mu$ be such that $\sum \lambda_{j} /2 < \mu < \lambda_{1} + \sum \lambda_{j} /2$. Then, $z_{0}$ is the unique resonance in $D ( E_{0} , \mu h)$ for $h$ small enough and $z_{0}^{0}$ is always simple. Assume that, for some $t \geq 0$, we can write
\begin{equation} \label{k8}
\chi e^{- i t P /h} \chi \psi ( P )= e^{- i t z_{0} /h} \chi \Pi_{z_{0}} \chi \psi (P) + R ,
\end{equation}
where $R$ is small. The left hand side is of order $1$ since the propagator is unitary. On the other hand, from Theorem \ref{e17} and Theorem \ref{a2}, the right hand side is of order $e^{- t \sum \lambda_{j} /2} h^{-\frac{n}{2}}$. Then, \eqref{k8} implies $t \geq \frac{n}{\sum \lambda_{j}} \vert \ln h \vert$. Remark that this critical time coincides with the one obtained by De Bi\`evre and Robert in the one dimensional case.

The situation is different for the well in the island case which was treated by Nakamura, Stefanov and Zworski \cite{NaStZw03_01}. In that setting, the cut-off Schr\"{o}dinger group is well approximated by the resonance expansion after a fix time. This is in adequacy with the geometrical interpretation since a fix time is enough to dispel the part of the initial data which is not localized in the well.

Nevertheless, G\'erard and Sigal \cite{GeSi92_01} have proved that the Schr\"{o}dinger equation with a quasiresonant state (sorts of quasimodes) as initial data is always well approximated by the resonance expansion for all time $t \geq 0$.

\begin{remark}\sl
When $t / \vert \ln h \vert \to + \infty$ as $h \to 0$, the sum over the resonances is negligible and Theorem \ref{c1} simply yields $\chi e^{- i t P /h} \chi \psi ( P ) = \CO ( h^{\infty} )$.
\end{remark}

The remainder terms $\CO ( h^{\infty} )$ in Theorem \ref{c1} come from the $C^{\infty}$ pseudodifferential calculus. Thus, if the cut-off functions $\chi , \psi$ are in some Gevrey class, it is perhaps possible to replace these remainder terms by $\CO (e^{- h^{- \delta}} )$ for some $\delta >0$. In that case, the sum over the resonances will dominate the remainders until $t$ is of order $h^{- \delta}$.

Burq and Zworski \cite{BuZw01_01} (see also Tang and Zworski \cite{TaZw00_01}) have obtained a long time expansion of semiclassical propagators in terms of resonances close to the real axis. Their result in the present situation gives $\chi e^{- i t P /h} \chi \psi ( P ) = \CO ( h^{\infty} )$ for all $t > h^{ - L}$ for some $L > 0$.

\begin{proof}
Let $f \in C^{\infty}_{0} ( [ E_{0} - 3 \varepsilon , E_{0} + 3 \varepsilon ])$ be such that $f =1$ near $[ E_{0} - 2 \varepsilon , E_{0} + 2 \varepsilon ]$. Then, from the pseudodifferential calculus, we get
\begin{align*}
I := & \chi e^{- i t P /h} \chi \psi ( P) = \chi e^{ - i t P /h} f (P) \chi \psi (P) + \CO (h^{\infty} )  \\
= & \int_{\R} e^{- i t z /h} f ( z ) \chi d E_{z} \chi \psi (P) + \CO (h^{\infty} ) ,
\end{align*}
where $d E_{z}$, the spectral projection, is given by the Stone formula
\begin{equation*}
d E_{z} = \frac{1}{2 \pi i} \big( R_{+} ( z ) - R_{-} ( z ) \big) \, d z ,
\end{equation*}
and $R_{\pm} ( z ) = ( P - z )^{-1}$ is analytic for $\pm \im z > 0$. Then, 
\begin{equation*}
I = \frac{1}{2 \pi i} \int_{\R} e^{- i t z /h} f ( z ) \chi \big( R_{+} ( z ) - R_{-} ( z ) \big) \chi \psi (P) \, d z + \CO (h^{\infty} ) ,
\end{equation*}
Making a change of contour, we obtain
\begin{align}
I = \sum_{z_{\alpha} \in \res (P) \cap D (E_{0} , \mu h)} - \chi \residue \big( e^{- i t z /h} & R_{+} (z) , z = z_{\alpha} \big) \chi \psi (P)  \nonumber \\
& + I_{1} + I_{2} + I_{3} + I_{4} + I_{5} + \CO (h^{\infty}) ,  \label{c2}
\end{align}
where
\begin{equation} \label{c3}
\begin{aligned}
I_{j} & = \frac{1}{2 \pi i} \int_{\Gamma_{j}} e^{- i t z /h} f ( z ) \chi \big( R_{+} ( z ) - R_{-} ( z ) \big) \chi \psi (P) \, d z  &&\text{ for } j = 1 , 5 ,  \\
I_{j} & = \frac{1}{2 \pi i} \int_{\Gamma_{j}} e^{- i t z /h} \chi \big( R_{+} ( z ) - R_{-} ( z ) \big) \chi \psi (P) \, d z &&\text{ for } j = 2 , 3 , 4 ,  \\
\end{aligned}
\end{equation}
and $\Gamma_{1} = ] - \infty , E_{0} - 2 \varepsilon ]$, $\Gamma_{2} = E_{0} - 2 \varepsilon + i [0 , - \mu h ]$, $\Gamma_{3} = [E_{0} - 2 \varepsilon , E_{0} + 2 \varepsilon ] - i \mu h$, $\Gamma_{4} = E_{0} + 2 \varepsilon + i [- \mu h , 0]$ and $\Gamma_{5} = [E_{0} + 2 \varepsilon , + \infty [$ (see Figure~\ref{f3}). The theorem will follow from the estimates on the $I_{j}$'s given below.

\begin{figure}
\begin{center}
\begin{picture}(0,0)%
\includegraphics{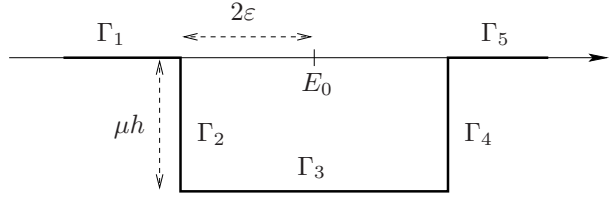}%
\end{picture}%
\setlength{\unitlength}{1381sp}%
\begingroup\makeatletter\ifx\SetFigFont\undefined%
\gdef\SetFigFont#1#2#3#4#5{%
  \reset@font\fontsize{#1}{#2pt}%
  \fontfamily{#3}\fontseries{#4}\fontshape{#5}%
  \selectfont}%
\fi\endgroup%
\begin{picture}(10844,3266)(504,-6405)
\put(2401,-5311){\makebox(0,0)[lb]{\smash{{\SetFigFont{10}{12.0}{\rmdefault}{\mddefault}{\updefault}$\mu h$}}}}
\put(2101,-3661){\makebox(0,0)[lb]{\smash{{\SetFigFont{10}{12.0}{\rmdefault}{\mddefault}{\updefault}$\Gamma_{1}$}}}}
\put(9001,-3661){\makebox(0,0)[lb]{\smash{{\SetFigFont{10}{12.0}{\rmdefault}{\mddefault}{\updefault}$\Gamma_{5}$}}}}
\put(5701,-6061){\makebox(0,0)[lb]{\smash{{\SetFigFont{10}{12.0}{\rmdefault}{\mddefault}{\updefault}$\Gamma_{3}$}}}}
\put(4501,-3286){\makebox(0,0)[lb]{\smash{{\SetFigFont{10}{12.0}{\rmdefault}{\mddefault}{\updefault}$2 \varepsilon$}}}}
\put(3901,-5461){\makebox(0,0)[lb]{\smash{{\SetFigFont{10}{12.0}{\rmdefault}{\mddefault}{\updefault}$\Gamma_{2}$}}}}
\put(8701,-5461){\makebox(0,0)[lb]{\smash{{\SetFigFont{10}{12.0}{\rmdefault}{\mddefault}{\updefault}$\Gamma_{4}$}}}}
\put(5776,-4561){\makebox(0,0)[lb]{\smash{{\SetFigFont{10}{12.0}{\rmdefault}{\mddefault}{\updefault}$E_{0}$}}}}
\end{picture}%
\end{center}
\caption{The contours $\Gamma_{j}$.}
\label{f3}
\end{figure}


$\bullet$ Estimations of $I_{1}$ and $I_{5}$. Using that $\Gamma_{1} \cap \supp \psi = \emptyset$, there exists $g \in C^{\infty}_{0} ( \R )$ such that $g=1$ near $\supp f \cap \Gamma_{1}$ and $g =0$ near $\supp \psi$. Then, by pseudodifferential calculus, $g (P) \chi \psi (P) = \CO (h^{\infty})$. Therefore, \eqref{c3} yields
\begin{align}
I_{1} &= \chi e^{- i t P /h} \one_{\Gamma_{1}} ( P ) f (P) \chi \psi (P) \nonumber  \\
&= \chi e^{- i t P /h} \one_{\Gamma_{1}} ( P ) f (P) g (P) \chi \psi (P) \nonumber  \\
&= \CO (h^{\infty} ). \label{c4}
\end{align}
The same way, we get $I_{5} = \CO (h^{\infty})$.


$\bullet$ Estimations of $I_{3}$. Using Theorem \ref{aaa} for $R_{\pm}$, we obtain
\begin{equation} \label{c5}
\Vert I_{3} \Vert \lesssim \int_{\Gamma_{3}} \big\vert e^{- i t z /h} \big\vert \big\Vert \chi \big( R_{+} ( z ) - R_{-} ( z ) \big) \chi \big\Vert \, d z  = \CO (  e^{- \mu t} h^{-K} ) .
\end{equation}


$\bullet$ Estimations of $I_{2}$ and $I_{4}$. Let $\theta = \nu h \vert \ln h \vert$ be as in Theorem \ref{aaa} and assume that the distortion occurs outside of the support of $\chi$. Then, $\chi R_{+} (z) \chi = \chi ( P_{\theta} - z)^{-1} \chi$ and $\chi R_{-} (z) \chi = \chi ( P_{- \theta} - z)^{-1} \chi$. In particular, we can write
\begin{equation} \label{c6}
I_{2} = \frac{1}{2 \pi i} \int_{\Gamma_{2}} e^{- i t z /h} \chi \big( ( P_{\theta} - z )^{-1} - ( P_{- \theta} - z )^{-1} \big) \chi \psi (P) \, d z .
\end{equation}
Let $k \in C^{\infty}_{0} ( \R )$ be such that $k =1$ near $E_{0} - 2 \varepsilon$ and $k=0$ near $\supp \psi$ (see Figure \ref{f4}). Then, for $z \in \Gamma_{2}$,
\begin{align*}
( P_{\pm \theta} -z)^{-1} &= ( P_{\pm \theta} -z)^{-1} k (P) + ( P_{\pm \theta} -z)^{-1} (1-k) (P)  \\
&= ( P_{\pm \theta} -z)^{-1} k (P) + ( P -z)^{-1} (1-k) (P)   \\
&\hspace{110pt} + ( P_{\pm \theta} -z)^{-1} ( P - P_{\pm \theta} ) ( P -z)^{-1} (1-k) (P) .
\end{align*}
Therefore \eqref{c6} becomes
\begin{equation} \label{c10}
I_{2} = J_{1}^{+} - J_{1}^{-} + J_{2}^{+} - J_{2}^{-} ,
\end{equation}
where
\begin{align*}
J_{1}^{\pm} &= \frac{1}{2 \pi i} \int_{\Gamma_{2}} e^{- i t z /h} \chi ( P_{\pm \theta} - z )^{-1} k (P) \chi \psi (P) \, d z   \\
J_{2}^{\pm} &= \frac{1}{2 \pi i} \int_{\Gamma_{2}} e^{- i t z /h} \chi ( P_{\pm \theta} -z)^{-1} ( P - P_{\pm \theta} ) ( P -z)^{-1} (1-k) (P) \chi \psi (P) \, d z .
\end{align*}

\begin{figure}
\begin{center}
\begin{picture}(0,0)%
\includegraphics{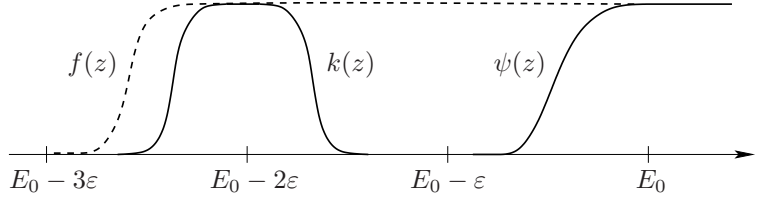}%
\end{picture}%
\setlength{\unitlength}{1381sp}%
\begingroup\makeatletter\ifx\SetFigFont\undefined%
\gdef\SetFigFont#1#2#3#4#5{%
  \reset@font\fontsize{#1}{#2pt}%
  \fontfamily{#3}\fontseries{#4}\fontshape{#5}%
  \selectfont}%
\fi\endgroup%
\begin{picture}(13469,3421)(-696,-4625)
\put(376,-2461){\makebox(0,0)[lb]{\smash{{\SetFigFont{10}{12.0}{\rmdefault}{\mddefault}{\updefault}$f (z)$}}}}
\put(2926,-4561){\makebox(0,0)[lb]{\smash{{\SetFigFont{10}{12.0}{\rmdefault}{\mddefault}{\updefault}$E_{0} - 2 \varepsilon$}}}}
\put(6526,-4561){\makebox(0,0)[lb]{\smash{{\SetFigFont{10}{12.0}{\rmdefault}{\mddefault}{\updefault}$E_{0} - \varepsilon$}}}}
\put(10576,-4561){\makebox(0,0)[lb]{\smash{{\SetFigFont{10}{12.0}{\rmdefault}{\mddefault}{\updefault}$E_{0}$}}}}
\put(-674,-4561){\makebox(0,0)[lb]{\smash{{\SetFigFont{10}{12.0}{\rmdefault}{\mddefault}{\updefault}$E_{0} - 3 \varepsilon$}}}}
\put(8026,-2461){\makebox(0,0)[lb]{\smash{{\SetFigFont{10}{12.0}{\rmdefault}{\mddefault}{\updefault}$\psi (z)$}}}}
\put(5026,-2461){\makebox(0,0)[lb]{\smash{{\SetFigFont{10}{12.0}{\rmdefault}{\mddefault}{\updefault}$k (z)$}}}}
\end{picture}%
\end{center}
\caption{The cut-off functions $f$, $k$ and $\psi$.}
\label{f4}
\end{figure}

Since $k$ and $\psi$ have disjoint support, the pseudodifferential calculus gives $k (P) \chi \psi (P) = \CO ( h^{\infty} )$. Thus, Theorem \ref{aaa} implies
\begin{equation} \label{c9}
\Vert J_{1}^{\pm} \Vert \lesssim \int_{\Gamma_{2}} \big\Vert ( P_{\pm \theta} - z )^{-1} \big\Vert  \big\Vert k (P) \chi \psi (P) \big\Vert \, \vert d z \vert = \CO (h^{\infty} ) ,
\end{equation}
since $\im z \leq 0$ for $z \in \Gamma_{2}$.

On the other hand, $P - P_{\pm\theta} \in \Psi_{h}^{0} ( \theta \< \xi \>^{2} )$ and $(P -z)^{-1} ( 1-k) (P) \in \Psi_{h}^{0} ( \< \xi \>^{-2} )$ uniformly in $z \in \Gamma_{2}$. Moreover, $P - P_{\pm \theta}$ is a differential operator whose coefficients are supported outside of the support of $\chi$. Then, the microlocal analysis gives
\begin{equation*}
\big\Vert ( P - P_{\pm \theta} ) ( P -z)^{-1} (1-k) (P) \chi \big\Vert = \CO (h^{\infty} ) ,
\end{equation*}
uniformly in $z \in \Gamma_{2}$. Combining this estimate with Theorem \ref{aaa}, we get
\begin{equation} \label{c7}
J_{2}^{\pm} = \CO (h^{\infty} ).
\end{equation}

Using \eqref{c9} and \eqref{c7} to estimate \eqref{c10}, we conclude $I_{2} = \CO (h^{\infty} )$. The same way, we have $I_{4} = \CO (h^{\infty})$.
\end{proof}

\appendix

\section{Construction of test curves}
\label{k6}

In this section, we construct Hamiltonian curves in $\Lambda_{-}$ with a prescribed asymptotic expansion at infinity. They are used in Section \ref{a1}, where test functions for the projection are built in a microlocal neighborhood of these curves. We will work on $\Lambda_{-}$, but the same work can be done in $\Lambda_{+}$.

Let $\gamma^{-} (t)$ be a Hamiltonian curve in $\Lambda_{-}$. From \cite[Section 3]{HeSj85_01}, the curve $\gamma^{-}$ satisfies, in the sense of expandible functions,
\begin{equation} \label{d7}
\gamma^{-} (t) \sim \sum_{k = 1}^{+ \infty} \gamma^{-}_{\mu_{k}} (t) e^{- \mu_{k} t} \qquad \text{with} \qquad \gamma^{-}_{\mu_{k}} (t) = \sum_{m=0}^{M_{\mu_{k}}} \gamma_{\mu_{k} ,m}^{-} t^{m} .
\end{equation}
The spectrum of $F_{p}$ is $\sigma ( F_{p} )= \{ - \lambda_{n} , \ldots , - \lambda_{1} , \lambda_{1} , \ldots , \lambda_{n} \}$. We denote by $\Pi_{\mu}$ the spectral projection on the eigenspace of $F_{p}$ associated to $- \mu$. Remark that
\begin{equation} \label{d4}
\ker ( F_{p} + \mu ) \oplus \im ( F_{p} + \mu ) = \R^{2n} .
\end{equation}

\begin{lemma}\sl \label{d3}
Let $\gamma^{-} (t)$ be a Hamiltonian curve in $\Lambda_{-}$. Assume that $\lambda_{j}$ is such that $\lambda_{j} = \alpha \cdot \lambda$, $\alpha \in \N^{n}$, implies $\vert \alpha \vert =1$. Then, $M_{\lambda_{j}} =0$ and $\gamma_{\lambda_{j} ,0}^{-} \in \ker ( F_{p} + \lambda_{j} )$.
\end{lemma}

\begin{proof}
We have $\partial_{t} \gamma^{-} (t) = H_{p} ( \gamma^{-}(t) )$. Taking the Taylor expansion of $H_{p}$ at $0$, we get
\begin{equation} \label{d2}
\partial_{t} \gamma^{-} (t) = F_{p} ( \gamma^{-} (t) ) + G_{2} ( \gamma^{-} (t) ) + \cdots + G_{K} ( \gamma^{-} (t) ) + \CO ( e^{- ( \lambda_{j} + \varepsilon ) t} ),
\end{equation}
where $G_{k}$ is a polynomial of order $k$ and $K > \lambda_{j} / \lambda_{1}$. Since $\lambda_{j}$ can not be written as the sum of at least two terms $\mu_{\ell}$, the cross products $G_{k}$ in the previous formula provide no term of the form $e^{- \lambda_{j} t}$. Then,
\begin{equation}
\sum_{m=0}^{M_{\lambda_{j}}} - \lambda_{j} \gamma_{\lambda_{j} ,m}^{-} t^{m} + m \gamma_{\lambda_{j} ,m}^{-} t^{m-1} = \sum_{m=0}^{M_{\lambda_{j}}} F_{p} ( \gamma_{\lambda_{j} ,m}^{-} ) t^{m} ,
\end{equation}
which can be written
\begin{equation} \label{d5}
\left\{ \begin{aligned}
&( F_{p} + \lambda_{j} ) \gamma_{\lambda_{j} ,m}^{-} = 0 &&\text{for } m = M_{\lambda_{j}}  \\
&( F_{p} + \lambda_{j} ) \gamma_{\lambda_{j} ,m}^{-} = (m+1) \gamma_{\lambda_{j} ,m+1}^{-} \quad &&\text{for } 0 \leq m < M_{\lambda_{j}} .
\end{aligned} \right.
\end{equation}
If $M_{\lambda_{j}} \geq 1$, the previous equation, together with \eqref{d4}, gives a contradiction. Thus, $M_{\lambda_{j}} =0$ and $\gamma_{\lambda_{j} ,0}^{-} \in \ker ( F_{p} + \lambda_{j} )$ from \eqref{d5}.
\end{proof}

We begin the construction with the following formal result.

\begin{lemma}\sl \label{d6}
If $\widetilde{\gamma}^{-}_{\lambda_{j},0} \in \ker ( F_{p} + \lambda_{j})$ for all $j \in \{ 1 , \ldots , n \}$, then there exists a formal Hamiltonian curve $\gamma^{-}$ of the form \eqref{d7} such that
\begin{equation} \label{d9}
\forall j \in \{ 1 , \ldots , n \} \qquad \Pi_{\lambda_{j}} ( \gamma^{-}_{\lambda_{j},0} ) = \widetilde{\gamma}^{-}_{\lambda_{j},0} .
\end{equation}
\end{lemma}

\begin{proof}
We construct the coefficients $\gamma_{\mu_{k}}^{-}$ inductively. Using the Taylor expansion of $H_{p}$ at $0$ as in \eqref{d2}, one can see that it is enough to find $\gamma_{\mu_{k}}^{-}$, $k \geq 0$, such that
\begin{equation} \label{d8}
\sum_{m=0}^{M_{\mu_{k}}} - \mu_{k} \gamma_{\mu_{k} ,m}^{-} t^{m} + m \gamma_{\mu_{k} ,m}^{-} t^{m-1} = \sum_{m=0}^{M_{\mu_{k}}} F_{p} ( \gamma_{\mu_{k} ,m}^{-} ) t^{m} + \sum_{m=0}^{N_{\mu_{k}}} R_{\mu_{k} ,m} t^{m} ,
\end{equation}
where the $R_{ \mu_{k} ,m}$ depend only on the $\gamma_{\mu_{\ell}}^{-}$ for $\ell < k$. Assume that the $\gamma_{\mu_{\ell}}^{-}$ have been chosen to satisfy \eqref{d8} for all $\mu_{\ell} < \mu_{k}$ and \eqref{d9} for all $\lambda_{j} < \mu_{k}$.

If $\mu_{k} \notin \{ \lambda_{1} , \ldots , \lambda_{n} \}$, then it is enough to take $M_{\mu_{k}} = N_{\mu_{k}}$,
\begin{equation*}
\gamma_{\mu_{k} ,M_{\mu_{k}}}^{-} = - ( F_{p} + \mu_{k} )^{-1} R_{\mu_{k} , M_{\mu_{k}}} ,
\end{equation*}
and, for $0 \leq m < M_{\mu_{k}}$,
\begin{equation*}
\gamma_{\mu_{k} ,m}^{-} = ( F_{p} + \mu_{k} )^{-1} \big( ( m +1) \gamma_{\mu_{k} ,m+1}^{-} - R_{\mu_{k} ,m} \big) .
\end{equation*}
If $\mu_{k} = \lambda_{j}$ for some $j$, then we take $M_{\mu_{k}} = N_{\mu_{k}} +1$ and
\begin{align*}
\gamma_{\mu_{k} ,M_{\mu_{k}}}^{-} &= M_{\mu_{k}}^{-1} \Pi_{\lambda_{j}} R_{\mu_{k} , M_{\mu_{k}} -1}  \\
\gamma_{\mu_{k} ,M_{\mu_{k}} -1 }^{-} &= ( M_{\mu_{k}} -1)^{-1} \Pi_{\lambda_{j}} R_{\mu_{k} , M_{\mu_{k}} -2} - K_{\lambda_{j}} ( 1 - \Pi_{\lambda_{j}} ) R_{\mu_{k} , M_{\mu_{k}} -1}  \\
\vdots \quad &= \qquad \vdots \\
\gamma_{\mu_{k} , 0}^{-} &= \widetilde{\gamma}^{-}_{\lambda_{j},0} - K_{\lambda_{j}} ( 1 - \Pi_{\lambda_{j}} ) R_{\mu_{k} , 0} .
\end{align*}
Here, $K_{\lambda_{j}}$ is the inverse of the map $F_{p} + \lambda_{j} : \im ( F_{p} + \lambda_{j} ) \longrightarrow \im ( F_{p} + \lambda_{j} )$. With these choices, \eqref{d9} and \eqref{d8} are always verified.
\end{proof}

\begin{proposition}\sl \label{d10}
If $\widetilde{\gamma}^{-}_{\lambda_{j},0} \in \ker ( F_{p} + \lambda_{j})$ for all $j \in \{ 1 , \ldots , n \}$, then there exists a Hamiltonian curve $\gamma^{-} \in \Lambda_{-}$ such that
\begin{equation*}
\forall j \in \{ 1 , \ldots , n \} \qquad \Pi_{\lambda_{j}} ( \gamma^{-}_{\lambda_{j},0} ) = \widetilde{\gamma}^{-}_{\lambda_{j},0} .
\end{equation*}
\end{proposition}

\begin{proof}
Let
\begin{equation*}
\rho (t) = \sum_{\mu_{k} \leq N} \sum_{m=0}^{M_{\mu_{k}}} \gamma_{\mu_{k} ,m}^{-} t^{m} e^{- \mu_{k} t} ,
\end{equation*}
where the $\gamma_{\mu_{k} ,m}^{-}$ are given by Lemma \ref{d6} and $N$ will be fixed ulteriorly. Since \eqref{d8} is verified for all $\mu_{k} \leq N$, we have
\begin{equation*}
\partial_{t} \rho (t) = H_{p} ( \rho (t)) + R (t),
\end{equation*}
with $R (t) = \CO ( e^{- (N + \varepsilon) t} )$. We seek a solution of the form $\gamma^{-} (t) = \rho (t) + r (t)$. Then, $r$ must satisfies
\begin{equation}
\partial_{t} r = H_{p} ( \rho +r) - H_{p} (\rho) - R.
\end{equation}
Let $T_{N} \geq 0$ be such that $\vert R (t) \vert \leq e^{- N t}$ and $\vert \rho (t) \vert \leq 1$ for all $t \geq T_{N}$. We define inductively $r_{j} (t)$ by
\begin{equation} \label{d13}
\left\{ \begin{aligned}
&r_{0} (t) = 0 \\
&r_{j+1} (t) = - \int_ {t}^{+ \infty} \big( H_{p} ( \rho +r_{j} ) - H_{p} (\rho) - R \big) (s) \, d s .
\end{aligned} \right.
\end{equation}

\begin{lemma}\sl \label{d11}
For $N$ large enough, the functions $( r_{j} )_{j \geq 0}$ exist on $[ T_{N} , + \infty [$ and
\begin{equation} \label{d12}
\vert r_{j} (t) \vert \leq e^{- N t} ,
\end{equation}
for all $t \geq T_{N}$.
\end{lemma}

\begin{proof}[Proof of Lemma \ref{d11}]
Define
\begin{equation}
C_{1} = \sup_{\vert u \vert \leq 2} \vert d H_{p} ( u ) \vert .
\end{equation}
We will prove the lemma inductively. First, $r_{0}$ satisfies \eqref{d12}. Assume now that $r_{j-1}$ exists on $[ T_{N} , + \infty [$ and verifies \eqref{d12}. In particular, $\vert r_{j-1} (t) \vert \leq 1$ for $t \geq T_{N}$. Then, \eqref{d13} gives
\begin{align}
\vert r_{j} (t) \vert &\leq \int_{t}^{+ \infty} \big( \big\vert H_{p} ( \rho +r_{j-1} ) - H_{p} (\rho) \big\vert + \vert R \vert \big) d s  \nonumber \\
&\leq \int_{t}^{+ \infty} \big( C_{1} \vert r_{j-1} \vert + \vert R \vert \big) d s \leq \int_{t}^{+ \infty} (C_{1} +1) e^{- N s} d s  \nonumber  \\
&\leq \frac{C_{1} +1}{N} e^{-N t} ,
\end{align}
for $t \geq T_{N}$. Therefore, if $N \geq C_{1}+1$, $r_{j}$ satisfies \eqref{d12} and the lemma follows.
\end{proof}

\begin{lemma}\sl \label{d14}
For $N$ large enough, we have
\begin{equation} \label{d15}
\big\vert r_{j+1} (t) - r_{j} (t) \big\vert \leq \frac{e^{-N t}}{2^{j}} ,
\end{equation}
for $j \geq 0$ and $t \geq T_{N}$.
\end{lemma}

\begin{proof}[Proof of Lemma \ref{d14}]
For $j=0$ and $N$ large enough, Lemma \ref{d11} gives $\vert r_{1} (t) - r_{0} (t) \vert = \vert r_{1} (t) \vert \leq e^{-N t}$. Assume that \eqref{d15} holds for some $j -1 \geq 0$. Using \eqref{d13}, we get
\begin{align}
\vert r_{j+1} (t) - r_{j} (t) \vert &\leq \int_{t}^{+ \infty} \big\vert H_{p} ( \rho +r_{j} ) - H_{p} (\rho + r_{j-1}) \big\vert \, d s \nonumber \\
&\leq C_{1} \int_{t}^{+ \infty} \big\vert r_{j} - r_{j-1} \big\vert \, d s \leq C_{1} \int_{t}^{+ \infty} \frac{e^{-N s}}{2^{j-1}} \, d s  \nonumber \\
&\leq \frac{2 C_{1}}{N} \frac{e^{-N t}}{2^{j}} .
\end{align}
Then, for $N \geq 2 C_{1}$, \eqref{d15} holds and the lemma follows.
\end{proof}

\begin{lemma}\sl \label{d16}
For $N$ large enough, there exists $r \in C^{\infty} ( [ T_{N} , + \infty [)$ such that

$i)$ for $t \geq T_{N}$, we have $\vert r (t) \vert \leq e^{-N t}$,

$ii)$ for all $j \geq 0$,
\begin{equation*}
\big\Vert e^{N t} (r_{j} - r) \big\Vert_{L^{\infty} ( [T_{N} , + \infty [)} \leq 2^{1-j} ,
\end{equation*}

$iii)$ the curve $\gamma^{-} = \rho + r$ satisfies $\partial_{t} \gamma^{-} = H_{p} ( \gamma^{-} )$.
\end{lemma}

\begin{proof}[Proof of Lemma \ref{d16}]
Using standard arguments, Lemma \ref{d14} provides us with a function $r \in C^{0} ( [ T_{N} , + \infty [)$ satisfying $ii)$. Then, part $i)$ follows directly from Lemma \ref{d11}. On the other hand,
\begin{align*}
\Big\vert \int_{t}^{+ \infty} \big( H_{p} ( \rho + r_{j} ) & - H_{p} (\rho) - R \big) - \big( H_{p} ( \rho + r ) - H_{p} (\rho) - R \big) d s \Big\vert   \\
&\leq \int_{t}^{+ \infty} \big\vert H_{p} (\rho + r_{j} ) - H_{p} ( \rho + r ) \big\vert d s \leq \int_{t}^{+ \infty} C_{1} \vert r_{j} - r \vert d s  \\
&\leq C_{1} 2^{1-j} \int_{t}^{+ \infty} e^{- N s} d s \leq \frac{C_{1}}{N} 2^{1-j} \longrightarrow 0,
\end{align*}
as $j \rightarrow + \infty$. Then, taking the limit $j \rightarrow + \infty$ in \eqref{d13}, we obtain
\begin{equation*}
r (t) = - \int_ {t}^{+ \infty} \big( H_{p} ( \rho +r ) - H_{p} (\rho) - R \big) (s) \, d s .
\end{equation*}
Thus, $r \in C^{\infty} ( [ T_{N} , + \infty [)$ and $\gamma^{-} = \rho + r$ satisfies $\partial_{t} \gamma^{-} = H_{p} ( \gamma^{-} )$.
\end{proof}

To finish the proof of Proposition \ref{d10}, we impose in addition that $\lambda_{n} < N$. Then, the function $\gamma^{-}$ of Lemma \ref{d16} is a Hamiltonian curve in $\Lambda_{-}$ and, since $r (t) = o (e^{- \lambda_{n} t})$, Lemma~\ref{d6} assures that $\Pi_{\lambda_{j}} ( \gamma^{-}_{\lambda_{j},0} ) = \widetilde{\gamma}^{-}_{\lambda_{j},0} $ for all $j \in \{ 1 , \ldots , n \}$.
\end{proof}

\bibliographystyle{amsplain}

\begin{thebibliography}{10}

\bibitem{AlBoRa08_01}
I.~Alexandrova, J.-F. Bony, and T.~Ramond, \emph{Semiclassical scattering
  amplitude at the maximum of the potential}, Asymptot. Anal. \textbf{58}
  (2008), no.~1-2, 57--125.

\bibitem{BoFuRaZe07_01}
J.-F. Bony, S.~Fujii\'e, T.~Ramond, and M.~Zerzeri, \emph{Microlocal kernel of
  pseudodifferential operators at a hyperbolic fixed point}, J. Funct. Anal.
  \textbf{252} (2007), no.~1, 68--125.

\bibitem{BoMi04_01}
J.-F. Bony and L.~Michel, \emph{Microlocalization of resonant states and
  estimates of the residue of the scattering amplitude}, Comm. Math. Phys.
  \textbf{246} (2004), no.~2, 375--402.

\bibitem{BrCoDu87_01}
P.~Briet, J.-M. Combes, and P.~Duclos, \emph{On the location of resonances for
  {S}chr\"odinger operators in the semiclassical limit. {I}. {R}esonances free
  domains}, J. Math. Anal. Appl. \textbf{126} (1987), no.~1, 90--99.

\bibitem{BrCoDu87_02}
P.~Briet, J.-M. Combes, and P.~Duclos, \emph{On the location of resonances for {S}chr{\"o}dinger operators in
  the semiclassical limit {II}: Barrier top resonances}, Comm. in Partial
  Differential Equations \textbf{2} (1987), no.~12, 201--222.

\bibitem{BuZw01_01}
N.~Burq and M.~Zworski, \emph{Resonance expansions in semi-classical
  propagation}, Comm. Math. Phys. \textbf{223} (2001), no.~1, 1--12.

\bibitem{BuZw04_01}
N.~Burq and M.~Zworski, \emph{Geometric control in the presence of a black box}, J. Amer.
  Math. Soc. \textbf{17} (2004), no.~2, 443--471.

\bibitem{ChZw00_01}
T.~Christiansen and M.~Zworski, \emph{Resonance wave expansions: two hyperbolic
  examples}, Comm. Math. Phys. \textbf{212} (2000), no.~2, 323--336.

\bibitem{DeRo03_01}
S.~De~Bi{\`e}vre and D.~Robert, \emph{Semiclassical propagation on {$\vert \log
  \hslash\vert $} time scales}, Int. Math. Res. Not. (2003), no.~12, 667--696.

\bibitem{DeGe97_01}
J.~Derezi{\'n}ski and C.~G{\'e}rard, \emph{Scattering theory of classical and
  quantum {$N$}-particle systems}, Texts and Monographs in Physics,
  Springer-Verlag, Berlin, 1997.

\bibitem{DiSj99_01}
M.~Dimassi and J.~Sj{\"o}strand, \emph{Spectral asymptotics in the
  semi-classical limit}, London Mathematical Society Lecture Note Series, vol.
  268, Cambridge University Press, Cambridge, 1999.

\bibitem{Ge88_01}
C.~G{\'e}rard, \emph{Asymptotique des p\^oles de la matrice de scattering pour
  deux obstacles strictement convexes}, M\'em. Soc. Math. France (1988),
  no.~31, 146.

\bibitem{GeMa89_02}
C.~G{\'e}rard and A.~Martinez, \emph{Prolongement m\'eromorphe de la matrice de
  scattering pour des probl\`emes \`a deux corps \`a longue port\'ee}, Ann.
  Inst. H. Poincar\'e Phys. Th\'eor. \textbf{51} (1989), no.~1, 81--110.

\bibitem{GeSi92_01}
C.~G{\'e}rard and I.~Sigal, \emph{Space-time picture of semiclassical
  resonances}, Comm. Math. Phys. \textbf{145} (1992), no.~2, 281--328.

\bibitem{GeSj87_01}
C.~G{\'e}rard and J.~Sj{\"o}strand, \emph{Semiclassical resonances generated by
  a closed trajectory of hyperbolic type}, Comm. Math. Phys. \textbf{108}
  (1987), 391--421.

\bibitem{GuNa09_01}
C.~Guillarmou and F.~Naud, \emph{Wave decay on convex co-compact hyperbolic
  manifolds}, Comm. Math. Phys. \textbf{287} (2009), no.~2, 489--511.

\bibitem{HaMeVa08_01}
A.~Hassell, R.~Melrose, and A.~Vasy, \emph{Microlocal propagation near radial
  points and scattering for symbolic potentials of order zero}, Anal. PDE
  \textbf{1} (2008), no.~2, 127--196.

\bibitem{HeMa87_01}
B.~Helffer and A.~Martinez, \emph{Comparaison entre les diverses notions de
  r\'esonances}, Helv. Phys. Acta \textbf{60} (1987), no.~8, 992--1003.

\bibitem{HeSj85_01}
B.~Helffer and J.~Sj{\"o}strand, \emph{Multiple wells in the semiclassical
  limit. {III}. {I}nteraction through nonresonant wells}, Math. Nachr.
  \textbf{124} (1985), 263--313.

\bibitem{HeSj86_01}
B.~Helffer and J.~Sj{\"o}strand, \emph{R\'esonances en limite semi-classique}, M\'em. Soc. Math. France
  (1986), no.~24-25, iv+228.

\bibitem{HiSjVu07_01}
M.~Hitrik, J.~Sj{\"o}strand, and S.~V{\~u}~Ng{\d{o}}c, \emph{Diophantine tori
  and spectral asymptotics for nonselfadjoint operators}, Amer. J. Math.
  \textbf{129} (2007), no.~1, 105--182.

\bibitem{Hu86_01}
W.~Hunziker, \emph{Distortion analyticity and molecular resonance curves}, Ann.
  Inst. H. Poincar\'e Phys. Th\'eor. \textbf{45} (1986), no.~4, 339--358.

\bibitem{IsKi85_01}
H.~Isozaki and H.~Kitada, \emph{Modified wave operators with time-independent
  modifiers}, J. Fac. Sci. Univ. Tokyo Sect. IA Math. \textbf{32} (1985),
  no.~1, 77--104.

\bibitem{IsKi86_01}
H.~Isozaki and H.~Kitada, \emph{Scattering matrices for two-body {S}chr\"odinger operators},
  Sci. Papers College Arts Sci. Univ. Tokyo \textbf{35} (1986), no.~2, 81--107.

\bibitem{KaKe00_01}
N.~Kaidi and P.~Kerdelhu{\'e}, \emph{Forme normale de {B}irkhoff et
  r\'esonances}, Asymptot. Anal. \textbf{23} (2000), no.~1, 1--21.

\bibitem{Be99_01}
A.~Lahmar-Benbernou, \emph{Estimation des r\'esidus de la matrice de diffusion
  associ\'es \`a des r\'esonances de forme. {I}}, Ann. Inst. H. Poincar\'e
  Phys. Th\'eor. \textbf{71} (1999), no.~3, 303--338.

\bibitem{LaMa99_01}
A.~Lahmar-Benbernou and A.~Martinez, \emph{Semiclassical asymptotics of the
  residues of the scattering matrix for shape resonances}, Asymptot. Anal.
  \textbf{20} (1999), no.~1, 13--38.

\bibitem{LaPh89_01}
P.~Lax and R.~Phillips, \emph{Scattering theory}, second ed., Pure and Applied
  Mathematics, vol.~26, Academic Press Inc., Boston, MA, 1989, With appendices
  by C. Morawetz and G. Schmidt.

\bibitem{Ma02_01}
A.~Martinez, \emph{Resonance free domains for non globally analytic
  potentials}, Ann. Henri Poincar\'e \textbf{3} (2002), no.~4, 739--756.

\bibitem{Mi03_01}
L.~Michel, \emph{Semi-classical estimate of the residues of the scattering
  amplitude for long-range potentials}, J. Phys. A \textbf{36} (2003), no.~15,
  4375--4393.

\bibitem{Na89_01}
S.~Nakamura, \emph{Scattering theory for the shape resonance model. {I}.
  {N}onresonant energies}, Ann. Inst. H. Poincar\'e Phys. Th\'eor. \textbf{50}
  (1989), no.~2, 115--131.

\bibitem{Na89_02}
S.~Nakamura, \emph{Scattering theory for the shape resonance model. {II}.
  {R}esonance scattering}, Ann. Inst. H. Poincar\'e Phys. Th\'eor. \textbf{50}
  (1989), no.~2, 133--142.

\bibitem{NaStZw03_01}
S.~Nakamura, P.~Stefanov, and M.~Zworski, \emph{Resonance expansions of
  propagators in the presence of potential barriers}, J. Funct. Anal.
  \textbf{205} (2003), no.~1, 180--205.

\bibitem{Ra96_01}
T.~Ramond, \emph{Semiclassical study of quantum scattering on the line}, Comm.
  Math. Phys. \textbf{177} (1996), no.~1, 221--254.

\bibitem{Sj87_01}
J.~Sj{\"o}strand, \emph{Semiclassical resonances generated by nondegenerate
  critical points}, Pseudodifferential operators (Oberwolfach, 1986), Lecture
  Notes in Math., vol. 1256, Springer, Berlin, 1987, pp.~402--429.

\bibitem{Sj97_01}
J.~Sj{\"o}strand, \emph{A trace formula and review of some estimates for resonances},
  Microlocal analysis and spectral theory ({L}ucca, 1996), NATO Adv. Sci. Inst.
  Ser. C Math. Phys. Sci., vol. 490, Kluwer Acad. Publ., Dordrecht, 1997,
  pp.~377--437.

\bibitem{SjZw91_01}
J.~Sj{\"o}strand and M.~Zworski, \emph{Complex scaling and the distribution of
  scattering poles}, J. Amer. Math. Soc. \textbf{4} (1991), no.~4, 729--769.

\bibitem{St02_01}
P.~Stefanov, \emph{Estimates on the residue of the scattering amplitude},
  Asymptot. Anal. \textbf{32} (2002), no.~3-4, 317--333.

\bibitem{TaZw98_01}
S.-H. Tang and M.~Zworski, \emph{From quasimodes to reasonances}, Math. Res.
  Lett. \textbf{5} (1998), no.~3, 261--272.

\bibitem{TaZw00_01}
S.-H. Tang and M.~Zworski, \emph{Resonance expansions of scattered waves}, Comm. Pure Appl. Math.
  \textbf{53} (2000), no.~10, 1305--1334.

\bibitem{Va89_01}
B.~Va{\u\i}nberg, \emph{Asymptotic methods in equations of mathematical
  physics}, Gordon \& Breach Science Publishers, New York, 1989, Translated
  from the Russian by E. Primrose.
\end{thebibliography}
\providecommand{\bysame}{\leavevmode\hbox to3em{\hrulefill}\thinspace}
\providecommand{\MR}{\relax\ifhmode\unskip\space\fi MR }
\providecommand{\MRhref}[2]{%
  \href{http://www.ams.org/mathscinet-getitem?mr=#1}{#2}
}
\providecommand{\href}[2]{#2}

\end{document}